\documentclass[12pt,onecolumn,journal]{IEEEtran}

\ifCLASSINFOpdf

\else

\fi
\usepackage[normalem]{ulem}
\usepackage{cite}
\usepackage{amsmath}
\usepackage{amssymb}
\usepackage{amsfonts}
\usepackage{amsthm}
\usepackage{mathtools}
\usepackage{graphicx}
\usepackage{color}
\usepackage{epstopdf}
\usepackage{caption}
\usepackage{subfigure}
\usepackage[numbers,sort&compress]{natbib}
\usepackage[left=2cm, right=2cm, top=2.54cm]{geometry}
\DeclareMathOperator{\Var}{\operatorname{Var}}

\newcommand{\norm}[1]{\left\lVert#1\right\rVert}
\newcommand{\stkout}[1]{\ifmmode\text{\sout{\ensuremath{#1}}}\else\sout{#1}\fi}

\usepackage{xcolor,cancel}

\newtheorem{thm}{Theorem}

\newtheorem{lem}{Lemma}
\newtheorem{ex}{Example}
\newtheorem{assum}{Assumption}
\newtheorem{prop}{Proposition}

\begin{document}
	
	\title{Minimax Rate Optimal Adaptive Nearest Neighbor Classification and Regression}
	
	\author{Puning Zhao and
		Lifeng Lai\thanks{Puning Zhao and Lifeng Lai are with Department of Electrical and Computer Engineering, University of California, Davis, CA, 95616. Email: \{pnzhao,lflai\}@ucdavis.edu. This work was supported by the National Science Foundation under grants CCF-17-17943, ECCS-17-11468, CNS-18-24553 and CCF-19-08258. This paper was presented in part at IEEE International Symposium on Information Theory, Paris, France, 2019 \cite{knnisit}.}
	}
	\maketitle
	
	\begin{abstract}
$k$ Nearest Neighbor (kNN) method is a simple and popular statistical method for classification and regression. For both classification and regression problems, existing works have shown that, if the distribution of the feature vector has bounded support and the probability density function is bounded away from zero in its support, the convergence rate of the standard kNN method, in which $k$ is the same for all test samples, is minimax optimal. On the contrary, if the distribution has unbounded support, we show that there is a gap between the convergence rate achieved by the standard kNN method and the minimax bound. To close this gap, we propose an adaptive kNN method, in which different $k$ is selected for different samples. Our selection rule does not require precise knowledge of the underlying distribution of features. The new proposed method significantly outperforms the standard one. We characterize the convergence rate of the proposed adaptive method, and show that it matches the minimax lower bound.
	\end{abstract}
\begin{IEEEkeywords}
	KNN, classification, regression
\end{IEEEkeywords}

	\IEEEpeerreviewmaketitle
	
	\section{Introduction}
In both classification and regression problems, our goal is to infer a function that predicts the target $Y$ from a feature vector $\mathbf{X}\in \mathbb{R}^d$. For classification problems, $Y$ takes values from a discrete set, while for regression problems, $Y$ takes values from a continuous set. Associated with both classification and regression problems, there exists a loss function $L(\hat{Y},Y)$, which measures the accuracy of the prediction $\hat{Y}$, and a risk function $R$, which is defined as the expectation of the loss function. If the joint distribution of $\mathbf{X}$ and $Y$ is known, we can decide an optimal prediction rule $\hat{Y}=g^*(\mathbf{X})$, which achieves the minimum risk named Bayes risk. However, in practice, such a joint distribution is unknown and the determination of the prediction rule is instead based on a finite number of identical and independently distributed (i.i.d) training samples $(\mathbf{X}_1,Y_1),\ldots, (\mathbf{X}_N,Y_N)$. As a result, it is inevitable for a classification or regression method to have some excess risk, which is defined as the gap between its risk and the Bayes risk. The problem of checking whether the excess risk of a classification or regression method converges to zero as the number of training samples increases, and characterizing the corresponding convergence rate, has attracted significant research interests \cite{biau2012analysis,blanchard2003rate,marron1983optimal,yang1999minimax,zhang2018rates}.

Among all classification and regression methods, $k$ Nearest Neighbor (kNN) method is a simple and popular one. It is easy to implement, and does not rely on any specific models, i.e., it can automatically adapt to supervised learning problems with arbitrary Bayes decision boundaries. For classification problems, given any test point $\mathbf{X}$, the kNN classifier assigns it with label $\hat{Y}$ determined by the majority vote from the labels of $k$ nearest neighbors of $\mathbf{X}$ among the training set. For regression problem, the mean of the observed labels of $k$ nearest neighbors of $\mathbf{X}$ is assigned to be the predicted label. The performances of kNN classification and regression have been extensively investigated. The kNN method with a fixed $k$, which depends neither on the sample size nor the position of each sample, was analyzed in \cite{cover1967nearest}. It was shown that under some weak assumptions, the risk will converge to a limit value that is higher than the Bayes risk. For larger fixed $k$, such a limit value is lower. However, as long as $k$ is fixed, the limit value can not reach the Bayes risk in general \cite{biau2015lectures}. An intuitive explanation is that, if $k$ is fixed, then the variance of predicted value does not converge to zero. Therefore, it is necessary to let $k$ grows with $N$. At the same time, the ratio $k/N$ needs to converge to zero, so that the bias can also be controlled simultaneously \cite{devroye1994strong,stone1977consistent}. Between these two extremes, there exists an optimal growth rate of $k$ over $N$, under which the best convergence rate is attained. 

Important progresses towards identifying the best growth rate of $k$ and finding the corresponding optimal convergence rate have been made in \cite{audibert2007fast,doring2017rate,chaudhuri2014rates,gyorfi1981rate,kpotufe2011k}. In particular, \cite{gyorfi1981rate} analyzes the convergence rate of local risk at a specific query point $\mathbf{X}=\mathbf{x}$, i.e., $\mathbb{E}[L(\hat{Y},Y)|\mathbf{X}=\mathbf{x}]$. Note that the total risk $R=\mathbb{E}[L(\hat{Y},Y)]$ is the expectation of the local risk, hence it can be shown that if the support of the distribution of $\mathbf{X}$ is bounded, and the underlying probability density function (pdf) $f(\mathbf{x})$ is bounded away from zero in its support, then the convergence rate of the local risk is of the same order as the convergence rate of the total risk \cite{biau2010rates,doring2017rate,biau2015lectures}. In this case, for both classification and regression problems, kNN method can achieve the best convergence rate in the minimax sense \cite{audibert2007fast,chaudhuri2014rates,tsybakov2009introduction,knnisit}.

However, for many common cases, the support of distribution of $\mathbf{X}$ is not bounded, or the pdf is not bounded away from zero. In these cases, we can no longer ensure that the convergence rate of the total risk is still of the same order as that of the local risk. For example, \cite{gadat2016classification} derives a bound of the convergence rate of the standard kNN classifier, which is slower than the minimax lower bound. This phenomenon can be explained by the fact that kNN distances tend to be large in the regions where the pdf of the features is low. Hence, the conditional distribution of the target at the test point can be quite different from those at its nearest neighbors. As a result, the inference using kNN method becomes less accurate. To improve the accuracy of kNN classification and regression, it is necessary to use adaptive $k$, which means that different values of $k$ are used for different samples. In particular, a ``sliced nearest neighbor'' method was proposed in \cite{gadat2016classification}. This method divides the support into several regions depending on the pdf of $\mathbf{X}$, and uses different $k$ in different regions. It was proved in \cite{gadat2016classification} that this new method attains the minimax convergence rate. However, this method requires us to know the underlying pdf. Although the pdf can be estimated from the training samples, the theoretical guarantee of the adaptive classifier is not established if we take the estimation error into consideration. If apart from a set of labeled training samples, we also have abundant unlabeled samples, which are many more than labeled samples, then it is possible to estimate the pdf with sufficient accuracy. We can then use the estimated density values to determine the optimal $k$. In this case, the problem is actually a semi-supervised learning problem, which was discussed in \cite{cannings2017local}. However, in supervised learning problems, those unlabeled data is not available. 

In this paper, we focus on both classification and regression problems with neither precise knowledge of the feature distribution nor any unlabeled data. We propose an adaptive kNN method that works for both classification and regression problems. We prove that the proposed adaptive kNN method is minimax rate optimal for a wide range of distributions for both classification and regression. Furthermore, unlike in~\cite{knnisit}, we show that the optimal choice of a key parameter depends only on the dimension of the feature. Hence, the proposed adaptive kNN method does not involve too much parameter tuning. 
In particular, we make the following contributions. 

Firstly, we derive a bound for the convergence rate of the standard kNN classification and regression, which uses the same $k$ for all test samples, under assumptions that are more general than those discussed in existing studies such as \cite{chaudhuri2014rates} and \cite{gadat2016classification}. In particular, we introduce parameters to describe the properties of the distribution of the feature vector, such as tail parameters, as well as parameters to describe the distribution of labels, such as margin parameters. Our bound depends on these parameters and is applicable to a broader class of distributions than those derived in the existing literatures. The derived bound recovers the bounds in the existing studies \cite{audibert2004classification,chaudhuri2014rates,doring2017rate,kohler2007rate}, although some assumptions are slightly different. Furthermore, we provide a lower bound for the excess risk of the standard kNN method over a set of distributions. We show that the lower bound and the upper bound almost match, therefore our bounds are tight and can not be further improved.

Secondly, we derive a minimax lower bound over all classification and regression methods that do not have the information of the underlying regression function, under the same assumptions as mentioned above. The result indicates that, if the distribution has tails, then there exists a gap between the convergence rate of the excess risk of the standard kNN method and the minimax convergence rate. Hence, the standard kNN classification and regression are not optimal under these scenarios.

Thirdly, to close the gap identified above, we propose and analyze a new adaptive kNN method, in which we use different value of $k$ for different test points. Our approach is based solely on labeled training samples, without requiring the precise knowledge of pdf $f(\mathbf{x})$. In particular, for a given test sample, we select $k$ as an increasing function of the number of training samples that fall in a fixed radius ball centered at this test sample. The purpose of our choice of $k$ is to achieve a desirable bias and variance tradeoff. This is motivated by the observation that if $k$ increases, the kNN distances will increase and hence the bias will also increase. On the other hand, the variance decreases with $k$. As a result, if there are many training samples around the test point $\mathbf{x}$, then we can safely use a larger $k$ to reduce the variance, without worrying too much about the bias, since the kNN distances are still not large when the training samples are dense around the test point. On the contrary, if there are fewer training samples around the test point, then we need to use a smaller $k$ to control the bias. Building on this intuition, we carefully design a selection rule of $k$, which is an increasing function of the number of samples in the fixed radius neighborhood of each testing point. Intuitively, the number of samples in the fixed radius neighborhood can be viewed as an estimation of density around the test point. However, here we do not expect to have a consistent density estimation, since the bias of this density estimate does not decay with the sample size $N$. Nevertheless, our method can still bridge the gap mentioned above, and achieve the minimax optimal convergence rate, despite that the density estimation using a fixed radius nearest neighbor search is not consistent. Furthermore, as will be clear in the sequel, our proposed method does not need too much parameter tuning. To the best of our knowledge, our method is the first nonparametric classification and regression method that is proven to be rate optimal for feature distributions with both bounded and unbounded support. 

The remainder of this paper is organized as follows. In Section~\ref{sec:statement}, we present the precise statement of the classification and regression problem and our proposed adaptive kNN method. The theoretical analyses for classification and regression problems are presented in Section~\ref{sec:cla} and \ref{sec:reg}, respectively. In Section~\ref{sec:num}, we conduct numerical experiments to compare the performance of our new proposed adaptive kNN with that of the standard one, for both classification and regression problems. Finally, we offer concluding remarks in Section~\ref{sec:conc}. 

\section{Problem Formulation and Proposed Method}\label{sec:statement}

For classification problems, we let the feature vector $\mathbf{X}$ and target $Y$ take values in $\mathbb{R}^d$ and $\{-1,1\}$, respectively. $(\mathbf{X},Y)$ follows an unknown joint distribution. Denote $f(\mathbf{x})$ as the pdf of $\mathbf{X}$. We use 0-1 loss function
\begin{eqnarray}
L(\hat{Y},Y)=\left\{
\begin{array}{ccc}
0 &\text{if} & \hat{Y}=Y\\
1 &\text{if} &\hat{Y}\neq Y
\end{array}
\right..
\end{eqnarray}
With this loss function, the risk of a classifier $\hat{Y}=g(\mathbf{X})$ is 
\begin{eqnarray}
R(g)=\mathbb{E}[L(Y,\hat{Y})]=\text{P}(g(\mathbf{X})\neq Y).
\label{eq:risk}
\end{eqnarray}
Define function $\eta$ as
\begin{eqnarray}
\eta(\mathbf{x}):=\mathbb{E}[Y|\mathbf{X}=\mathbf{x}]=\text{P}(Y=1|\mathbf{X}=\mathbf{x})-\text{P}(Y=-1|\mathbf{X}=\mathbf{x}).
\end{eqnarray}
It can be shown that the Bayes optimal classification rule is given by \cite{doring2017rate}:
\begin{eqnarray}
g^*(\mathbf{x})=\text{sign}(\eta(\mathbf{x})),
\end{eqnarray}
and the corresponding risk, called Bayes risk, is
\begin{eqnarray}
R^*=\text{P}(g^*(\mathbf{X})\neq Y)=\mathbb{E}\left[\frac{1-|\eta(\mathbf{X})|}{2}\right].
\label{eq:bayesrisk}
\end{eqnarray}
From \eqref{eq:risk} and \eqref{eq:bayesrisk}, it can be shown that the excess risk $R-R^*$ takes the following form:
\begin{eqnarray}
R-R^*=\mathbb{E}[\mathbf{1}(g(\mathbf{X})\neq g^*(\mathbf{X}))|\eta(\mathbf{X})|],
\end{eqnarray}
in which $\mathbf{1}(\cdot)$ is the indicator function.

For regression problems, the target $Y$ can take value in $\mathbb{R}$. In this paper, we assume that $Y$ has the following relationship with $\mathbf{X}$:
\begin{eqnarray}
Y=\eta(\mathbf{X})+\epsilon,
\end{eqnarray}
in which $\eta$ is the true underlying regression function and $\epsilon$ denotes the noise that satisfies $\mathbb{E}[\epsilon|\mathbf{X}=\mathbf{x}]=0$ for all $\mathbf{x}$. We use $\ell_2$ loss to evaluate the regression accuracy: $L(\hat{Y},Y)=(\hat{Y}-Y)^2$. With this loss function, the risk of regression function $\hat{Y}=g(\mathbf{X})$ is
\begin{eqnarray}
R=\mathbb{E}[(g(\mathbf{X})-Y)^2].
\label{eq:risk-reg}
\end{eqnarray}
Under $\ell_2$ loss, the Bayes optimal regression rule is given by $g^*(\mathbf{x})=\eta(\mathbf{x})$, and the corresponding Bayes risk is $R^*=\mathbb{E}[(\eta(\mathbf{X})-Y)^2]=\mathbb{E}[\epsilon^2]$. Then the excess risk can be expressed as
\begin{eqnarray}
R-R^*=\mathbb{E}[(g(\mathbf{X})-\eta(\mathbf{X}))^2].
\end{eqnarray}

In practice, for both classification and regression problems, $f(\mathbf{x})$ and $\eta(\mathbf{x})$ are unknown. Instead, the prediction rule is based on $N$ i.i.d samples $(\mathbf{X}_i,Y_i)$, $i=1,\ldots,N$, which are all drawn from the joint distribution of $\mathbf{X}$ and $Y$. Since for any classification or regression method, $R\geq R^*$ always holds, we evaluate their performance using the excess risk $R-R^*$. In particular, we characterize the convergence rate, i.e. the rate at which the excess risk goes to zero. 

\subsection{The standard kNN rules}

The standard kNN classification rule has the following form:
\begin{eqnarray}
g(\mathbf{x})=\text{sign}\left(\hat{\eta}(\mathbf{x})\right),
\label{eq:stdcls}
\end{eqnarray}
in which 
\begin{eqnarray}
	\hat{\eta}(\mathbf{x})=\frac{1}{k}\sum_{i=1}^k Y^{(i)}\label{eq:etahat}
\end{eqnarray}
with $Y^{(i)}$ being the target value corresponding to $\mathbf{X}^{(i)}$ and $\mathbf{X}^{(i)}$ being the $i$-th nearest neighbor of $\mathbf{x}$. The distance of $\mathbf{X}_i$ and $\mathbf{X}_j$ is $\norm{\mathbf{X}_i-\mathbf{X}_j}$, in which $\norm{\cdot}$ can be any norm. In the standard kNN classification, $k$ is the same for all samples. 

The standard kNN regression rule is
\begin{eqnarray}
g(\mathbf{x})=\frac{1}{k}\sum_{i=1}^k Y^{(i)},
\label{eq:stdreg}
\end{eqnarray}
with $Y^{(i)}$ defined similarly as the target value of the $i$-th nearest neighbor of $\mathbf{x}$. Again, here $k$ is the same for all samples. 

\subsection{Proposed adaptive kNN method}\label{sec:proposed}

Our proposed adaptive kNN classification and regression methods has the same form as \eqref{eq:stdcls} and \eqref{eq:stdreg}. However, instead of using the same $k$ for all testing samples, we use a sample dependent $k$. In particular, for a given query point $\mathbf{x}$, let $B(\mathbf{x},A)$ be a ball centered at $\mathbf{x}$ with a fixed radius $A$, in which the norm used for this radius is the same as the norm for kNN distances. We select $k$ as:
\begin{eqnarray}
k=\left\lfloor Kn^q\right\rfloor+1,
\label{eq:kdef}
\end{eqnarray}
in which $0<q<1$, and
\begin{eqnarray}
n=\sum_{i=1}^N \mathbf{1}(\mathbf{X}_i\in B(\mathbf{x},A))
\label{eq:ndef}
\end{eqnarray}
is the number of training samples falling in $B(\mathbf{x},A)$. $K,A,q$ are three design parameters. In Sections \ref{sec:cla} and \ref{sec:reg}, we will show that for both classification and regression problems, the parameters $K$ and $A$ do not impact the convergence rates of the excess risk, as long as $K$ and $A$ are fixed with respect to sample size $N$. $q$ will impact the convergence rate. We will show that the optimal $q$, under which the best convergence rate is achieved, is $4/(d+4)$.  This optimal choice of $q$ depends only on the dimension $d$ of the random variable $\mathbf{X}$. Hence, the proposed method does not involve too much parameter tuning.

Our design is motivated by the observation that in the regions where $f(\mathbf{x})$ is small, the kNN distances are large, thus the values of the underlying regression function $\eta(\mathbf{x})$ can be quite different at these $k$ points. As a result, the inference of $\eta(\mathbf{x})$ from these $k$ neighbors may not be accurate. To solve this problem, we use smaller $k$ at the tail of distribution. \cite{gadat2016classification} uses similar ideas, but the method to choose $k$ in \cite{gadat2016classification} needs the exact value of $f(\mathbf{x})$. In particular, the scheme in \cite{gadat2016classification} divides the support of distribution into several regions based on the value of pdf. Each region corresponds to a different choice of $k$, which is then used to predict the target value of a test point, if it falls on this region. Nevertheless, in practice, $f(\mathbf{x})$ is unknown. In our algorithm, we use \eqref{eq:ndef} as a proxy to measure $f(\mathbf{x})$, and use \eqref{eq:kdef} to adaptively set the value of $k$. It is easy to see from \eqref{eq:kdef} and \eqref{eq:ndef} that $n$ (and hence $k$) tends to be smaller in regions with smaller density, and vise versa. The purpose of adding $1$ to $\lfloor Kn^q\rfloor$ in \eqref{eq:kdef} is to ensure that $k$ is at least $1$. Our method shares some similarity with \cite{cannings2017local}, which uses the result of kernel density estimate to determine $k$. However, \cite{cannings2017local} requires a sufficiently large number of unlabeled data to ensure that the estimated density function is sufficiently close to the real density function, so that the adaptive kNN algorithm converges as fast as the case in which $f(\mathbf{x})$ is known and the selection of $k$ is based on the real $f(\mathbf{x})$. On the contrary, our method does not require unlabeled data, and we do not hope to have an accurate estimation of the density. In fact, since the radius is fixed, the bias of density estimation using \eqref{eq:ndef} will not converge to zero as sample size $N$ increases. Nevertheless, despite that the density estimation is not consistent, we can still show that our method is minimax rate optimal. 
\section{Classification}\label{sec:cla}
In this section, we focus on classification problems. We begin with the analysis of the convergence rate of the excess risk of the standard kNN classification. We then derive a minimax lower bound. Finally, we characterize the convergence rate of our adaptive method and show that our new method is minimax optimal. 

The analysis of the classification risk is based on the following assumptions:
\begin{assum}\label{ass:basic}
	There exist finite constants $C_a, C_b, C_c$ and $\alpha>0$, $\beta>0$, such that:
	
	(a) For all $t>0$,
	\begin{eqnarray}
	\text{P}(0<|\eta(\mathbf{X})|\leq t)\leq C_at^\alpha;
	\end{eqnarray}
	
	(b) For all $t>0$,
	\begin{eqnarray}
	\text{P}(f(\mathbf{X})\leq t)\leq C_b t^\beta;
	\end{eqnarray}
	
	(c) For an arbitrary $r>0$ and any $\mathbf{x}$ in the support of $f(\mathbf{x})$,
	\begin{eqnarray}
	|\eta(B(\mathbf{x},r))-\eta(\mathbf{x})|\leq C_c r^2,
	\label{eq:smooth}
	\end{eqnarray}
	in which $\eta(B(\mathbf{x},r)):=\mathbb{E}[Y|\mathbf{X}\in B(\mathbf{x},r)]$;
	
	(d) $\exists D>0$ such that
	\begin{eqnarray}
	\text{P}(B(\mathbf{x},r))\geq C_df(\mathbf{x}) V(B(\mathbf{x},r))
	\end{eqnarray}
	for all $\mathbf{x}$ and $0<r<D$, in which $B(\mathbf{x},r)$ is a ball centered at $\mathbf{x}$ with radius $r$, $V(B(\mathbf{x},r))$ is the volume of $B(\mathbf{x},r)$, and $\text{P}(B(\mathbf{x},r))$ is the probability mass of $B(\mathbf{x},r)$.
	
\end{assum}

The assumptions here share some similarities with previous work \cite{chaudhuri2014rates,gadat2016classification}. In particular, Assumption 1 (a) is called margin assumption, which controls the size of the region near the Bayes decision boundary. This assumption is reasonable because misclassification is easier to occur at the position where $\text{P}(Y=1|\mathbf{x})$ and $\text{P}(Y=-1|\mathbf{x})$ are close. The same assumption was used in \cite{chaudhuri2014rates,doring2017rate,gadat2016classification}. Assumption 1 (b) controls the tail of the distribution. If the distribution of feature vector $\mathbf{X}$ has unbounded support, then the maximum $\beta$ such that Assumption 1 (b) holds for constant $C_b$ is at most 1. On the contrary, if the support is bounded, then the maximum $\beta$ is at least 1. Furthermore, if the density is bounded away from zero, then Assumption 1 (b) holds for arbitrarily large $\beta$. Assumption 1 (c) describes the smoothness of the regression function $\eta(\mathbf{x})$. A traditional quantity that evaluates the smoothness of functions is the H{\"o}lder parameter. As discussed in \cite{gadat2016classification} (Remark 2.1), for the standard kNN algorithm, it is not suitable to assume that the smoothness index is greater than $1$. Here, we use \eqref{eq:smooth} to replace the H{\"o}lder condition, so that it is possible to impose an assumption that is approximately the same as requiring the second-order smoothness of $\eta$. 
Assumption 1 (d) is the minimum probability mass assumption, which was already used in existing works \cite{gadat2016classification,cannings2017local}.

The following proposition provides sufficient conditions for Assumption 1 (b) and (c).
\begin{prop}\label{prop:smoothness}
	(A) If the $\tau$-th moment of $\mathbf{X}$ is bounded, i.e., $\mathbb{E}[||\mathbf{X}||^\tau]<\infty$, then for any $\beta<\tau/(d+\tau)$, there exists a constant $C_b$ such that Assumption \ref{ass:basic} (b) holds.
	
	(B) If Assumption \ref{ass:basic} (d) holds, $\eta(\mathbf{x})$ has bounded Hessian, i.e., there exists a constant $C_H$, such that $||\nabla^2 \eta(\mathbf{x})||_{op}\leq C_H$, in which $||\cdot||_{op}$ denotes the operator norm, and there exists a constant $D'$, such that
	\begin{eqnarray}
	\underset{\mathbf{u}\in B(\mathbf{x},D')}{\sup} \frac{\norm{\nabla \eta(\mathbf{x})}_2\norm{\nabla f(\mathbf{u})}_2}{f(\mathbf{x})}\leq C_0,
	\label{eq:condition2}
	\end{eqnarray} 
	in which $C_0$ is a constant, then Assumption \ref{ass:basic} (c) holds.
\end{prop}
For the proof of Proposition \ref{prop:smoothness} (A), please refer to Appendix F in \cite{zhao2018analysis}. For (B), the proof is shown in Appendix \ref{sec:smoothness}.

The condition in Proposition \ref{prop:smoothness} (A) shows that our tail assumption is weaker than assuming the boundedness of moments of feature vector $\mathbf{X}$. Roughly speaking, the condition in Proposition \ref{prop:smoothness} (B) requires that the gradient of density function decays with the same rate as the density itself. Similar assumption was already used in \cite{berrett2019efficient} and \cite{cannings2017local}. For example, if $\mathbf{X}$ follows Laplace distribution and $\eta$ is sinusoidal, then Assumption \ref{ass:basic} (c) is satisfied. 


\subsection{Convergence rate of the standard kNN classifier}
Now under Assumption \ref{ass:basic}, we provide a bound of the convergence rate of the standard kNN classifiers, which select the same value of $k$ for every $\mathbf{x}$. In the following analysis, the kNN distance is based on metric $d(\mathbf{x}_1,\mathbf{x}_2)=||\mathbf{x}_2-\mathbf{x}_1||$, in which $||\cdot||$ is an arbitrary norm. The convergence rate depends on the growth rate of $k$ over sample size $N$. In the following theorem, we show the best convergence rate when such a growth rate is optimally selected.

\begin{thm}\label{thm:standard}
	Under Assumption \ref{ass:basic} (a)-(d), if $k$ is optimally selected, then the convergence rate of excess risk is
	\begin{eqnarray}
	R-R^*=\left\{
	\begin{array}{ccc}
	\mathcal{O}\left(N^{-\min\left\{ \frac{\beta(\alpha+1)}{2\beta+\alpha+1},\frac{2\beta(\alpha+1)}{\beta d+2(\alpha+2\beta)} \right\}}\right) &\text{if} & \beta\neq \frac{2}{d};\\
	\mathcal{O}\left(N^{-\frac{\beta(\alpha+1)}{2\beta+\alpha+1}}\ln N\right) &\text{if} &\beta=\frac{2}{d}.
	\end{array}
	\right.
	\label{eq:std-ub}
	\end{eqnarray}
	The above rate is attained if
	\begin{eqnarray}
	k\sim \left\{
	\begin{array}{ccc}
	N^{\frac{2\beta}{2\beta+\alpha+1}} &\text{if} &\beta\leq \frac{2}{d};\\
	N^\frac{4\beta}{2\alpha+\beta(d+4)} &\text{if} &\beta>\frac{2}{d}.
	\end{array}
	\right.
	\label{eq:kstd}
	\end{eqnarray}
	Moreover, this bound is almost tight. In particular, denote $\mathcal{S}$ as the set of all pairs $(f,\eta)$ such that Assumption \ref{ass:basic} (a)-(d) hold with sufficiently large $C_a$, $C_b$, $C_c$, then for the standard kNN classification, 
	\begin{eqnarray}
\underset{k}{\inf}	\underset{(f,\eta)\in \mathcal{S}}{\sup} (R-R^*)=\Omega\left(N^{-\min\left\{ \frac{\beta(\alpha+1)}{2\beta+\alpha+1},\frac{2\beta(\alpha+1)}{\beta d+2(\alpha+2\beta)} \right\}}\right).
\label{eq:std-lb}
	\end{eqnarray}
\end{thm}
\color{black}
\begin{proof}
	(Outline) For the proof of our upper bound, let $\delta$ and $\Delta$ be two parameters to be determined, we divide the support into four regions and analyze each region separately.
	
	\begin{itemize}
		\item $S_1=\{\mathbf{x}|f(\mathbf{x})\geq N^{-\delta}, |\eta(\mathbf{x})|>2\Delta \}$. In this region, the pdf is larger than threshold $N^{-\delta}$, and the underlying regression function at $\mathbf{x}$ is at least $(2\Delta)$-far away from zero. For any test point $\mathbf{x}$ in this region, the label prediction of the standard kNN classifier is different from the prediction of Bayes classifier only if the estimated regression function $\hat{\eta}(\mathbf{x})$ has different sign with the real regression function $\eta(\mathbf{x})$, which happens with a decreasing probability as the sample size $N$ increases. The excess risk can then be bounded by giving a bound of this probability. 
		\item $S_2=\{\mathbf{x}|f(\mathbf{x})\geq N^{-\delta}, |\eta(\mathbf{x})|\leq 2\Delta \}$. In this region, the pdf is larger than $N^{-\delta}$, but the underlying regression function is close to zero. Therefore, $\text{P}(Y=1|\mathbf{x})$ is close to $\text{P}(Y=-1|\mathbf{x})$, which indicates that the inherent randomness is large. Therefore, the risk of both the kNN classifier and the Bayes optimal classifier are large in this region. The conditional excess risk in $S_2$ can be bounded by $2\Delta$.  
		\item $S_3=\left\{ \mathbf{x}|C_0 k/N<f(\mathbf{x})<N^{-\delta} \right\}$ for some constant $C_0$. In this region, the pdf is relatively small, and the probability that $\hat{\eta}(\mathbf{x})$ and $\eta(\mathbf{x})$ have opposite sign becomes larger, thus the technique for the analysis of $S_1$ is no longer effective. However, $f(\mathbf{x})>C_0k/N$ ensures that with high probability, all of the $k$ nearest neighbors are not too far away from test point $\mathbf{x}$. We can then use the estimation error to bound the excess risk of classification in this region.
		\item $S_4=\left\{\mathbf{x}|f(\mathbf{x})\leq C_0k/N\right\}$. In this region, the pdf is too small and classification can be pretty inaccurate. Hence, we bound the excess risk simply with the probability of a sample falling in this region.
	\end{itemize}

For the proof of our lower bound, we construct three types of distributions. For each type, we can find a lower bound of $R-R^*$, in terms of $k$ and $N$. The first type of distribution is just a uniform distribution, and the first lower bound indicates the impact of variance. The second type of distribution involves $n$ cubes with relatively low density and one cube with relatively high density. We adjust $n$ and the density in these cubes, so that the estimation of $\eta(\mathbf{x})$ in the first $n$ cubes with low density is sufficiently inaccurate, therefore we can get another lower bound proportional to the total probability mass of these $n$ cubes. This bound indicates the effect of tail. The third type of distribution also uses $(n+1)$ cubes, similar to the second type. However, the cube size becomes adaptive, and thus can generate a new bound. These three bounds are then combined together. It turns out that if $k$ is larger, than the first bound becomes lower, but the second and third bound becomes higher, and vice versa. We then find the infimum of the maximum of these three bounds by adjusting $k$. 

Detailed proofs of the upper and lower bounds are shown in Appendix \ref{sec:original-ub} and Appendix \ref{sec:original-lb}, respectively.
\end{proof}
\color{black}
Now we compare our result with that of the existing works. If the distribution has a density that is bounded below by a positive constant, then for arbitrarily large $\beta$, there exists a constant $C_b$ so that Assumption \ref{ass:basic} (b) holds. This assumption corresponds to the strong density assumption in \cite{audibert2007fast}. In this case, we have
\begin{eqnarray}
R-R^*=\mathcal{O}\left(N^{-\frac{2(\alpha+1)}{d+4}+\epsilon}\right),
\label{eq:SDAbound}
\end{eqnarray}
for arbitrarily small $\epsilon>0$. \eqref{eq:SDAbound} agrees with the result of \cite{chaudhuri2014rates,doring2017rate,kohler2007rate}. For distributions with tails, our convergence rate is faster than the result in Theorem 4.3 in \cite{gadat2016classification}, because our Assumption \ref{ass:basic} (c) requires two order smoothness.

From this theorem, we observe that, to achieve the best convergence rate for the standard kNN, the selection of $k$ depends on parameters $\alpha$ and $\beta$, which may not be available in practice. On the contrary, the proposed adaptive kNN method presented in Section~\ref{sec:proposed} does not need this information. Furthermore, we will show in Section~\ref{sec:rate} that the proposed method achieves a better convergence rate.
\subsection{Minimax convergence rate}
We now derive the minimax convergence rate of all classifiers (including those classifiers that do not use kNN distances) under Assumption \ref{ass:basic}. Denote $\mathcal{S}$ as the collection of $(f,\eta)$ that satisfy Assumption \ref{ass:basic}, $g$ as the possible classifier. We have the following minimax convergence rate that holds for all classifiers that do not have the knowledge of the underlying regression function $\eta(\mathbf{x})$.
\begin{thm}\label{thm:minimax}
	If
	\begin{eqnarray}
	\beta(2\alpha-d)\leq 2\alpha,
	\label{eq:cond}
	\end{eqnarray}
	then
	\begin{eqnarray}
	\underset{g}{\inf}\underset{(f,\eta)\in \mathcal{S}}{\sup} (R-R^*) =\Omega\left(N^{-\min\left\{\beta,\frac{2\beta (\alpha+1)}{\beta d+2(\alpha+2\beta)} \right\}}\right).
	\label{eq:minimax}
	\end{eqnarray}
\end{thm}
\begin{proof}
	(Outline) A common approach to obtain the minimax bound is to find a subset of $\mathcal{S}$, and then convert the problem into a hypothesis testing problem using Assouad lemma \cite{audibert2004classification}. We refer to \cite{audibert2004classification} and \cite{tsybakov2009introduction} for a detailed introduction of this type of method.
	
	In our proof, we carefully select a subset $\mathcal{S}^*\subset \mathcal{S}$. In particular, $\mathcal{S}^*$ contains a number of pairs $(f,\eta_\mathbf{v})$, in which $\mathbf{v}$ is a vector with each component taking binary values, such that the marginal distributions of features are the same among $\mathcal{S}^*$, but the underlying regression functions are different depending on $\mathbf{v}$. Then the problem of finding the minimax lower bound of classification can be converted into the problem of finding the minimum error probability of hypothesis testing. Since $\mathcal{S}^*\subset \mathcal{S}$, the minimax convergence rate among $\mathcal{S}^*$ can also be used as a lower bound of the minimax rate among $\mathcal{S}$. The detailed proof can be found in Appendix \ref{sec:minimax}.
\end{proof}

We now discuss the additional condition \eqref{eq:cond} and the result \eqref{eq:minimax}. Note that if the distribution has unbounded support, then as discussed above, the maximum $\beta$ such that there exists a constant $C_b$ so that Assumption \ref{ass:basic} (b) holds is no more than $1$. As a result, regardless of the dimension $d$ and the margin parameter $\alpha$, \eqref{eq:cond} always holds. Our result generalizes and improves some previous results \cite{audibert2007fast,gadat2016classification}. Under the strong density assumption, i.e., the support is bounded and the density is bounded away from zero, our result on the minimax convergence rate is also consistent with Theorem 3.5 in \cite{audibert2007fast}. If $\beta=1$, then our minimax convergence rate is consistent with Theorem 4.1 in \cite{audibert2007fast}, and faster than the result in Theorem 4.2 in \cite{gadat2016classification}, since Assumption \ref{ass:basic} (c) requires two-order smoothness of function $\eta$.

\subsection{Convergence rate of the proposed adaptive kNN classification}\label{sec:rate}
As we can observe from Theorem \ref{thm:standard} and Theorem \ref{thm:minimax}, there exists a gap between the convergence rates of the standard kNN classifier and the minimax lower bound. In this section, we show that this gap can be closed using the new adaptive kNN method presented in Section~\ref{sec:proposed}. To obtain the convergence rate of this adaptive kNN classifier, we need the following additional assumption.
\begin{assum}\label{ass:additional}
	For any $t>0$,
	\begin{eqnarray}
	\text{P}\left(\frac{f(\mathbf{X})}{\text{P}^q(B(\mathbf{X},A))}<t^{1-q}\right)\leq C_b' t^\beta,
	\label{eq:additional}
	\end{eqnarray}
	for some constant $C_b'$.
\end{assum}

The following theorem provides a bound of the convergence rate of the excess risk of the proposed adaptive kNN classifier.
\begin{thm}\label{thm:adaptive}
	Let
	\begin{eqnarray}
	\lambda=\min\left\{\frac{1}{2}q, \frac{2}{d}(1-q) \right\}.
	\label{eq:lamdef}
	\end{eqnarray}
	Under Assumption \ref{ass:basic} and Assumption \ref{ass:additional}, the convergence rate of the excess risk of the adaptive kNN classifier is bounded by
	\begin{eqnarray}
	R-R^*=\left\{
	\begin{array}{ccc}
	\mathcal{O}\left(N^{-\min\left\{\frac{\lambda\beta(\alpha+1)}{\lambda \alpha+\beta}, \beta \right\}}\right) &\text{if} &\beta\neq\lambda\\
	\mathcal{O}(N^{-\beta}\ln N) &\text{if} & \beta=\lambda
	\end{array}	
	\right..
	\label{eq:adaptive}
	\end{eqnarray}
	As a result, the optimal $q$ is $q^*=4/(d+4)$, and the corresponding optimal convergence rate is
	\begin{eqnarray}
	R-R^*=\left\{
	\begin{array}{ccc}
	\mathcal{O}\left(N^{-\min\left\{ \frac{2\beta(\alpha+1)}{\beta d+2(\alpha+2\beta)},\beta \right\}}\right) &\text{if} & \beta\neq \frac{2}{d+4}\\
	\mathcal{O}(N^{-\beta}\ln N)&\text{if}&\beta=\frac{2}{d+4}
	\end{array}
	\right..
	\label{eq:ada-opt-cl}
	\end{eqnarray}
\end{thm}
\begin{proof}
	(Outline) Similar to the proof of Theorem \ref{thm:standard}, we divide the support set into four regions and derive bound for each region separately. In particular, we divide the support into four regions $S_1, \ldots, S_4$. $S_1$ and $S_2$ are defined similar to $S_1$ and $S_2$ for the standard kNN classifier. In $S_1$, we obtain lower and upper bounds of $k$, which hold with high probability. Then we bound the probability that $\hat{\eta}(\mathbf{x})$ has different sign with the real regression function $\eta(\mathbf{x})$ using the derived upper and lowers bound of $k$. For $S_2$, we use similar bounds derived in the analysis of the standard kNN classifier.
	
	We further divide the tail region into $S_3$ and $S_4$. Here, $S_3$ is selected to ensure that $k\leq n$ with a high probability, hence the kNN distance will not exceed $A$. As discussed in Section \ref{sec:statement}, the main reason for the performance improvement of our new adaptive classifier, as compared to the standard one, is that we use adaptive $k$, so that the estimation of the underlying regression function at the tail of the feature distribution becomes more accurate. More specifically, we can obtain a better convergence rate in $S_3$. It was shown in \cite{mammen1999smooth} that the excess classification error probability can be bounded by the estimation error, hence we can find the bound of the estimation error first, and then use this bound to obtain a bound for the excess classification error probability. In $S_4$, which denotes the region on which the pdf of feature is very small, we can no longer ensure that $k\leq n$ with a high probability, hence the kNN distance can be larger than $A$. As a result, the estimation of the regression function in this region can be quite inaccurate. In this case, we bound the excess risk simply with the probability of a test sample falling in this region. The detailed proof is shown in the Appendix \ref{sec:adaptive}.
\end{proof}

Theorem \ref{thm:adaptive} shows that the convergence rate of the proposed adaptive classifier depends on the choice of $q$. However, the optimal $q$ is $q^*=4/(d+4)$, which depends only on dimension $d$. It does not depend on the sample size $N$ or any unknown parameters of the underlying distribution. $K$ and $A$ do not affect the convergence rate. Therefore our new adaptive method requires less parameter tuning than the standard one. We also observe that \eqref{eq:ada-opt-cl} matches the minimax lower bound provided in Theorem \ref{thm:minimax}, except the special case $\beta=2/(d+4)$, in which an additional $\ln N$ factor is needed.

Now we comment on Assumption \ref{ass:additional}. Given Assumption \ref{ass:basic}, a sufficient condition for Assumption \ref{ass:additional} to hold is that there exists a constant $C_d'$ such that for any $\mathbf{x}$ in the support of the distribution of the feature vector, we have
\begin{eqnarray}
\text{P}(B(\mathbf{x},A))\leq C_c'f(\mathbf{x})V(B(\mathbf{x},A)).
\label{eq:dcomplement}
\end{eqnarray}

This condition is a complement of Assumption \ref{ass:basic} (d). With \eqref{eq:dcomplement} and Assumption \ref{ass:basic}, we know that the average density $\text{P}(B(\mathbf{x},A))/V(B(\mathbf{x},A))$ is upper and lower bounded by $f(\mathbf{x})$ times constant factors $C_d'$ and $C_d$, respectively. Therefore, it can be proved that Assumption \ref{ass:additional} holds. For many common distributions, such as uniform, exponential and Cauchy distributions, \eqref{eq:dcomplement} holds. Therefore, with \eqref{eq:dcomplement} and Assumption \ref{ass:basic}, Assumption \ref{ass:additional} is also satisfied, hence the result in Theorem \ref{thm:adaptive} holds. For some other distributions, \eqref{eq:dcomplement} is not satisfied. However, the convergence rate can still be approximately minimax. Here we show such an example.

\begin{ex}
	Gaussian distributions do not satisfy \eqref{eq:dcomplement}. However, Gaussian distributions satisfy \eqref{eq:additional} for any $0<\beta<1$. Based on this fact, we can bound the convergence rate of the adaptive kNN classifier with $k$ selected by \eqref{eq:kdef}. Assume $\alpha=1$. Furthermore, if Assumption \ref{ass:basic} is satisfied, then Assumption \ref{ass:additional} holds. In this case, according to \eqref{eq:adaptive}, 
	\begin{eqnarray}
	R-R^*=\mathcal{O}\left(N^{-\frac{4\beta}{2+\beta(d+4)}+\delta}\right), \forall 0<\beta<1, \delta>0,
	\end{eqnarray}
	which is equivalent to the following result:
	\begin{eqnarray}
	R-R^*=\mathcal{O}\left(N^{-\frac{4}{d+6}+\delta}\right),
	\end{eqnarray}
	for arbitrarily small $\delta>0$.
\end{ex}
\section{Regression}\label{sec:reg}
In this section, we extend the study to kNN regression. For kNN regression, we replace Assumption \ref{ass:basic} (a) with the conditional variance assumption.
\begin{assum}\label{ass:basicreg}
	Assume that Assumptions \ref{ass:basic} (b), (c), (d) hold, and (a) is replaced by
	\begin{eqnarray}
	\Var[Y|\mathbf{X}=\mathbf{x}]\leq C_a, \forall \mathbf{x}.
	\label{eq:condvar}
	\end{eqnarray}
	
\end{assum}
\eqref{eq:condvar} means that the noise variance is bounded. We will analyze the convergence rate of kNN nonparametric regression for the case where $\eta(\mathbf{x})$ is bounded and unbounded separately.

\subsection{Bounded $\eta(\mathbf{x})$}
We first analyze the convergence rate for the case where $\eta(\mathbf{x})$ is bounded. We specify this additional assumption as following:
\begin{assum}\label{ass:bounded}
	There exists a constant $M$, such that for all $\mathbf{x}$, $|\eta(\mathbf{x})|\leq M$.
\end{assum}
Under this assumption, the following theorem gives a bound of the convergence rate of the standard kNN regression when $k$ is optimally selected.
\begin{thm}\label{thm:standardreg}
	Under Assumptions \ref{ass:basicreg} and \ref{ass:bounded}, the optimal growth rate of $k$ is
	\begin{eqnarray}
	k\sim \left\{
	\begin{array}{ccc}
	N^{\frac{4}{d+4}} &\text{if} & \beta>\frac{4}{d}\\
	N^{\frac{\beta}{\beta+1}} &\text{if} & \beta\leq \frac{4}{d}
	\end{array}
	\right..
	\label{eq:kdef-reg}
	\end{eqnarray}
	If $k$ is selected according to \eqref{eq:kdef-reg}, then the convergence rate of the standard kNN regression method is
	\begin{eqnarray}
	R-R^*=\left\{
	\begin{array}{ccc}
	\mathcal{O}\left(N^{-\frac{4}{d+4}}\right) &\text{if} & \beta>\frac{4}{d}\\
	\mathcal{O}\left(N^{-\frac{\beta}{\beta+1}}\ln N\right) &\text{if} & \beta= \frac{4}{d}\\
	\mathcal{O}\left(N^{-\frac{\beta}{\beta+1}}\right) &\text{if} & \beta< \frac{4}{d}\\
	\end{array}
	\right..
	\label{eq:standardreg}
	\end{eqnarray}
	
	Moreover, the above convergence rate is almost tight. In particular, denote $\mathcal{S}$ as the set of all pairs $(f,\eta)$ such that Assumptions \ref{ass:basicreg} and \ref{ass:bounded} hold with sufficiently large $C_a, C_b, C_c, M$, then for the standard kNN regression method,
	\begin{eqnarray}
	\underset{k}{\inf}\underset{(f,\eta)\in \mathcal{S}}{\sup} (R-R^*)=\left\{
	\begin{array}{ccc}
	\Omega\left(N^{-\frac{4}{d+4}}\right) &\text{if} & \beta>\frac{4}{d}\\
	\Omega\left(N^{-\frac{\beta}{\beta+1}}\right) &\text{if} & \beta\leq \frac{4}{d}.
	\end{array}
	\right.
	\end{eqnarray}
	
\end{thm}
Now we compare our result with that of existing results. If the feature distribution has bounded support and the density is bounded away from zero, then our result is consistent with previous results, including \cite{gyorfi1981rate,biau2010rates} and Section 3 of \cite{doring2017rate}. If the density is not bounded away from zero, then the convergence rate depends on the tail parameter $\beta$. For many common distributions, we have $\beta\leq 4/d$, hence the convergence rate of the mean square error of the standard kNN regression is slow. For example, if the feature follows exponential distribution and the regression function $\eta$ has bounded Hessian, then we have $\beta=1$, and Assumptions \ref{ass:basic} (b), (c) and (d) hold. In this case, the convergence rate if $\mathcal{O}(N^{-1/2})$. Similar to the kNN classification, an intuitive explanation of the reason why the standard kNN regression converges slowly is that the kNN distances are large at the tail region.

The following theorem shows a minimax lower bound of nonparametric regression.
\begin{thm}\label{thm:minimaxreg}
	Denote $\mathcal{S}$ as the set of all $(f,\eta)$ such that Assumption \ref{ass:basicreg} is satisfied with $f$ and $\eta$, then
	\begin{eqnarray}
	\underset{g}{\inf} \underset{(f,\eta)\in \mathcal{S}}{\sup} (R-R^*)=\Omega\left(N^{-\min\left\{\frac{4}{d+4} ,\beta\right\}}\right).
	\label{eq:minimaxreg}
	\end{eqnarray}
\end{thm}
Similar to the standard kNN classification, for regression problems, from \eqref{eq:standardreg} and \eqref{eq:minimaxreg}, we observe that for many common feature distributions with tails, the standard kNN converges slower than the minimax lower bound. 

We now show the convergence rate of our new adaptive kNN regression. Similar to the kNN classification problem, our analysis for regression also requires Assumption \ref{ass:additional}.

\begin{thm}\label{thm:adaptivereg}
	Define $\lambda$ as \eqref{eq:lamdef}, then with fixed $K,A$ and $q$, under Assumptions \ref{ass:additional}, \ref{ass:basicreg}, \ref{ass:bounded}, the convergence rate of the adaptive kNN regression is bounded by
	\begin{eqnarray}
	R-R^*=\left\{
	\begin{array}{ccc}
	\mathcal{O}\left(N^{-\min\{\beta,2\lambda\}}\right)&\text{if} & \beta\neq 2\lambda\\
	\mathcal{O}\left(N^{-\beta}\ln N\right) &\text{if} & \beta=2\lambda
	\end{array}
	\right..
	\label{eq:adaptivereg}
	\end{eqnarray}
	
	As a result, the optimal $q$ is $q^*=4/(d+4)$, the corresponding $\lambda$ is $2/(d+4)$. Thus except the special case $\beta=4/(d+4)$, the optimal convergence rate is
	\begin{eqnarray}
	R-R^*=\mathcal{O}\left(N^{-\min\left\{\beta,\frac{4}{d+4} \right\}}\right).
	\label{eq:ada-opt}
	\end{eqnarray}
\end{thm}
The above result shows that the convergence rate of our new method in \eqref{eq:ada-opt} is an improvement over the standard kNN regression method, for distributions with $\beta<4/d$. This bound also matches the lower bound provided in Theorem \ref{thm:minimaxreg}, showing that our new method is minimax rate optimal. For example, for one dimensional exponential feature distribution and regression function $\eta$ with bounded Hessian, the convergence rate is $\mathcal{O}\left(N^{-0.8}\right)$, which is an improvement over $\mathcal{O}\left(N^{-1/2}\right)$ achieved by the standard kNN. This implies that the accuracy of our adaptive method can significantly outperform that of the standard kNN regression, especially for distributions with heavier tails. In addition, similar to the classification problem, the selection of parameters for the adaptive classifier also becomes more convenient than that of the standard one, because $q^*$ only depends on $d$, while the optimal $k$ for the standard kNN regression depends on unknown $\beta$. 

\subsection{Unbounded $\eta(\mathbf{x})$}

Now we generalize the above analysis to the case where $\eta(\mathbf{x})$ is not necessarily bounded. In this case, for the test samples whose kNN distances among the training samples are large, the accuracy of estimation of $\eta(\mathbf{x})$ deteriorates more seriously, as large kNN distances occuring at the tail of the distribution here can cause a more serious effect. In this case, we need to change some assumptions. For example, the second order moment of the feature distribution must be bounded, otherwise there is no universally consistent regression method. Then under the new assumption, we derive bounds of the convergence rates of the standard kNN and our adaptive kNN method. The analysis shows that our proposed adaptive method can still outperform the standard one.

We formulate all the assumptions required for the analysis of the cases with unbounded $\eta$ as follows.
\begin{assum}\label{ass:unbounded}
	Suppose that \eqref{eq:condvar} and Assumptions \ref{ass:basic} (c), (d) hold. In addition,
	
	(b') $\mathbb{E}[\norm{\mathbf{X}}^2]\leq M_X$ for some constant $M_X$ and  $\int (1+ \norm{\mathbf{x}}^2) e^{-bf(\mathbf{x})}f(\mathbf{x})d\mathbf{x}\leq C_b' b^{-\beta'} $ for all $b\geq 0$;
	
	(e) For any $\mathbf{x}_1$ and $\mathbf{x}_2$ with $\norm{\mathbf{x}_2-\mathbf{x}_1}\geq D$, in which $D$ is the constant in Assumption \ref{ass:basic} (d), there exists a constant $L$ such that $|\eta(\mathbf{x}_2)-\eta(\mathbf{x}_1)|\leq L\norm{\mathbf{x}_2-\mathbf{x}_1}$.
\end{assum}
Assumption \ref{ass:unbounded} (b') is a modification of Assumption \ref{ass:basic} (b). We now compare these two assumptions. It can be proved that if Assumption \ref{ass:unbounded} (b') holds for some $\beta'$, then Assumption \ref{ass:basic} (b) holds for $\beta=\beta'$, but the converse is not true. For many distributions with heavy tails, the maximum $\beta'$ such that Assumption \ref{ass:unbounded} (b') holds is smaller than the maximum $\beta$ that Assumption \ref{ass:basic} (b) holds.

The following theorem shows that without the new tail Assumption \ref{ass:unbounded} (b'), we can not find a regressor that is uniformly consistent, which implies that the new tail Assumption \ref{ass:unbounded} (b') is necessary.
\begin{thm}\label{thm:nounif}
	Under \eqref{eq:condvar}, Assumptions \ref{ass:basic} (c), (d), and Assumption \ref{ass:unbounded} (e), if Assumption \ref{ass:unbounded} (b') does not hold, then no regressor is uniformly consistent, i.e. there exists a $\delta>0$, such that
	\begin{eqnarray}
	\underset{N\rightarrow \infty}{\lim\sup} \underset{(f,\eta)\in \mathcal{S}}{\sup} (R-R^*)\geq \delta,
	\end{eqnarray}
	in which $\mathcal{S}$ denotes the set of all $(f,\eta)$ that satisfy the assumptions mentioned above.
\end{thm}

With Assumption \ref{ass:unbounded}, our bounds of the convergence rates of the standard kNN and the adaptive kNN regression are shown in Theorem \ref{thm:stdunbounded} and Theorem \ref{thm:adaunbounded}, respectively.

\begin{thm}\label{thm:stdunbounded}
	Under Assumption \ref{ass:unbounded}, the optimal growth rate of $k$ is
	\begin{eqnarray}
	k\sim \left\{
	\begin{array}{ccc}
	N^{\frac{4}{d+4}} &\text{if} & \beta'>\frac{4}{d},\\
	N^{\frac{\beta'}{\beta'+1}} &\text{if} &\beta'\leq \frac{4}{d}.
	\end{array}
	\right.
	\end{eqnarray}
	If $k$ is selected in this way, then the convergence rate of the standard kNN regression, without requiring the boundedness $\eta(\mathbf{x})$, is bounded by:
	\begin{eqnarray}
	R-R^*=\left\{
	\begin{array}{ccc}
	\mathcal{O}\left(N^{-\frac{4}{d+4}}\right) &\text{if} & \beta'>\frac{4}{d},\\
	\mathcal{O}\left(N^{-\frac{\beta'}{\beta'+1}}\ln N\right) &\text{if} &\beta'= \frac{4}{d},\\
	\mathcal{O}\left(N^{-\frac{\beta'}{\beta'+1}}\right) &\text{if} &\beta'< \frac{4}{d}.
	\end{array}
	\right.
	\end{eqnarray}
\end{thm}
\begin{thm}\label{thm:adaunbounded}
	Under Assumptions \ref{ass:additional} and \ref{ass:unbounded}, the convergence rate of the adaptive kNN regressor is bounded by:
	\begin{eqnarray}
	R-R^*=\left\{
	\begin{array}{ccc}
	\mathcal{O}\left(N^{-\min\{\beta',2\lambda\}}\right)&\text{if} & \beta'\neq 2\lambda\\
	\mathcal{O}\left(N^{-\beta'}\ln N\right) &\text{if} & \beta'=2\lambda
	\end{array}
	\right.,
	\label{eq:unbounded}
	\end{eqnarray}
	in which $\lambda$ is defined in \eqref{eq:lamdef}. The optimal $q$ is $q^*=4/(d+4)$. The corresponding convergence rate is
	\begin{eqnarray}
	R-R^*=\left\{
	\begin{array}{ccc}
	\mathcal{O}\left(N^{-\min\left\{\beta',\frac{4}{d+4} \right\}}\right) &\text{if} &\beta'\neq 4/(d+4)\\
	\mathcal{O}\left(N^{-\beta'}\ln N\right)&\text{if} &\beta'=4/(d+4)
	\end{array}
	\right..
	\end{eqnarray}
\end{thm}

We observe that if the feature distribution has a bounded support, or is sub-Gaussian or sub-exponential, then the convergence rate does not suffer seriously from the unbounded regression function $\eta(\mathbf{x})$, since it can be shown that in this case, Assumption \ref{ass:unbounded} (b') holds with any $\beta'<1$, and we can just let $\beta'$ to be sufficiently close to $1$, and therefore \eqref{eq:unbounded} becomes $R-R^*=\mathcal{O}(N^{-4/(d+4)})$, if we let $q=4/(d+4)$. This rate is the same as the convergence rate we derived for the case with bounded $\eta(\mathbf{x})$. Such observation can be explained by the fact that the training samples are not too far away from each other. For example, for sub-exponential distributions, we have $\mathbb{E}\left[\underset{i,j}{\max} \norm{\mathbf{X}_i-\mathbf{X}_j}\right]=\mathcal{O}(\ln N)$, which implies that with Assumption \ref{ass:unbounded} (e), the difference between the values of $\eta$ at all samples can not exceed $\mathcal{O}(\ln N)$ on average. In this case, the performance of both standard and adaptive kNN regression are similar to the case with bounded $\eta$, except that there may be an additional $\ln N$ factor. However, for distributions with heavy tails, the maximum $\beta'$ such that Assumption \ref{ass:unbounded} (b') holds is smaller than the maximum $\beta$ such that Assumption \ref{ass:basic} (b) holds. Hence the convergence rate with unbounded regression function can be substantially slower than the case with a bounded real regression function. This phenomenon can be explained by the fact that the distances between samples can be large, which can cause serious effect when we estimate $\eta$ based on the nearest neighbors.

\section{Numerical Examples}\label{sec:num}
In this section, we provide numerical experiments to illustrate the analytical results derived in this paper. In these experiments, we compare the empirical performance of our adaptive kNN classification and regression methods with that of the standard one. 

To make the comparison between the proposed adaptive kNN and the standard kNN as fair as possible, we set the parameter in the following way. For the proposed adaptive kNN, we fix $A=1$ in all of the numerical experiments, and then find best $K$ to minimize the empirical risk at $N=500$ by conducting a series of numerical simulations with different $K$. Similarly, at $N=500$, we also find best $k$ for the standard kNN method. After $K$ in the proposed method and $k$ in the standard kNN are both optimally tuned, we compare the performance for different sample sizes. For our new adaptive method, we use the same value of $A$ and $K$ as discussed above to determine $k$ in~\eqref{eq:kdef}. For $q$ in~\eqref{eq:kdef}, we use $q=4/(d+4)$. For the standard kNN method, we let $k$ grow with $N$, and the growth rate is specified in the Theorems \ref{thm:standard}, \ref{thm:standardreg} and \ref{thm:stdunbounded}. 

We show the simulation results for classification and regression separately.
\subsection{Classification}
The results of simulations on one and two dimensional feature distributions are shown in Fig. \ref{fig:compare} and \ref{fig:compare2}, respectively.

\begin{figure}[h]
	\begin{center}
		\subfigure[\small 1d Laplace distribution, $\eta(x)=\cos(5x)$.]{\includegraphics[width=0.45\linewidth,height=0.34\linewidth]{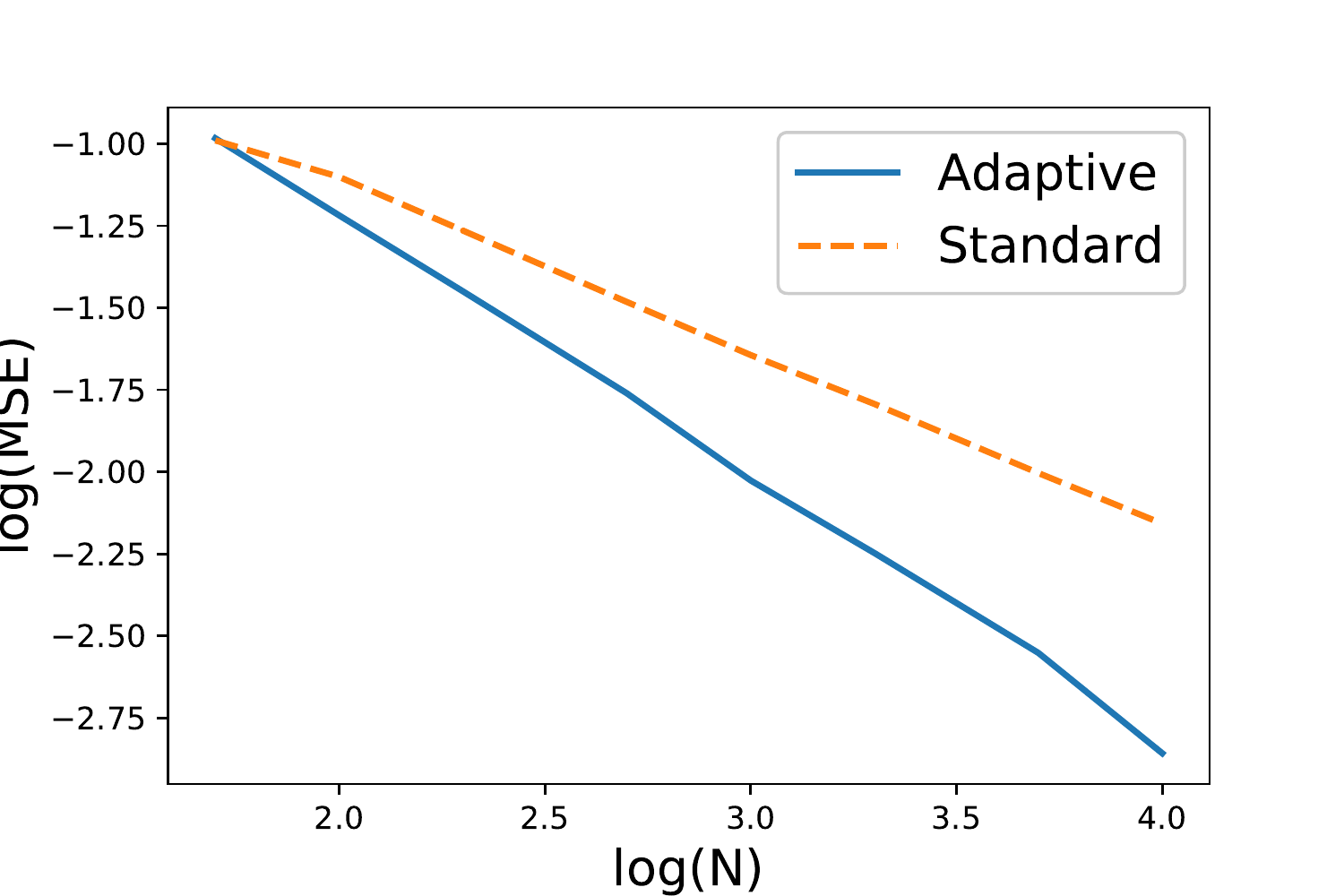}}
		\subfigure[\small $t_5$ distribution, with regression function \hspace{1cm} $\eta(x)=\cos(5x)$.]{\includegraphics[width=0.45\linewidth,height=0.34\linewidth]{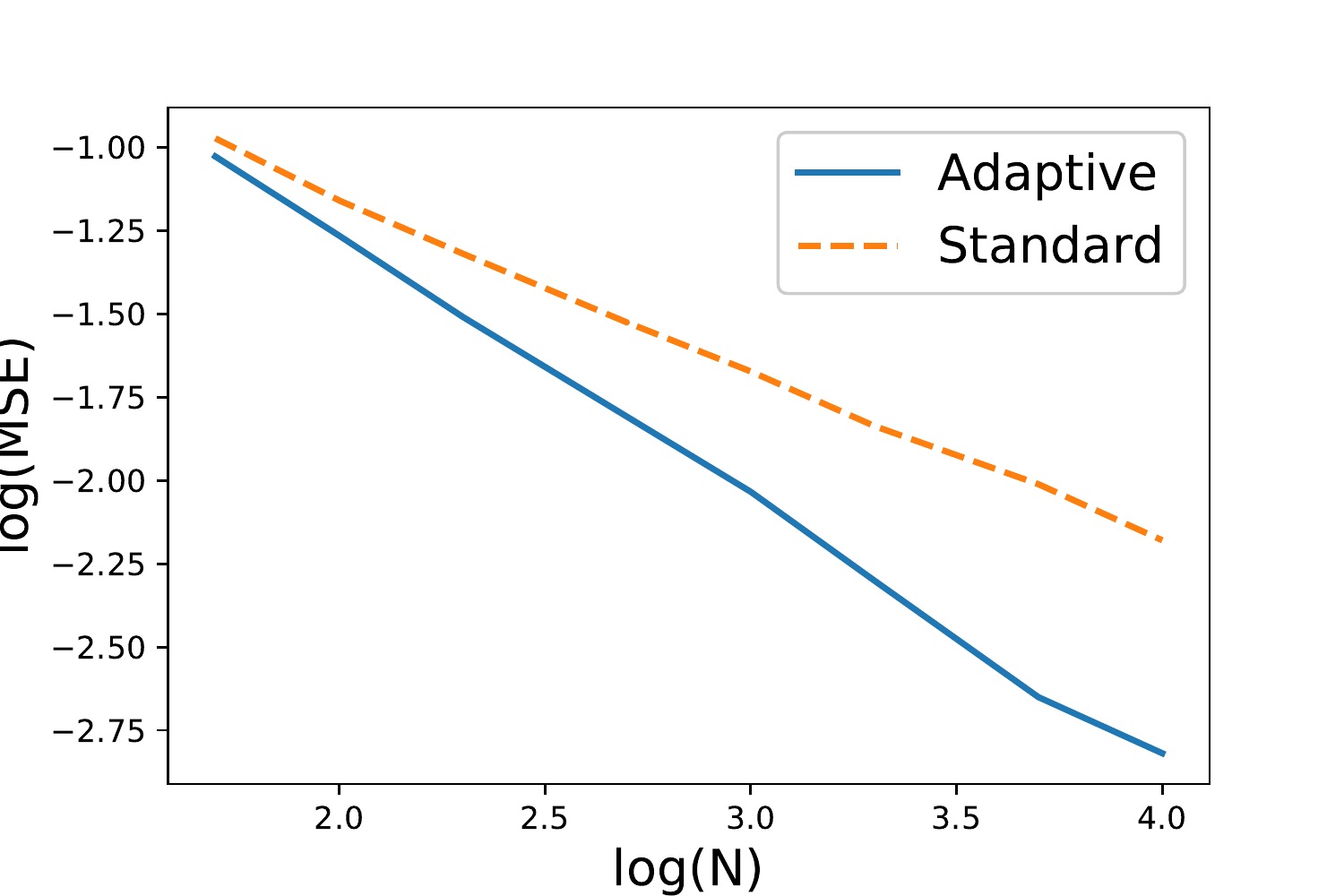}}		
		\subfigure[\small $t_2$ distribution, with regression function \hspace{1cm} $\eta(x)=\cos(5x)$.]{\includegraphics[width=0.45\linewidth,height=0.34\linewidth]{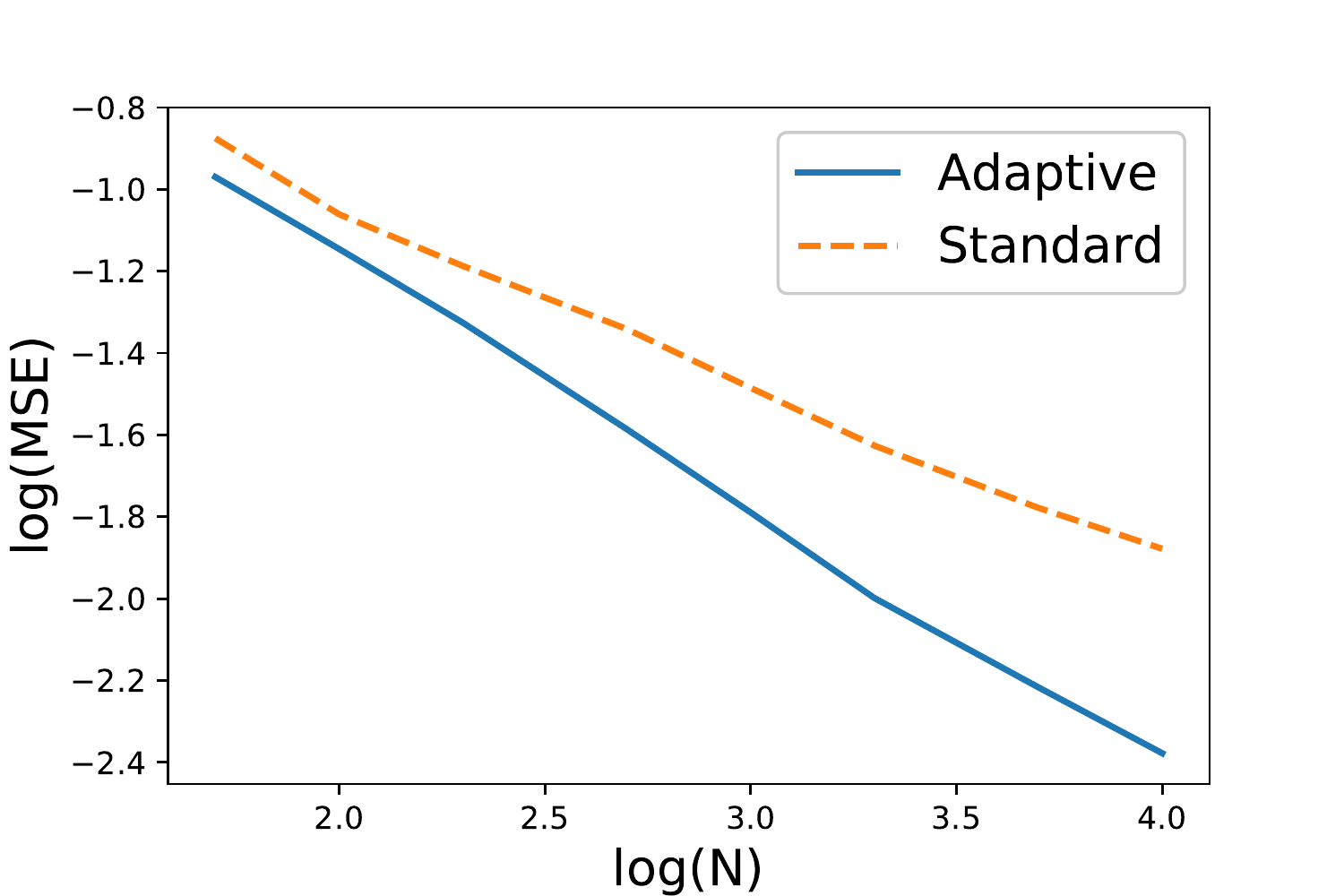}}
		\subfigure[\small 1d Laplace distribution. $\eta$ is determined in \eqref{eq:eta2}. ]{\includegraphics[width=0.48\linewidth,height=0.34\linewidth]{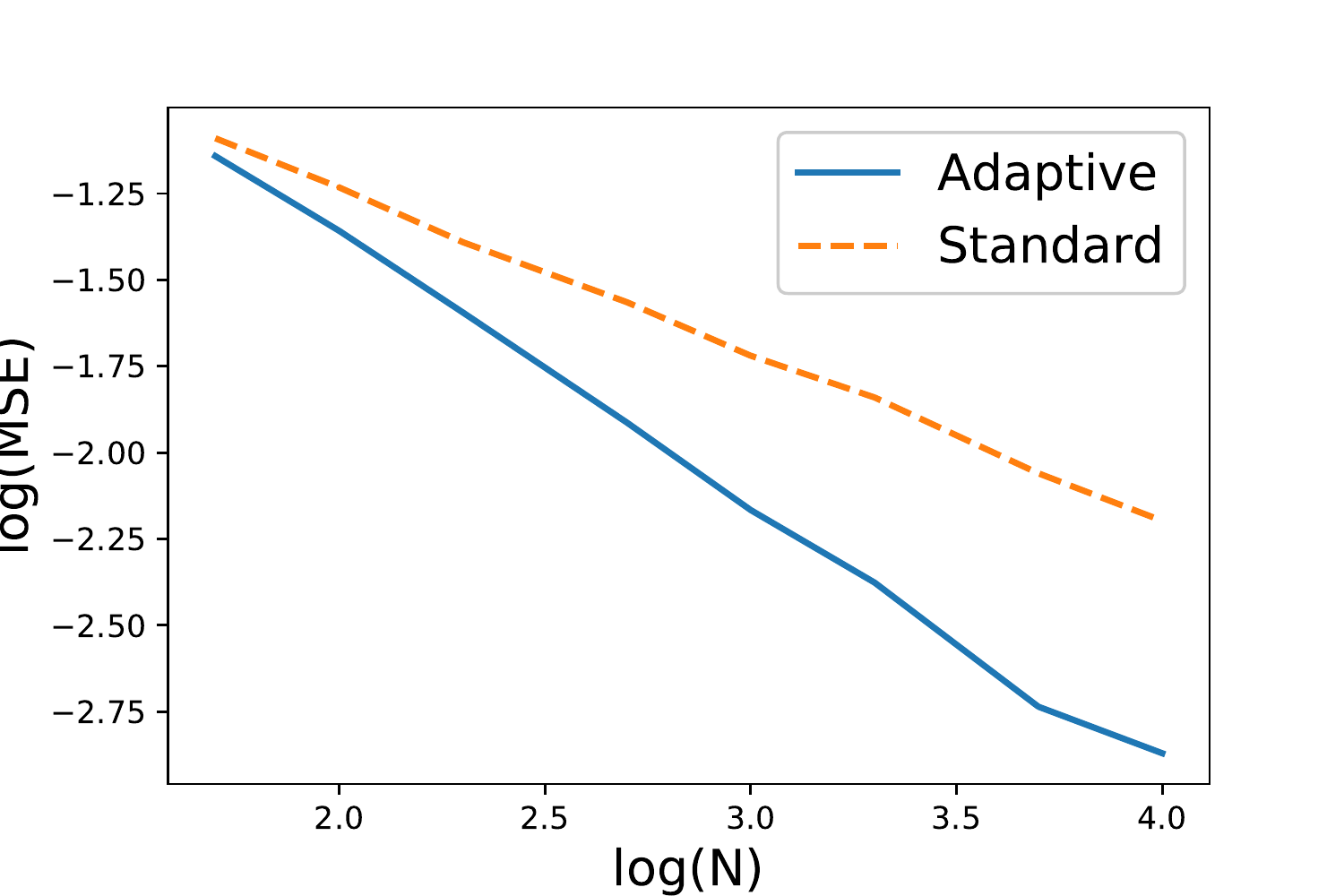}}	
		\caption{Comparison of excess risk of the proposed adaptive kNN classifier and the standard kNN classifier on one dimensional distributions. Blue line corresponds to the adaptive classifier. Orange dashed line corresponds to the standard classifier.}	\label{fig:compare}
	\end{center}
\end{figure}
\begin{figure}[h]	
	\begin{center}
		\subfigure[\small 2d Laplace distribution, with regression function $\eta(\mathbf{x})=\cos(2x_1+2x_2)$.]{\includegraphics[width=0.45\linewidth,height=0.33\linewidth]{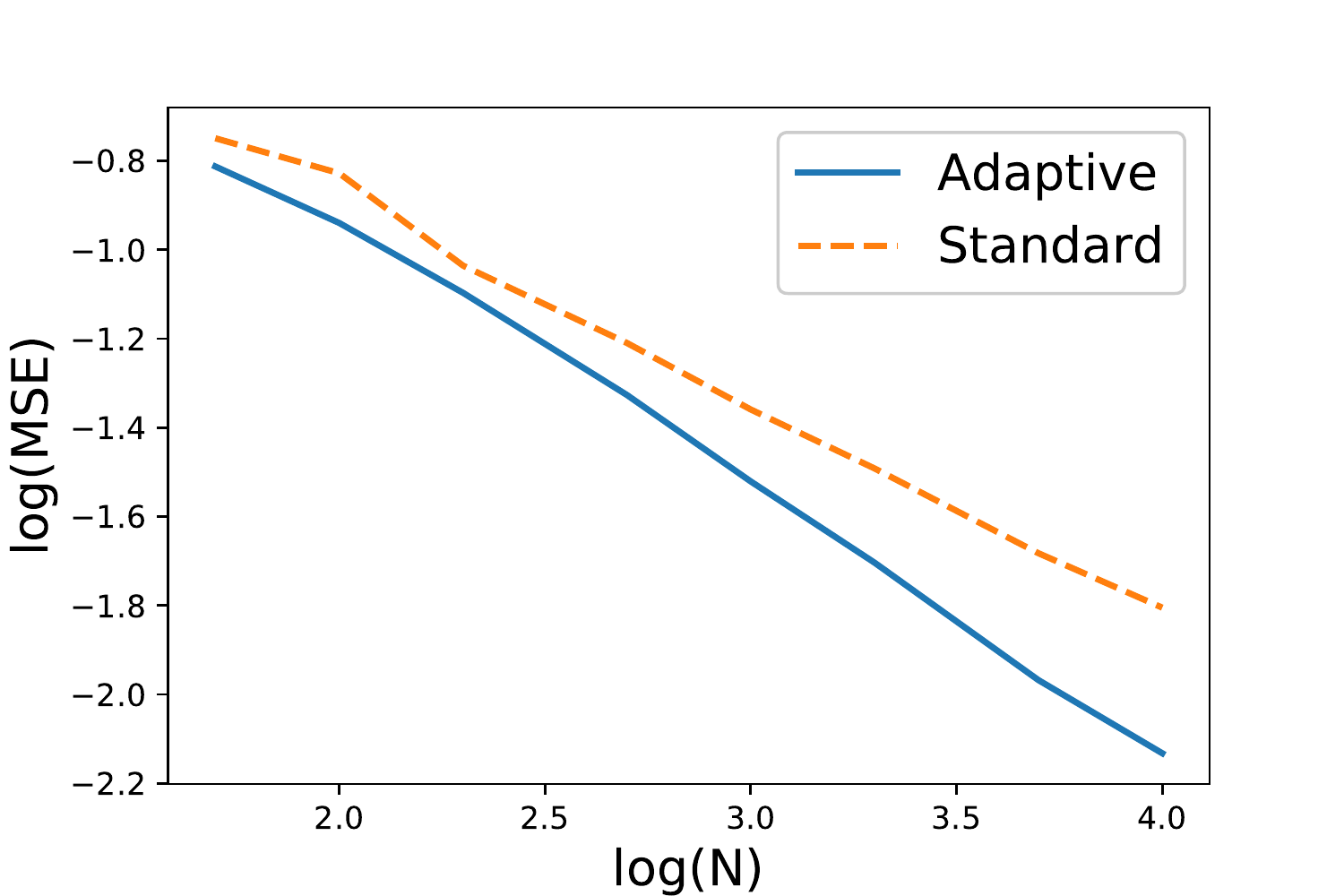}}	
		\subfigure[\small 2d Laplace distribution, $\eta(\mathbf{x})=\cos(2x_1)$.]{\includegraphics[width=0.45\linewidth,height=0.33\linewidth]{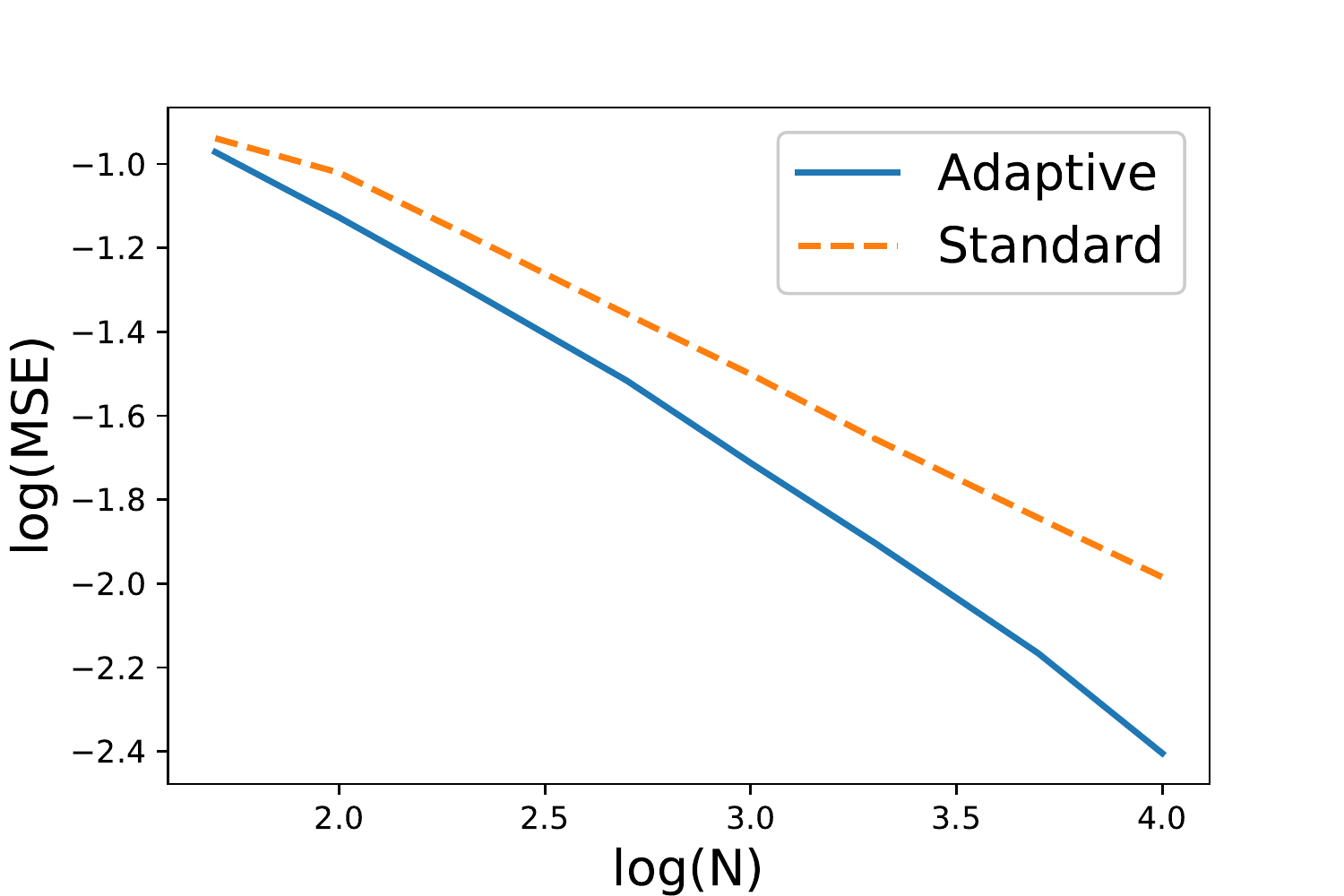}}		 		
	\end{center}
\caption{Numerical simulation for two dimensional distributions.}\label{fig:compare2}
\end{figure}
In Fig. \ref{fig:compare} (a)-(c), the underlying distributions are: (a)  Laplace distribution; (b) $t_5$ distribution; and (c) $t_2$ distribution, respectively. In these experiments, the underlying regression function is $\eta(x)=\cos(5x)$. In (d), the feature distribution is one dimensional standard Laplace distribution, and $\eta(x)$ is periodic, with period $2$. For $0\leq x<2$,
\begin{eqnarray}
\eta(x)=\left\{
\begin{array}{ccc}
2x &\text{if} & x\in [0,\frac{1}{2})\\
2(1-x) &\text{if} & x\in [\frac{1}{2},\frac{3}{2})\\
2(x-2) &\text{if} & x\in [\frac{3}{2},2)
\end{array}
\right..
\label{eq:eta2}
\end{eqnarray}

Fig.~\ref{fig:compare2} shows the simulation results for two dimensional cases. In both (a) and (b), the feature vector follows the standard Gaussian distribution. In (a), the regression function is $\eta(\mathbf{x})=\cos(2x_1+2x_2)$. This is an example where $\eta$ depends on two components of $\mathbf{X}$. In (b), $\eta(\mathbf{x})=\cos(2x_1)$, which implies that there is only one useful feature among two features. With these settings, we show the base-10 log-log plot of the classification error rate minus the Bayes risk, with respect to the training sample size. The test sample size is fixed at $N'=1000$, and each point in the curves is averaged over $1,000$ trials. 

In addition, in Table \ref{tab:classification}, we list the comparison of the empirical convergence rates and the theoretical convergence rates for both our adaptive kNN classifier and the standard one. The empirical convergence rates are the negative slope of the curves in Figures \ref{fig:compare} and \ref{fig:compare2}, which are calculated by linear regression. Theoretical rates are calculated from Theorems \ref{thm:standard} and \ref{thm:adaptive}. For presentation convenience, if the theoretical convergence rate is $\mathcal{O}(N^{-\mu})$, we then list $\mu$ in Table \ref{tab:classification}. For all cases in the simulation, we have $\alpha=1$. For Gaussian and Laplace distributions, $\beta=1$. For $t_5$ and $t_2$ distributions, $\beta=5/6$ and $0.5$, respectively. 
\begin{table}
	\begin{center}
	\caption{Comparison of convergence rates of kNN classification}
	\label{tab:classification}
	\begin{tabular}{ccc}
		\hline
		Distribution & Standard & Adaptive  \\
		& Empirical/Theoretical & Empirical/Theoretical\\
		\hline
		Fig 1(a) & 0.51/0.50 & 0.80/0.57 \\
		Fig 1(b) & 0.50/0.45 & 0.79/0.54 \\
		Fig 1(c) & 0.43/0.40 & 0.62/0.50 \\
		Fig 1(d) & 0.49/0.50 & 0.77/0.57 \\
		Fig 2(a) & 0.48/0.50 & 0.58/0.50 \\
		Fig 2(b) & 0.48/0.50 & 0.61/0.50 \\
		\hline
	\end{tabular}
\end{center}
\end{table}

The results from Figures \ref{fig:compare} and \ref{fig:compare2} show that the excess risk of both the standard kNN and our adaptive kNN method converges to zero with a stable convergence rate. Our result also indicates that the convergence rate of the standard kNN classifier is not optimal, due to the large kNN distances at the regions with low density. For all these distributions, our adaptive classifier significantly outperforms the standard one. If the sample size is large, then the advantage of our new classifier is more obvious. This observation is consistent with our theoretical analysis. Moreover, as discussed before, the convergence rate for the standard kNN method is obtained using the optimal choice of $k$ that depends on unknown parameters $\alpha$ and $\beta$. In practice, such information is not available, thus the convergence rate is usually worse if we pick a suboptimal selection rule of $k$.

\textcolor{black}{We also observe from Table \ref{tab:classification} that for all these six cases, the empirical convergence rates of the standard kNN classifiers are close to the theoretical rate indicated in Theorem \ref{thm:standard}. However, the adaptive kNN method actually converges faster than the theoretical results from Theorem \ref{thm:adaptive}. This phenomenon can be explained by the fact that all results derived in Section \ref{sec:cla} are rates of uniform convergence. For a specific distribution, the bound may not be tight. }

\subsection{Regression}
Now we compare the empirical convergence rates of the adaptive and the standard kNN regression.

We first present results for one dimension case. In our numerical experiments, $X$ follows standard Laplace, $t_2$ and Cauchy distribution, respectively, corresponding to different tail strength. For each distribution, we conduct simulations with $\eta(x)=\sin (x)$ and $\eta(x)=x$ separately, in which the former one is an example of bounded regression function, and the latter one is an example of unbounded regression function. Similar to the simulation of kNN classification, we still tune the parameter $k$ and $K$ optimally at $N=500$ first. Moreover, we fix $q=4/(d+4)=0.8$ and $A=0.5$ for our adaptive method for all of these experiments.

Fig. \ref{fig:compare-reg} shows the log-log plot of the mean square estimation error against the training sample size $N$, for some one dimensional distributions, in which each curve is averaged over $500$ trials. It can be shown that the expectation of mean square error is the excess risk $R-R^*$. From Fig. \ref{fig:compare-reg}, we observe that our new adaptive regression method significantly outperforms the standard kNN method, especially for large sample sizes. For $t_2$ and Cauchy distributions, we only plot the result with a bounded regression function. For the unbounded case, the curves are not plotted since the estimated MSE error of both regression methods are unstable for these two distributions. This phenomenon is reasonable, because for these two distributions, $\mathbb{E}[X^2]$ is infinite, which violates Assumption \ref{ass:unbounded}. As a result, $R-R^*$ is infinite. 
\begin{figure}[h]
	\begin{center}
		\subfigure[Bounded Laplace]{\includegraphics[width=0.48\linewidth,height=0.33\linewidth]{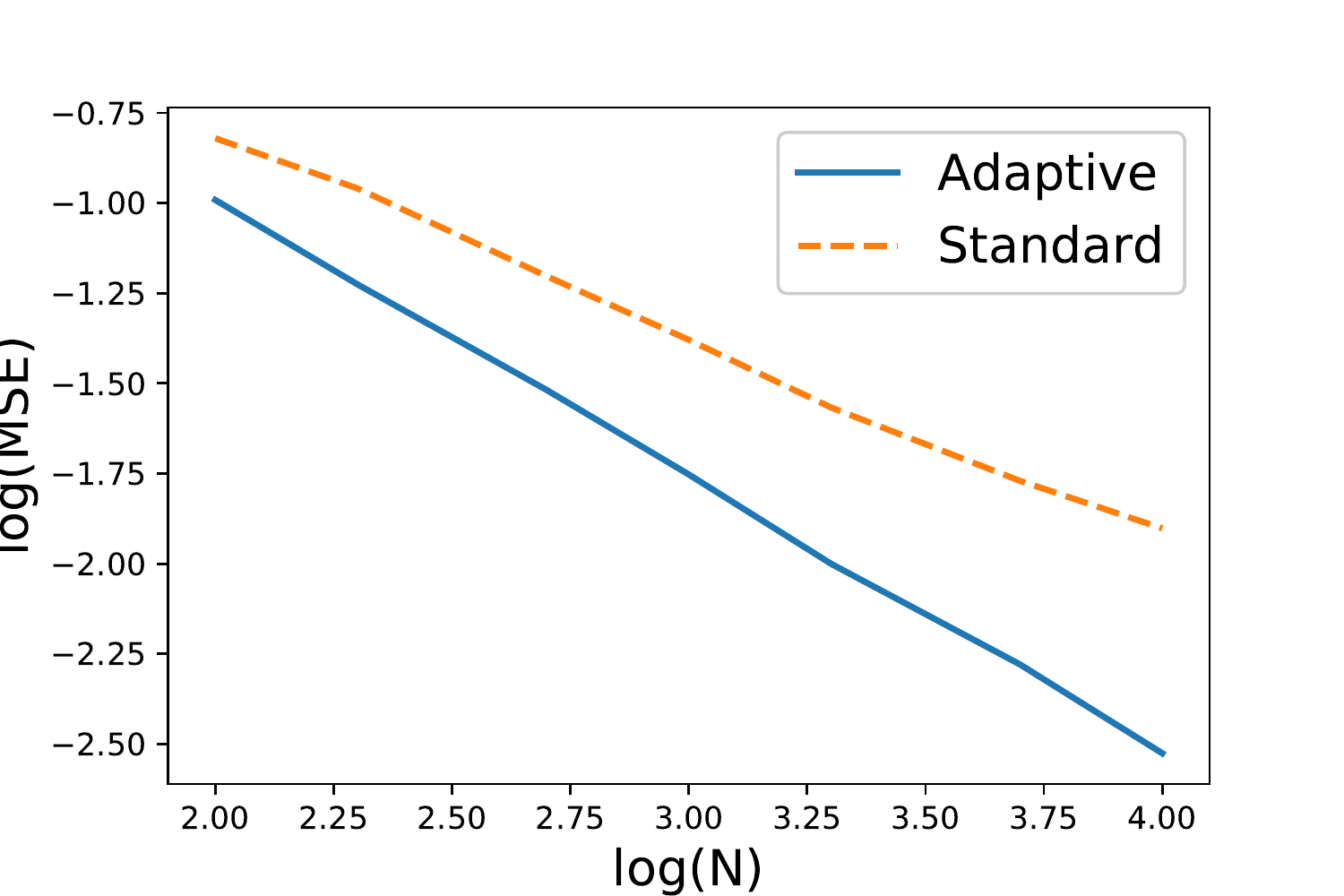}}
		\subfigure[Unbounded Laplace]{\includegraphics[width=0.48\linewidth,height=0.33\linewidth]{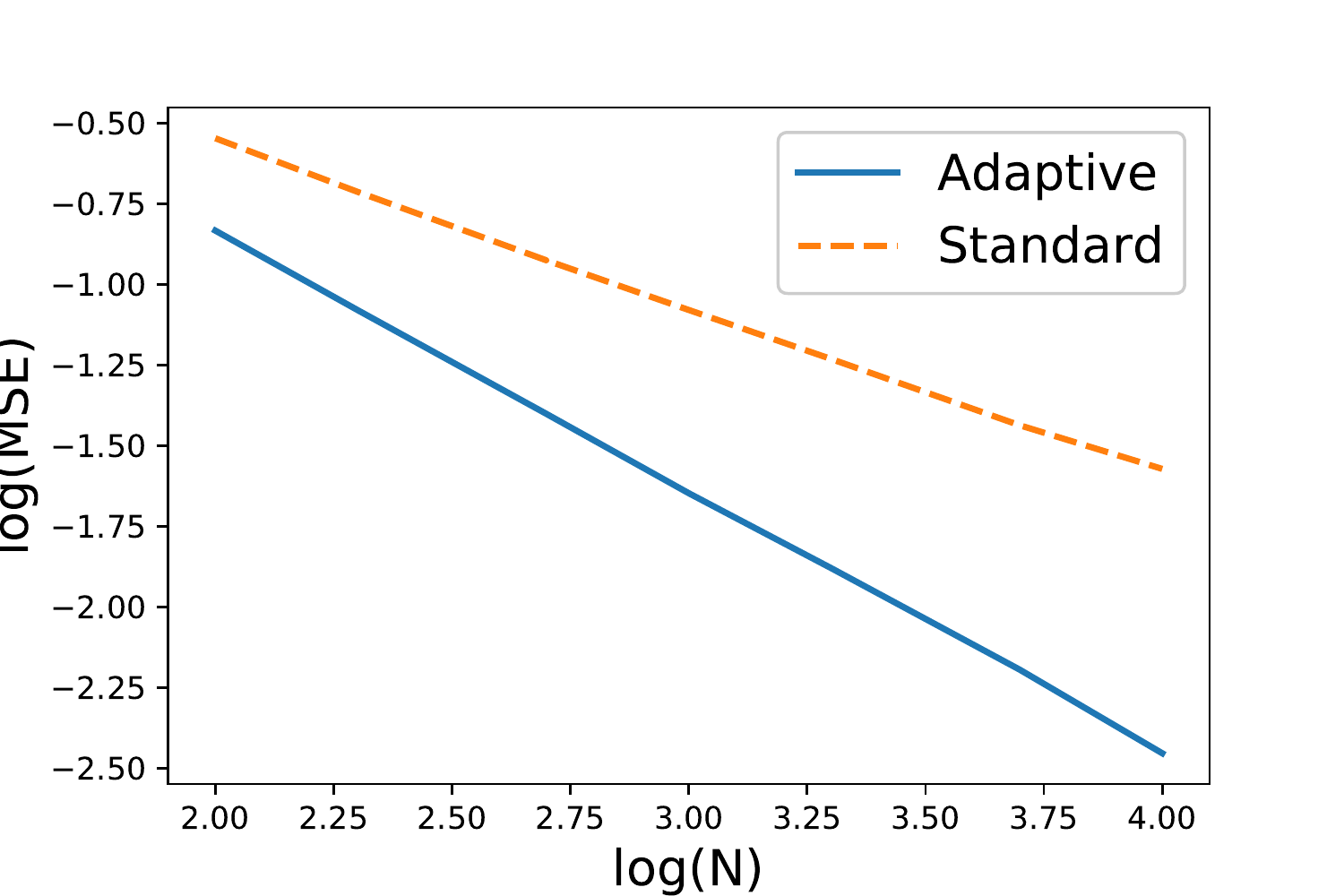}}
		\subfigure[Bounded $t_2$]{\includegraphics[width=0.48\linewidth,height=0.33\linewidth]{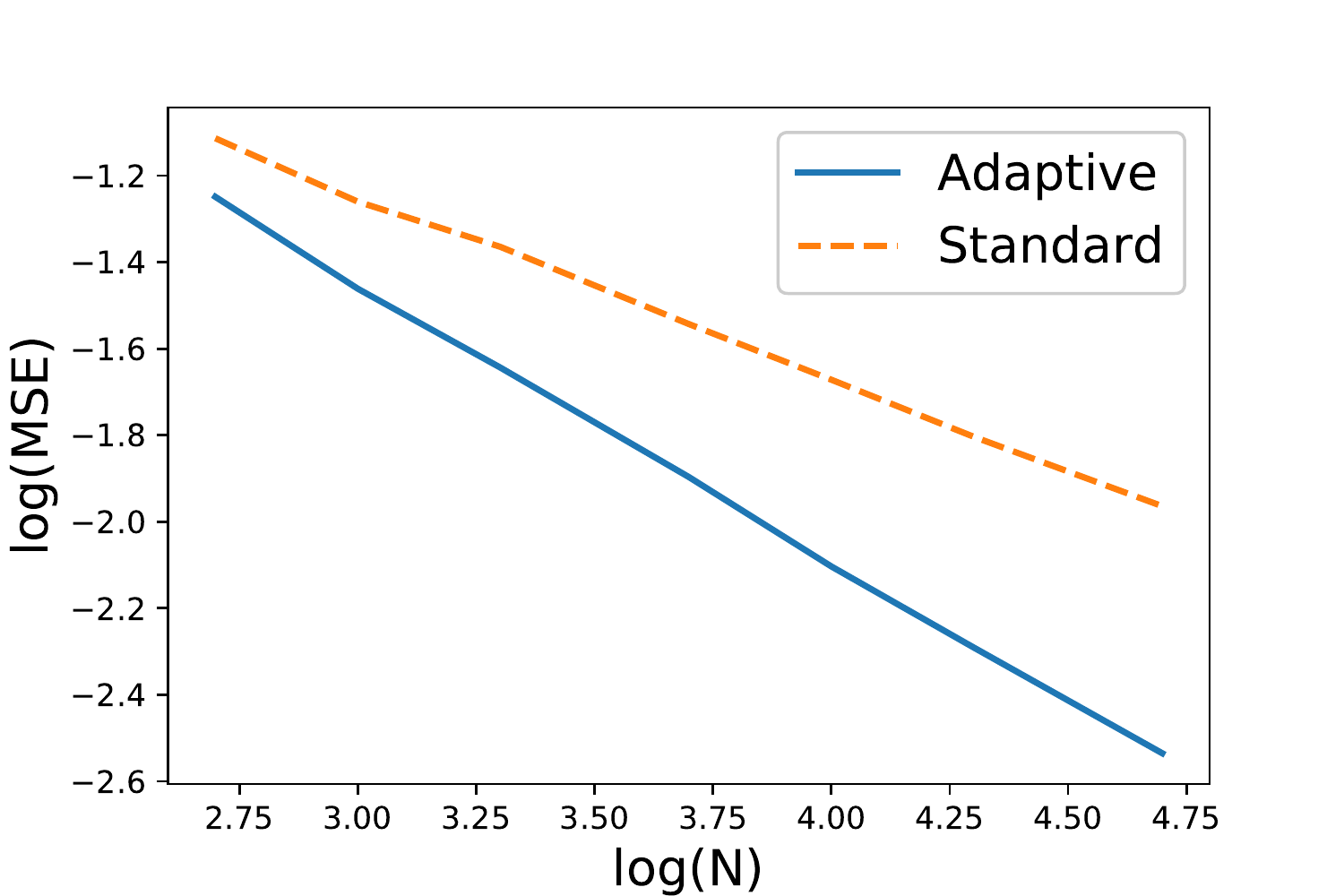}}
		\subfigure[Bounded Cauchy]{\includegraphics[width=0.48\linewidth,height=0.33\linewidth]{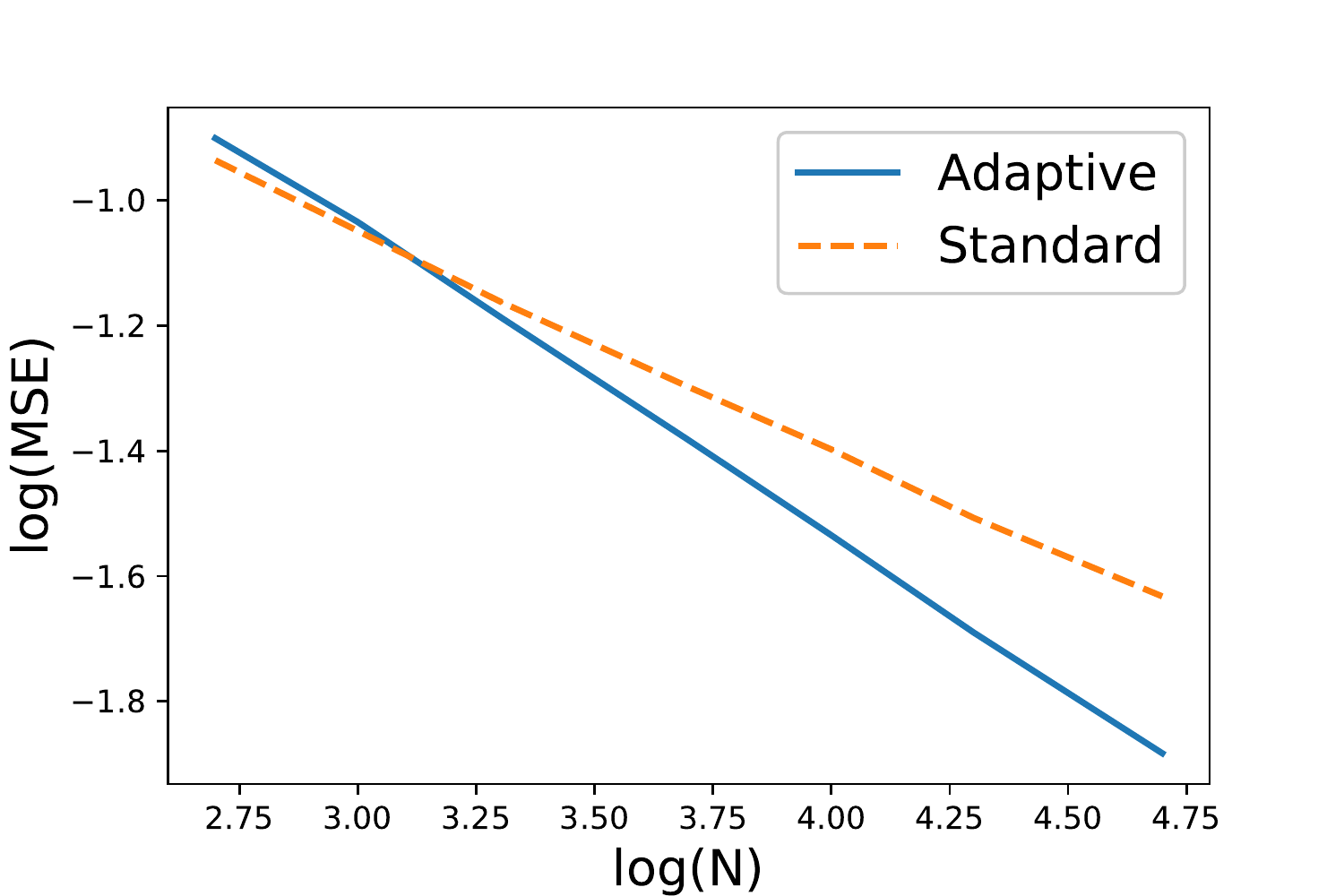}}										
		\caption{MSE of the proposed adaptive kNN regression method vs the standard kNN regression with $d=1$. Blue line corresponds to adaptive regression. Orange dashed line corresponds to the standard kNN regression.}	\label{fig:compare-reg}
	\end{center}
\end{figure}

Fig. \ref{fig:compare-reg2} shows simulation results for distributions with higher dimensions. We focus on Laplace distribution with $d=2$ and $d=3$. The parameter selection follows the same rule as the case with $d=1$, and the parameter $q$ of the adaptive method is selected according to $q=4/(d+4)$.

\begin{figure}[h!]
	\subfigure[Unbounded Laplace, $d=2$]{\includegraphics[width=0.48\linewidth,height=0.33\linewidth]{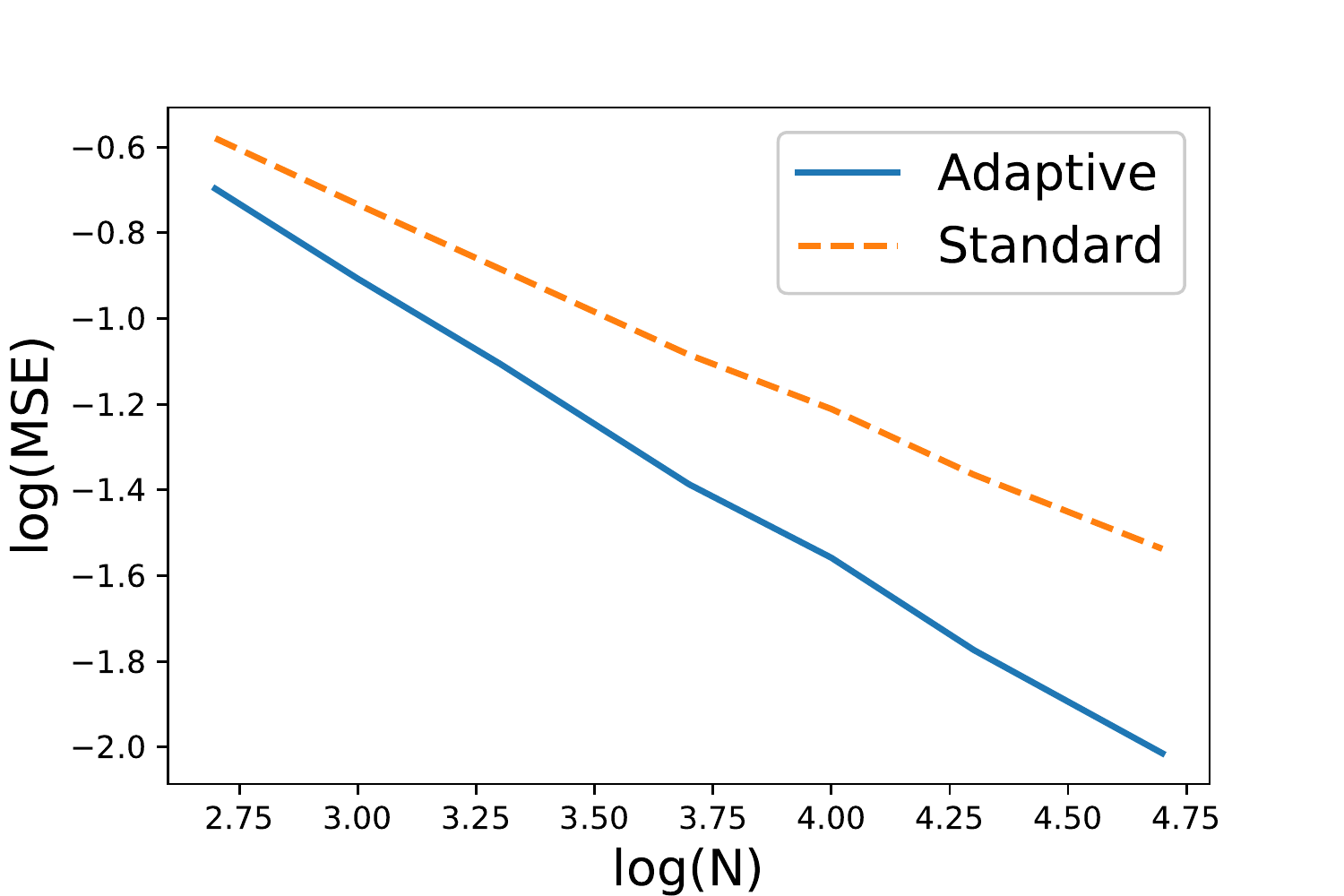}}
	\subfigure[Unbounded Laplace, $d=3$]{\includegraphics[width=0.48\linewidth,height=0.33\linewidth]{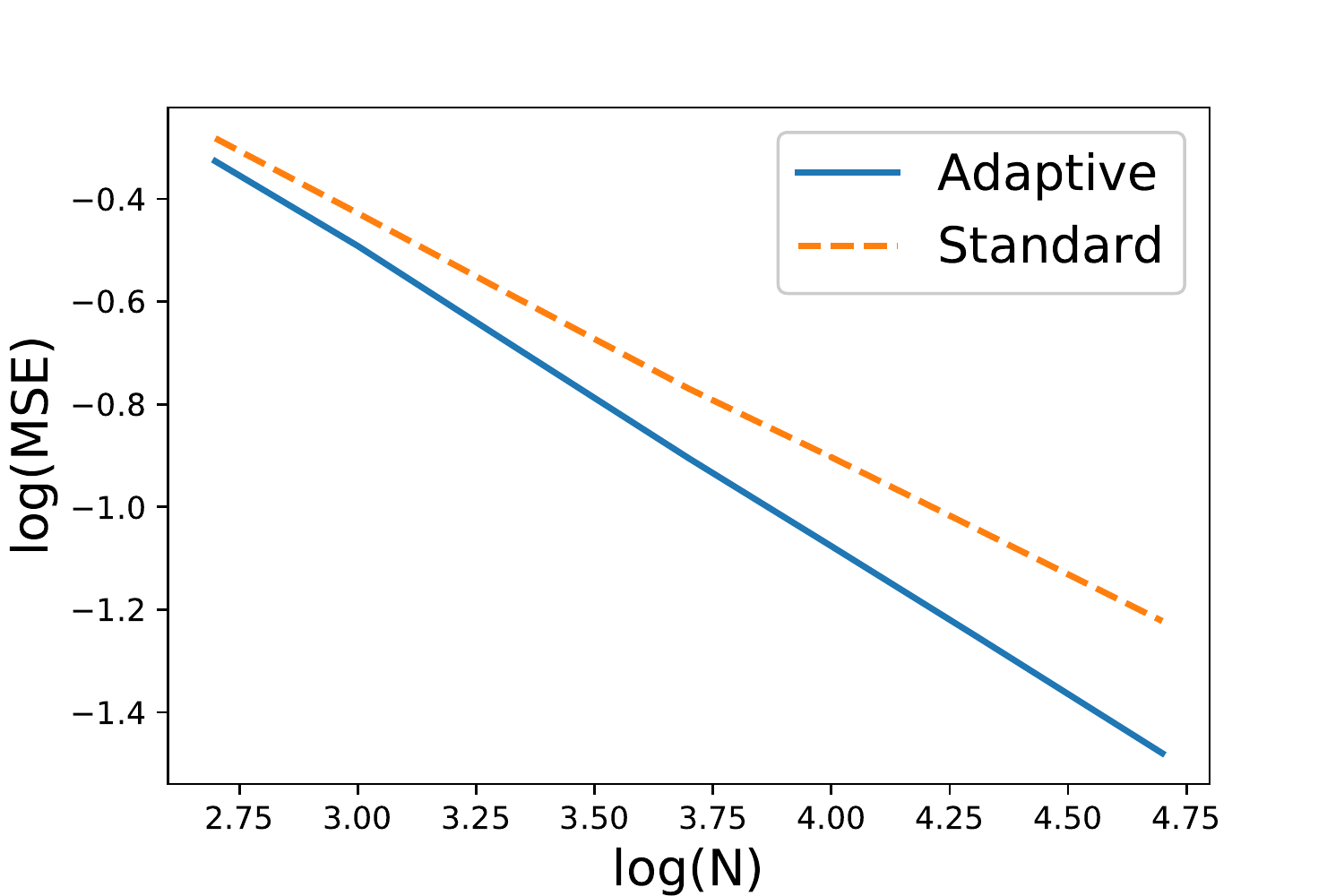}}
	\caption{MSE of the proposed adaptive kNN regression method vs the standard kNN regression for higher dimensions. Blue line corresponds to adaptive regression. Orange dashed line corresponds to the standard kNN regression.}
	\label{fig:compare-reg2}
\end{figure}
Moreover, we compare the empirical and theoretical convergence rates in Table \ref{tab:regression}. We use the same methods to calculate these rates as are already used in Table \ref{tab:classification}. 

\begin{table*}
	\begin{center}
	\caption{Comparison of convergence rates of kNN regression}
	\label{tab:regression}
	\begin{tabular}{ccccc}
		\hline
		Distribution & dimension & $\eta$& Standard & Adaptive  \\
		&&& Empirical/Theoretical & Empirical/Theoretical\\
		\hline
		Laplace &1 &Bounded & 0.55/0.50 &  0.77/0.80 \\
		Laplace &1 &Unbounded& 0.51/0.50 & 0.81/0.80 \\
		$t_2$ &1&Bounded & 0.42/0.40 & 0.65/0.66 \\
		Cauchy &1&Bounded & 0.34/0.33 & 0.50/0.50 \\
		Laplace &2& Unbounded & 0.48/0.50 & 0.66/0.67\\
		Laplace &3& Unbounded & 0.48/0.50 & 0.57/0.57\\
		\hline
	\end{tabular}
\end{center}
\end{table*}

The results in Fig.\ref{fig:compare-reg}, Fig.\ref{fig:compare-reg2} and Table \ref{tab:regression} agree with our theoretical prediction. All of the above results show that the adaptive kNN regression significantly outperforms the standard one, and the empirical convergence rate agrees well with our theoretical prediction.

\section{Conclusion}\label{sec:conc}
In this paper, we have analyzed the convergence rate of the standard kNN classification and regression, and derived a minimax lower bound for all nonparametric classification methods, under some tail, smoothness and margin assumptions. Building on these analysis, which show that there is a gap between the convergence rates of the standard kNN and the minimax bound, we have then proposed an adaptive kNN method to close this gap, which can be used for both classification and regression problems. In the proposed method, we select $k$ based on the number of training samples in the fixed radius nearest neighbor of the test point. We have obtained an upper bound of the excess risk of the proposed method that matches the minimax lower bound under some general assumptions. For regression problems, we have extended our analysis to cases with unbounded regression function $\eta$. Since the most important parameter of our adaptive kNN method, i.e., $q$, can be selected without any knowledge of the underlying distribution, the parameter tuning of our adaptive kNN method is simpler than the standard one. Moreover, numerical results illustrate that our new method significantly outperforms the standard kNN method, especially for large training datasets.	

\color{black}
	\ifCLASSOPTIONcaptionsoff
	\newpage
	\fi



\newpage
\appendices

\section{Proof of Proposition \ref{prop:smoothness} (B)}\label{sec:smoothness}
Here, we prove that if the conditions in Proposition 1 (B) are satisfied, then with Assumption 1 (d), Assumption 1(c) is also satisfied. For presentation simplicity, in the following proof, we assume that $\ell_2$ norm is used. According to the definition of function $\eta$, we have
\begin{eqnarray}
|\eta(B(\mathbf{x},r))-\eta(\mathbf{x})|=\left|\frac{1}{\text{P}(B(\mathbf{x},r))}\int_{B(\mathbf{x},r)} f(\mathbf{u})\eta(\mathbf{u})d\mathbf{u}-\frac{1}{\text{P}(B(\mathbf{x},r))}\int_{B(\mathbf{x},r)} f(\mathbf{u})\eta(\mathbf{x})d\mathbf{u}\right|.
\end{eqnarray}
By Taylor expansion, we have 
$\eta(\mathbf{u})=\eta(\mathbf{x})+\nabla\eta(\mathbf{x})^T (\mathbf{u}-\mathbf{x})+\frac{1}{2}(\mathbf{u}-\mathbf{x})^T\nabla^2\eta(\xi(\mathbf{u})) (\mathbf{u}-\mathbf{x})$
for some $\xi(\mathbf{u})$ that is in between $\mathbf{u}$ and $\mathbf{x}$. Hence
\begin{eqnarray}
|\eta(B(\mathbf{x},r))-\eta(\mathbf{x})|=\left|\frac{1}{\text{P}(B(\mathbf{x},r))} \int_{B(\mathbf{x},r)} f(\mathbf{u})\left(\nabla \eta(\mathbf{x})^T (\mathbf{u}-\mathbf{x})+\frac{1}{2}(\mathbf{u}-\mathbf{x})^T \nabla^2 \eta(\xi(\mathbf{u}))(\mathbf{u}-\mathbf{x})\right)d\mathbf{u} \right|.\nonumber
\end{eqnarray}
Note that due to symmetry, we have $\int_{B(\mathbf{x},r)} f(\mathbf{x})\nabla\eta(\mathbf{x})^T(\mathbf{u}-\mathbf{x})d\mathbf{u}=0$. Then for any $r<D'$,
\begin{eqnarray}
\int_{B(\mathbf{x},r)} f(\mathbf{u})\nabla \eta(\mathbf{x})^T (\mathbf{u}-\mathbf{x})d\mathbf{u}&=&\int_{B(\mathbf{x},r)} (f(\mathbf{u})-f(\mathbf{x}))\nabla \eta(\mathbf{x})^T (\mathbf{u}-\mathbf{x}) d\mathbf{u}\nonumber\\
&\leq & \int_{B(\mathbf{x},r)} \left(\underset{\mathbf{v}\in B(\mathbf{x},D')}{\sup} \frac{\norm{\nabla \eta(\mathbf{x})}\norm{\nabla f(\mathbf{v})}}{f(\mathbf{x})}\right) f(\mathbf{x})\norm{\mathbf{u}-\mathbf{x}}^2 d\mathbf{u}\nonumber\\
&\leq & C_0 r^2 f(\mathbf{x})V(B(\mathbf{x},r)).
\end{eqnarray}
In addition,
\begin{eqnarray}
\int \frac{1}{2}f(\mathbf{u})(\mathbf{u}-\mathbf{x})^T \nabla^2 \eta(\xi(\mathbf{u}))(\mathbf{u}-\mathbf{x})d\mathbf{u}\leq \frac{1}{2}C_H\int f(\mathbf{u})\norm{\mathbf{u}-\mathbf{x}}^2 d\mathbf{u}\leq \frac{1}{2}C_H r^2 \text{P}(B(\mathbf{x},r)).
\end{eqnarray}
Therefore,
\begin{eqnarray}
|\eta(B(\mathbf{x},r))-\eta(\mathbf{x})|&\leq& \frac{1}{\text{P}(B(\mathbf{x},r))}\left(C_0 r^2 f(\mathbf{x}) V(B(\mathbf{x},r))+\frac{1}{2} C_H r^2 \text{P}(B(\mathbf{x},r))\right)\\
&\leq &\left(\frac{C_0}{C_d}+\frac{1}{2} C_H\right)r^2,\nonumber
\end{eqnarray}
in which the last step uses Assumption 1 (d).
\section{Proof of Theorem \ref{thm:standard}: Convergence rate of the standard kNN classification}\label{sec:original}
\color{black}
\subsection{Upper Bound}\label{sec:original-ub}
In this section, we prove the convergence rate of an upper bound of the excess risk of the standard kNN classification under Assumption 1. Recall that $R$ and $R^*$ are defined as
$R=\text{P}(g(\mathbf{X})\neq Y)$, $R^*=\text{P}(g^*(\mathbf{X})\neq Y)$,
in which
\begin{eqnarray}
g(\mathbf{x})&=&\text{sign}(\hat{\eta}(\mathbf{x})),\label{eq:gknn}\\
g^*(\mathbf{x})&=&\text{sign}(\eta(\mathbf{x})),\label{eq:gbayes}
\end{eqnarray}
and $\hat{\eta}(\mathbf{x})$ is defined in~\eqref{eq:etahat}.

Hence we have
\begin{eqnarray}
R-R^*&=&\mathbb{E}\left[\text{P}(g(\mathbf{X})\neq Y|\mathbf{X}=\mathbf{x})-\text{P}(g^*(\mathbf{X})\neq Y|\mathbf{X}=\mathbf{x}) \right] \nonumber\\
&=&\mathbb{E}[\mathbf{1}(g(\mathbf{X})\neq g^*(\mathbf{X}))|\eta(\mathbf{X})|].
\label{eq:exrisk}
\end{eqnarray}
We divide the support into four regions:
\begin{eqnarray}
S_1&=&\{\mathbf{x}|f(\mathbf{x})\geq N^{-\delta}, |\eta(\mathbf{x})|>2\Delta\}; \label{eq:S1def}\\
S_2&=&\{\mathbf{x}|f(\mathbf{x})\geq N^{-\delta}, |\eta(\mathbf{x})|\leq 2\Delta\};\\
S_3&=&\left\{ \mathbf{x}|C_0 \frac{k}{N}<f(\mathbf{x})<N^{-\delta} \right\};\\
S_4&=&\left\{\mathbf{x}|f(\mathbf{x})\leq C_0\frac{k}{N}\right\},
\end{eqnarray}
in which $\Delta$ and $\delta$ are two parameters that will be determined later, and $C_0=2/(C_dv_dD^d)$.

Then we can rewrite the excess risk as
\begin{eqnarray}
R-R^*=\sum_{i=1}^4 \mathbb{E}\left[\mathbf{1}(g(\mathbf{X})\neq g^*(\mathbf{X} ))|\eta(\mathbf{X})|\mathbf{1}(\mathbf{X}\in S_i)\right] := I_1+I_2+I_3+I_4,
\label{eq:Rdecomp}
\end{eqnarray}
in which $\mathbf{1}(\cdot)$ is the indication function. In the following, we bound these four terms separately. 

Firstly, for $I_2$ we have
\begin{eqnarray}
I_2&=&\mathbb{E}[\mathbf{1}(g(\mathbf{X})\neq g^* (\mathbf{X}))|\eta(\mathbf{X})|\mathbf{1}(f(\mathbf{X})\geq N^{-\delta},|\eta(\mathbf{X})|\leq 2\Delta)] \leq  \text{P}(|\eta(\mathbf{X})|\leq 2\Delta) 2\Delta \leq  C_a (2\Delta)^{\alpha+1},\nonumber\\
\label{eq:I2}
\end{eqnarray}
in which the last inequality uses Assumption 1 (a). 

Secondly, for $I_4$, we have
\begin{eqnarray}
I_4=\mathbb{E}\left[\mathbf{1}(g(\mathbf{X})\neq g^*(\mathbf{X}))|\eta(\mathbf{X})|\mathbf{1}\left(f(\mathbf{X})<C_0\frac{k}{N}\right)\right] \leq \text{P}\left(f(\mathbf{X})\leq C_0\frac{k}{N}\right) \leq  C_b \left(C_0 \frac{k}{N}\right)^\beta,
\label{eq:I3}
\end{eqnarray}
in which we use Assumption 1 (b).

Now it remains to bound $I_1$ and $I_3$. 

\textbf{Bound of $I_1$.}
Define
\begin{eqnarray}
a_N=\left(\frac{2k}{C_d v_d} N^{\delta-1}\right)^\frac{1}{d},
\label{eq:an}
\end{eqnarray}
in which $v_d$ is the volume of the ball with unit radius, depending on the distance metric we use. For example, if we use Euclidean distance, then $v_d=\pi^\frac{d}{2}/\Gamma(\frac{d}{2}+1)$, in which $\Gamma$ is the Gamma function, 
\begin{eqnarray}\Gamma(u)=\int_0^\infty t^{u-1} e^{-t}dt,\; u>0.\label{eq:Gamma}\end{eqnarray}
From now on, we assume that
\begin{eqnarray}
\underset{N\rightarrow \infty}{\lim} kN^{\delta-1}=0.
\label{eq:klim}
\end{eqnarray} 
\eqref{eq:klim} will be checked after we finish the proof. With \eqref{eq:klim}, for sufficiently large $N$, $a_N<D$. According to Assumption 1 (d), for all $\mathbf{x}\in S_1$,
\begin{eqnarray}
\mathbf{P}(B(\mathbf{x},a_N))&\geq & C_d f(\mathbf{x}) v_d a_N^d= C_d f(\mathbf{x})v_d \frac{2k}{C_dv_d} N^{\delta-1}\geq \frac{2k}{N},
\end{eqnarray}
in which the last inequality uses the definition of $S_1$ \eqref{eq:S1def}. Denote $\rho$ as the distance from the test point $\mathbf{x}$ to its $(k+1)$-th nearest neighbor, then according to Chernoff inequality, for all $\mathbf{x}\in S_1$,
\begin{eqnarray}
\text{P}(\rho>a_N|\mathbf{x})&\leq &e^{-N\text{P}(B(\mathbf{x},a_N))} \left(\frac{eN\text{P}(B(\mathbf{x},a_N))}{k}\right)^k\leq  e^{-2k} (2e)^k= e^{-k(1-\ln 2)}.
\label{eq:largerho}
\end{eqnarray}
Recall the definition of $g$ and $g^*$ in \eqref{eq:gknn} and \eqref{eq:gbayes}, if $\text{sign}(\hat{\eta}(\mathbf{x}))\neq \text{sign}(\eta(\mathbf{x}))$, then we must have $|\hat{\eta}(\mathbf{x})-\eta(\mathbf{x})|>|\eta(\mathbf{x})|$. Therefore, for all $\mathbf{x}\in S_1$, the misclassification probability is bounded by
\begin{eqnarray}
&&\hspace{-8mm} \text{P}(g(\mathbf{x})\neq g^*(\mathbf{x}))\nonumber\\
  &\leq&  \text{P}(\rho>a_N|\mathbf{x})+\text{P}(\rho\leq a_N,|\hat{\eta}(\mathbf{x})-\eta(\mathbf{x})|>|\eta(\mathbf{x}) | | \mathbf{x}) \nonumber \\
&\leq &e^{-k(1-\ln 2)} + \text{P}(\rho\leq a_N, |\hat{\eta}(\mathbf{x})-\eta(B(\mathbf{x},\rho))|>|\eta(\mathbf{x})|-|\eta(\mathbf{x})-\eta(B(\mathbf{x},\rho))| |\mathbf{x}),
\label{eq:errprob}
\end{eqnarray}
in which the last inequality uses \eqref{eq:largerho} and triangular inequality. For the second term, according to Assumption 1 (c), and let $\Delta=C_ca_N^2$:
\begin{eqnarray}
|\eta(B(\mathbf{x},\rho))-\eta(\mathbf{x})|\leq C_c\rho^2 \leq C_c a_N^2 :=\Delta.
\label{eq:delta}
\end{eqnarray}
Recall that $\hat{\eta}(\mathbf{x})=\frac{1}{k}\sum_{i=1}^k Y^{(i)}$. Here $Y^{(i)}$ are not independent. However, we can show that the Hoeffding's inequality still holds. We provide a proof in Appendix~\ref{sec:lemmas}, Lemma \ref{lem:conc}. Based on Lemma \ref{lem:conc} in Appendix~\ref{sec:lemmas}, \eqref{eq:errprob} and \eqref{eq:delta}, we have for all $\mathbf{x}\in S_1$,
\begin{eqnarray}
\text{P}(g(\mathbf{x})\neq g^*(\mathbf{x}))\leq e^{-k(1-\ln 2)}+2e^{-\frac{1}{2}k(\eta(\mathbf{x})-\Delta)_+^2},
\label{eq:errprob2}
\end{eqnarray}
in which we define $U_+=\max\{U,0\}$. Then for all $\mathbf{x}\in S_1$, we have
\begin{eqnarray}
\eta(\mathbf{x})-\Delta>\frac{1}{2}\eta(\mathbf{x}).
\label{eq:eta}
\end{eqnarray}  
Plug \eqref{eq:eta} into \eqref{eq:errprob2}, then $I_1$ can be bounded as following:
\begin{eqnarray}
I_1&=&\mathbb{E}[\mathbf{1}(g(\mathbf{X})\neq g^*(\mathbf{X}))|\eta(\mathbf{X})|\mathbf{1}(\mathbf{X}\in S_1)] \nonumber \\
&\leq & e^{-k(1-\ln 2)}+2\mathbb{E}[|\eta(\mathbf{X})|e^{-\frac{1}{8}k|\eta(\mathbf{X})|^2}]. 
\label{eq:I1}
\end{eqnarray}
The first term of \eqref{eq:I1} decays exponentially. For the second term, using Assumption \ref{ass:basic}(a),
\begin{eqnarray}
\mathbb{E}[|\eta(\mathbf{X})|e^{-\frac{1}{8}k|\eta(\mathbf{X})|^2}]&=&\frac{1}{\sqrt{k}}\mathbb{E}\left[\left(\sqrt{k}|\eta(\mathbf{X})|e^{-\frac{1}{16}k|\eta(\mathbf{X})|^2}\right)e^{-\frac{1}{16}k|\eta(\mathbf{X})|^2}\right]\nonumber\\
&\leq &\frac{2\sqrt{2}e^{-\frac{1}{2}}}{\sqrt{k}}\mathbb{E}\left[e^{-\frac{1}{16}k|\eta(\mathbf{X})|^2}\right]\nonumber\\
&=&\frac{2\sqrt{2}e^{-\frac{1}{2}}}{\sqrt{k}}\int_0^1 \text{P}\left(e^{-\frac{1}{16}k|\eta(\mathbf{X})|^2}>t\right)dt\nonumber\\
&=&\frac{2\sqrt{2}e^{-\frac{1}{2}}}{\sqrt{k}}\int_0^1 \text{P}\left(|\eta(\mathbf{X})|<4\sqrt{\frac{\ln(1/t)}{k}}\right)dt\nonumber\\
&\leq &\frac{2\sqrt{2}e^{-\frac{1}{2}}}{\sqrt{k}}\int_0^1 C_a\left(4\sqrt{\frac{\ln(1/t)}{k}}\right)^\alpha dt.
\label{eq:etabound}
\end{eqnarray}
Therefore this term decays with $\mathcal{O}\left(k^{-\frac{\alpha+1}{2}}\right)$. Combine two terms in \eqref{eq:I1}, we get
\begin{eqnarray}
I_1= \mathcal{O}\left(k^{-\frac{\alpha+1}{2}}\right).
\end{eqnarray} 

\textbf{Bound of $I_3$.} According to the definition of $I_3$ in \eqref{eq:Rdecomp},
\begin{eqnarray}
I_3&=&\mathbb{E}\left[\mathbf{1}(g(\mathbf{X})\neq g^*(\mathbf{X}))|\eta(\mathbf{X})|\mathbf{1}\left(C_0\frac{k}{N}<f(\mathbf{X})<N^{-\delta}\right)\right]\nonumber\\
&\leq & \mathbb{E}\left[|\hat{\eta}(\mathbf{X})-\eta(\mathbf{X})|\mathbf{1}\left(C_0\frac{k}{N}<f(\mathbf{X})<N^{-\delta}\right)\right],
\end{eqnarray}
in which the inequality holds because $g(\mathbf{x})=\text{sign}(\hat{\eta}(\mathbf{x}))$ and $g^*(\mathbf{x})=\text{sign}(\eta(\mathbf{x}))$.

Define $r_N(\mathbf{x})$ as:
\begin{eqnarray}
r_N(\mathbf{x})=\left(\frac{2k}{NC_dv_df(\mathbf{x})}\right)^\frac{1}{d}.
\end{eqnarray}
In $S_3$, $f(\mathbf{x})>C_0k/N$, thus it can be shown that $r_N(\mathbf{x})\leq D$ always holds if $\mathbf{x} \in S_3$. Then according to Assumption \ref{ass:basic}(d),
$\text{P}(B(\mathbf{x}.r_N(\mathbf{x})))\geq C_df(\mathbf{x})v_dr_N^d(\mathbf{x})=2k/N$,
and
\begin{eqnarray}
\text{P}(\rho>r_N(\mathbf{x}))\leq e^{-N\text{P}(B(\mathbf{x},r_N(\mathbf{x})))}\left(\frac{eN\text{P}(B(\mathbf{x},r_N(\mathbf{x})))}{k}\right)^k\leq e^{-(1-\ln 2)k}.
\end{eqnarray}
To give a bound of $I_3$, note that
\begin{eqnarray}
\mathbb{E}[|\hat{\eta}(\mathbf{x})-\eta(\mathbf{x})|]&\leq& \sqrt{\mathbb{E}[(\hat{\eta}(\mathbf{x})-\eta(\mathbf{x}))^2]}\nonumber\\
&=&\sqrt{\Var[\hat{\eta}(\mathbf{x})]+(\mathbb{E}[\hat{\eta}(\mathbf{x})-\eta(\mathbf{x})])^2}\nonumber\\
&\leq & \sqrt{\Var[\hat{\eta}(\mathbf{x})]}+|\mathbb{E}[\hat{\eta}(\mathbf{x})]-\eta(\mathbf{x})|.
\label{eq:I3bound}
\end{eqnarray}
For the first term in \eqref{eq:I3bound}, define $U_i$ as a random variable drawn from $f(\cdot |\mathbf{X}\in B(\mathbf{x},\rho))$, for $i=1,\ldots, k$. $U_1, \ldots, U_k$ are conditionally i.i.d given $\rho$. Then
\begin{eqnarray}
&&\hspace{-8mm}\Var\left[\left.\frac{1}{k}\sum_{i=1}^k Y^{(i)}\right|\rho\right]\nonumber\\
&\overset{(a)}{=}&\Var\left[\frac{1}{k}\mathbb{E}\left[\left.\sum_{i=1}^k Y^{(i)} \right|\rho,\mathbf{X}^{(1)},\ldots,\mathbf{X}^{(N)}\right]\right]+\mathbb{E}\left[\Var\left[\left.\frac{1}{k}\sum_{i=1}^k Y^{(i)}\right|\rho,\mathbf{X}^{(1)},\ldots,\mathbf{X}^{(N)}\right]\right]\nonumber\\
&\overset{(b)}{\leq}&\Var\left[\left.\frac{1}{k}\sum_{i=1}^k \eta(\mathbf{X}^{(i)})\right|\rho\right]+\frac{1}{k}\nonumber\\
&=&\Var\left[\left.\frac{1}{k}\sum_{i=1}^k \eta(U_i)\right|\rho\right]+\frac{1}{k}\nonumber\\
&\overset{(c)}{=}&\frac{1}{k}\Var[\eta(U_1)|\rho]+\frac{1}{k}\nonumber\\
&\leq&\frac{2}{k},
\label{eq:variance}
\end{eqnarray}
in which (a) uses the total law of variance. In (b), note that $Y^{(i)}$ are conditionally independent given $\rho$ and the position of testing point and all training samples, and the conditional variance of $Y^{(i)}$ is no more than $1$. (c) uses the fact that $U_1,\ldots, U_k$ are conditionally i.i.d given $\rho$.

For the second term in \eqref{eq:I3bound}, 
\begin{eqnarray}
&&\hspace{-6mm}|\mathbb{E}[\hat{\eta}(\mathbf{x})]-\eta(\mathbf{x})|\nonumber\\&\leq& \text{P}(\rho>r_N(\mathbf{x}))\left|\mathbb{E}[\hat{\eta}(\mathbf{x})|\rho>r_N(\mathbf{x})]-\eta(\mathbf{x})\right|+\text{P}(\rho\leq r_N(\mathbf{x}))\left|\mathbb{E}[\hat{\eta}(\mathbf{x})|\rho\leq r_N(\mathbf{x})]-\eta(\mathbf{x})\right|\nonumber\\
&\leq & 2\text{P}(\rho>r_N(\mathbf{x}))+\left|\mathbb{E}[\eta(B(\mathbf{x},\rho))|\rho\leq r_N(\mathbf{x})]-\eta(\mathbf{x})\right|\nonumber\\
&\leq & 2e^{-k(1-\ln 2)}+C_c \left(\frac{2k}{NC_dv_df(\mathbf{x})}\right)^\frac{2}{d}.
\end{eqnarray}
Therefore, using Lemma \ref{lem:tail} in Appendix~\ref{sec:lemmas}, we have
\begin{eqnarray}
I_3&=&\int_{S_3} \mathbb{E}|\hat{\eta}(\mathbf{x})-\eta(\mathbf{x})|f(\mathbf{x})d\mathbf{x}\nonumber\\
&\leq & \int_{S_3} \left[\sqrt{\frac{2}{k}}+2e^{-k(1-\ln 2)}+C_c \left(\frac{2k}{NC_dv_df(\mathbf{x})}\right)^\frac{2}{d}\right]f(\mathbf{x})d\mathbf{x}\nonumber\\
&=&\left\{
\begin{array}{ccc}
\mathcal{O}\left(k^{-\frac{1}{2}}N^{-\beta\delta}\right) +\mathcal{O}\left(\left(\frac{k}{N}\right)^\beta\right)&\text{if} & \beta<\frac{2}{d}\\
\mathcal{O}\left(k^{-\frac{1}{2}}N^{-\beta\delta}\right) +\mathcal{O}\left(\left(\frac{k}{N}\right)^\frac{2}{d}\ln N\right)&\text{if} & \beta=\frac{2}{d}\\
\mathcal{O}\left(k^{-\frac{1}{2}}N^{-\beta\delta}\right) + \mathcal{O}\left(N^{-\beta\delta}(kN^{\delta-1})^\frac{2}{d}\right) &\text{if} & \beta>\frac{2}{d}.
\end{array}
\right.
\label{eq:I3final}
\end{eqnarray}

Combine $I_1$, $I_2$, $I_3$ and $I_4$, the excess risk can be expressed as
\begin{eqnarray}
R-R^*= \mathcal{O}\left(k^{-\frac{\alpha+1}{2}}\right)+\mathcal{O}\left(\Delta^{\alpha+1}\right)+\mathcal{O}\left(\left(\frac{k}{N}\right)^\beta\right)+I_3,
\label{eq:excess}
\end{eqnarray} 
in which $I_3$ is expressed in \eqref{eq:I3final}. Moreover, according to \eqref{eq:an}, \eqref{eq:delta}, we have $\Delta \sim (kN^{\delta-1})^{2/d}$.

Adjust $\delta$ as well as the growth rate of $k$ over $N$, we get the following results.

The optimal growth rate of $k$ is
\begin{eqnarray}
k\sim \left\{
\begin{array}{ccc}
N^{\frac{2\beta}{2\beta+\alpha+1}} &\text{if} &\beta\leq \frac{2}{d}\\
N^\frac{4\beta}{2\alpha+\beta(d+4)} &\text{if} &\beta>\frac{2}{d}
\end{array}.
\right.
\end{eqnarray}
The corresponding convergence rate is:
\begin{eqnarray}
R-R^*=\left\{
\begin{array}{ccc}
\mathcal{O}\left(N^{-\frac{\beta(\alpha+1)}{2\beta+\alpha+1}}\right) &\text{if} &\beta<\frac{2}{d}\\
\mathcal{O}\left(N^{-\frac{\beta(\alpha+1)}{2\beta+\alpha+1}}\ln N\right) &\text{if} &\beta=\frac{2}{d}\\
\mathcal{O}\left(N^{-\frac{2\beta(\alpha+1)}{\beta d+2(\alpha+2\beta)}}\right) &\text{if} &\beta>\frac{2}{d}.
\end{array}
\right.
\end{eqnarray}
The proof of an upper bound of the excess risk of the standard kNN classification is complete.
\subsection{Lower Bound}\label{sec:original-lb}
We prove the following statements separately:
\begin{eqnarray}
\underset{(f,\eta)\in \mathcal{S}}{\sup} (R-R^*)&\gtrsim& k^{-\frac{1+\alpha}{2}};
\label{eq:clslb1}\\
\underset{(f,\eta) \in \mathcal{S}}{\sup} (R-R^*)&\gtrsim& \left(\frac{k}{N}\right)^\beta;
\label{eq:clslb2}\\
\underset{(f,\eta)\in \mathcal{S}}{\sup} (R-R^*)&\gtrsim&\underset{0\leq \delta\leq 1}{\sup} \min\left\{ N^{-\beta\delta} (kN^{\delta-1})^\frac{2}{d}, (kN^{\delta-1})^\frac{2(\alpha+1)}{d} \right\}.
\label{eq:clslb3}
\end{eqnarray}

\textbf{Proof of \eqref{eq:clslb1}}.
Let $\mathbf{X}$ be uniformly distributed in $A\cup B$, and let $\eta(\mathbf{x})=a>0$ for all $\mathbf{x}\in A$, $\eta(\mathbf{x})=1$ for all $\mathbf{x}\in B$, in which $A$ and $B$ are two disjoint sets.

Then for any $\mathbf{x}\in A$,
\begin{eqnarray}
\text{P}\left(g(\mathbf{x})\neq g^*(\mathbf{x})\right)=\text{P}(g(\mathbf{x})=-1)=\text{P}\left(\frac{1}{k} \sum_{i=1}^k Y^{(i)}<0\right).
\end{eqnarray}
If $a\sim 1/\sqrt{k}$, then $\text{P}(g(\mathbf{x})\neq g^*(\mathbf{x}))\rightarrow c>0$.

Note that according to Assumption \ref{ass:basic}(a), $\text{P}(|\eta(\mathbf{X})|\leq a)\leq C_aa^\alpha$. Thus $\text{P}(A)\leq C_a a^\alpha$. Now we set $a\sim 1/\sqrt{k}$, and let $\text{P}(A)=C_a a^\alpha$, then
\begin{eqnarray}
R-R^*=\mathbb{E}[\mathbf{1}(g(\mathbf{X})\neq g^*(\mathbf{X}))|\eta(\mathbf{X})|]\geq  a\text{P}(A)\text{P}(g(\mathbf{X})\neq g^*(\mathbf{X}))\sim  a^{1+\alpha}.
\label{eq:clslb1final}
\end{eqnarray}
Substitute $a$ in \eqref{eq:clslb1final} with $1/\sqrt{k}$, the proof of \eqref{eq:clslb1} is complete. 

\textbf{Proof of \eqref{eq:clslb2}}.
Construct $(n+1)$ cubes $I_1, \ldots, I_{n+1}$:
\begin{eqnarray}
I_j=\left\{ \mathbf{x}|4j-1<x_1<4j+1, x_2,\ldots,x_d\in [-1,1] \right\}.
\end{eqnarray}
Let random variable $\mathbf{X}$ be supported in these cubes. Within each cube, the distribution is uniform. Let $m$ be the pdf value in the first $n$ cubes. $n$ and $m$ will change with $k$ and $N$. For the $(n+1)$-th cube, the pdf should be $(1-2^d nm)/2^d$, so that the total probability mass of all $(n+1)$ cubes is 1. For any $k$ and $N$, let $m=k/(3\times 2^dN)$.

Let $\eta(\mathbf{x})=(-1)^j$ for $\mathbf{x}\in I_j$. For $\mathbf{x}\notin \cup_{j=1}^{n+1} I_j$, $\eta(\mathbf{x})$ can be set arbitrarily as long as it satisfies Assumption \ref{ass:basic}(c) with constant $C_c$. It can be shown that for $j=3,\ldots, n-2$, if $2k/3N<\text{P}(B(\mathbf{x},\rho))<4k/3N$, then $B(\mathbf{x},\rho)$ contains more than $2$ and less than $4$ cubes among $I_1,\ldots, I_n$. In this case, the average value of $\eta$ in $B(\mathbf{x},\rho)$, i.e. $\eta(B(\mathbf{x},\rho))$, has opposite sign with $\eta(\mathbf{x})$. As a result, if for a specific test point $\mathbf{x}$, $2k/3N<\text{P}(B(\mathbf{x},\rho))<4k/3N$,
$\text{P}(g(\mathbf{x})\neq g^*(\mathbf{x})|\text{P}(B(\mathbf{x},\rho))=\text{P}(\text{sign}(\hat{\eta}(\mathbf{x})\neq \text{sign}(\eta(\mathbf{x})))>1/2$.
Hence,
\begin{eqnarray}
\text{P}(g(\mathbf{x})\neq g^*(\mathbf{x}))\geq \frac{1}{2}\text{P}\left(\frac{2k}{3N}<\text{P}(B(\mathbf{x},\rho))<\frac{4k}{3N}\right).
\end{eqnarray}
It can be shown that $\text{P}(2k/3N<\text{P}(B(\mathbf{x},\rho))<4k/3N)\rightarrow 1$ as $k\rightarrow \infty$. Thus we have $\text{P}(g(\mathbf{x})\neq g^*(\mathbf{x}))\rightarrow 1/2$ for all $\mathbf{x}\in I_3,\ldots, I_{n-2}$. Then
\begin{eqnarray}
R-R^*=\mathbb{E}[\mathbf{1}(g(\mathbf{X})\neq g^*(\mathbf{X}))|\eta(\mathbf{X})|]\gtrsim  \text{P}(\mathbf{X}\in I_3\cup \ldots \cup I_{n-2})\sim  nm.
\end{eqnarray}
Assumption (b) requires that $\text{P}(f(\mathbf{X})\leq m)\leq C_bm^\beta$, thus
\begin{eqnarray}
R-R^*\gtrsim m^\beta \sim \left(\frac{k}{N}\right)^\beta.
\end{eqnarray}
The proof of \eqref{eq:clslb2} is complete.

\textbf{Proof of \eqref{eq:clslb3}}.
Construct $(n+1)$ cubes in similar way as the previous step, i.e. the proof of \eqref{eq:clslb2}. However, now we use adaptive cube size. Let
\begin{eqnarray}
L=\frac{1}{2}\left(kN^{\delta-1}\right)^\frac{1}{d}
\end{eqnarray}
for some $0\leq \delta\leq 1$. Then construct $(n+1)$ cubes $I_1, \ldots, I_{n+1}$ as following:
\begin{eqnarray}
I_j=\left\{\mathbf{x}|(4j-1)L<x_1<(4j+1)L, x_2,\ldots, x_d\in [-L,L] \right\}.
\end{eqnarray}
Let the distribution be uniform within each cube, and the pdf in $I_1, \ldots, I_n$ are the same and denoted as $m$. Similar to the proof of \eqref{eq:clslb2}, $m$, $n$ change with $k$ and $N$. Here we let $m=(1/3)N^{-\delta}$. Then $\text{P}(\mathbf{X}\in I_j)=2^d mL^d=k/(3N)$ for $j=1,\ldots, n$. Moreover, let
\begin{eqnarray}
\eta(\mathbf{x})=\frac{1}{4}(-1)^j \left(kN^{\delta-1}\right)^\frac{2}{d}
\end{eqnarray}
for $\mathbf{x}\in I_j$. For $\mathbf{x}\notin I_j$, $\eta(\mathbf{x})$ need to be set to satisfy Assumption \ref{ass:basic}(c) with constant $C_c$. To show that such $\eta$ exists, define $\eta_0(\mathbf{x})$ as the $\eta$ constructed in the previous step, i.e. proof of \eqref{eq:clslb2}. Then let $\eta(\mathbf{x})=L^2\eta_0(\mathbf{x}/L)$, then as long as $|\eta_0(B(\mathbf{x},r))-\eta_0(\mathbf{x})|\leq C_cr^2$ for any $\mathbf{x}$ and $r>0$, $|\eta(B(\mathbf{x},r))-\eta(\mathbf{x})|\leq C_cr^2$ also holds for any $\mathbf{x}$ and $r>0$. 

The remaining proof is similar to the proof of \eqref{eq:clslb2}. For $\mathbf{x}\in I_j$, $j=3,\ldots,n-2$,
\begin{eqnarray}
\text{P}(g(\mathbf{x})\neq g^*(\mathbf{x}))>\frac{1}{2}\text{P}\left(\frac{2k}{3N}<\text{P}(B(\mathbf{x},\rho))<\frac{4k}{3N}\right).
\end{eqnarray}
Then
\begin{eqnarray}
R-R^*=\mathbb{E}[\mathbf{1}(g(\mathbf{X})\neq g^*(\mathbf{X}))|\eta(\mathbf{X})|]\gtrsim \text{P}(\mathbf{X}\in I_3\cup \ldots \cup I_{n-2})(kN^{\delta-1})^\frac{2}{d}\sim  nm\left(kN^{\delta-1}\right)^\frac{2}{d}.
\end{eqnarray}
According to Assumptions \ref{ass:basic} (a) and (b),
$\text{P}\left(|\eta(\mathbf{X}|\leq\left(kN^{\delta-1}\right)^\frac{2}{d}\right)\leq C_a\left(kN^{\delta-1}\right)^\frac{2\alpha}{d}$, and $\text{P}(f(\mathbf{X})\leq m)\leq C_bm^\beta$,
thus $nm\lesssim \min\{ (kN^{\delta-1}), m^\beta\}$. Since we have already set $m=(1/3)N^{-\delta}$, 
\begin{eqnarray}
R-R^*\gtrsim \min\left\{N^{-\beta\delta} \left(kN^{\delta-1}\right)^\frac{2}{d}, \left(kN^{\delta-1}\right)^\frac{2(\alpha+1)}{d}\right\}.
\end{eqnarray}
The above equation holds for any $0\leq \delta\leq 1$. Hence \eqref{eq:clslb3} holds.

Based on \eqref{eq:clslb1}, \eqref{eq:clslb2}, \eqref{eq:clslb3}, we get
\begin{eqnarray}
\underset{k}{\inf}\underset{(f,\eta)\in \mathcal{S}}{\sup} (R-R^*)=\Omega\left(N^{-\min\left\{ \frac{\beta(\alpha+1)}{2\beta+\alpha+1},\frac{2\beta(\alpha+1)}{\beta d+2(\alpha+2\beta)} \right\}}\right).
\end{eqnarray}
\color{black}
\section{Proof of Theorem \ref{thm:minimax}: Minimax convergence rate of classification}\label{sec:minimax}
The minimax lower bound of the convergence rate is defined as $\underset{g}{\inf}\underset{(f,\eta)\in \mathcal{S}}{\sup} (R-R^*)$. To obtain this bound, a common approach is to find an appropriate finite subset of $\mathcal{S}$, so that the problem can be reduced to a hypothesis testing problem. The minimax rate among this subset can then be used as a lower bound of the minimax rate over the whole class $\mathcal{S}$. A detailed introduction of this type of method can be found in \cite{tsybakov2009introduction}.

In our proof, the construction of the finite subset $\mathcal{S}^*$ of $\mathcal{S}$ is based on Assouad cube method, which has been used in \cite{audibert2007fast} and \cite{gadat2016classification}.

Let $f(\mathbf{x})$ be supported on $(n_0+1)$ balls, in which $n_0$ will be determined later:
\begin{eqnarray}
f(\mathbf{x})=m\sum_{j=1}^{n_0} \mathbf{1}(\mathbf{x}\in B(\mathbf{a}_j,L))+m_0\mathbf{1}(\mathbf{x}\in B(\mathbf{a}_0,L_0)),
\label{eq:pdf}
\end{eqnarray}
in which $L$ and $L_0$ depend on the sample size $N$. $m$, $m_0$ are also two parameters that need to be determined later. 


We require that $\norm{\mathbf{a}_i-\mathbf{a}_j}>r_0$ for all $i\neq j$, in which
\begin{eqnarray}
r_0=2\max\left\{\left(\frac{1}{C_c}\right)^\frac{1}{2},1 \right\}L,
\label{eq:r0}
\end{eqnarray}
in which $C_c$ is the constant in Assumption 1 (c). We arrange $\mathbf{a}_0, \ldots, \mathbf{a}_{n_0}$ in the following way. Define
\begin{eqnarray}
r_M=\inf\{r|\mathcal{P}(\bar{B}(\mathbf{0},r),r_0)\geq n_0 \},
\label{eq:L0}
\end{eqnarray}
in which $\bar{B}(\mathbf{0},r)$ is the closure of $B(\mathbf{0},r)$, and $\mathcal{P}$ denotes the packing number. Since $\bar{B}(\mathbf{0},r)$ is a closed set, we know that the packing number should be right continuous in $r$. As a result, we have $\mathcal{P}(\bar{B}(\mathbf{0},r_M),r_0)=n_0$. We can then pick $\mathbf{a}_1,\ldots,\mathbf{a}_{n_0}$, so that the pairwise distances between them are no less than $r_0$. Besides, we pick $\mathbf{a}_0$ such that $\norm{\mathbf{a}_0}>r_M+L_0$. Under this condition, $B(\mathbf{a}_0,L_0)$ does not intersect with any other $n_0$ balls $B(\mathbf{a}_j,r_0)$.

We also let $\mathbf{a}_0$ to be sufficiently far away from $\mathbf{a}_j, j=1,\ldots,n_0$. Furthermore, define 
\begin{eqnarray}
\eta_\mathbf{v}(\mathbf{x})=\sum_{j=1}^{n_0} \mathbf{v}(j) L^2\mathbf{1}(\mathbf{x}\in B(\mathbf{a}_j,L)),
\label{eq:etav}
\end{eqnarray}
in which $\mathbf{v}\in \{-1,1\}^{n_0}$. To ensure that \eqref{eq:pdf} is a normalized pdf, we have the following constraints:
\begin{eqnarray}
n_0 mv_dL^d+m_0v_dL_0^d=1,
\end{eqnarray}
in which, as defined before, $v_d$ is the volume of the unit radius ball.


Recall that $\mathcal{S}$ is the set of all pdfs and regression functions that satisfy Assumption 1 (a)-(d). We have the following lemma:
\begin{lem}\label{lem:restriction}
	$(f, \eta_\mathbf{v})$ satisfies Assumption 1(a)-(d) for $\forall \mathbf{v}\in \{-1,1\}^{n_0}$ if: (1) $n_0 mv_dL^d\leq C_aL^{2\alpha}$; (2) $n_0 mv_dL^d\leq C_bm^\beta;$ (3) $L\leq 1$.
\end{lem}
\begin{proof}
	Please see Appendix \ref{sec:restriction} for proof.
\end{proof}
Define
\begin{eqnarray}
\mathcal{S}^*=\{(f,\eta_\mathbf{v})|\mathbf{v}\in\{-1,1\}^{n_0} \},
\label{eq:Sstar}
\end{eqnarray}
where $f$ and $\eta_\mathbf{v}$ satisfy the requirements in Lemma \ref{lem:restriction}, then $\mathcal{S}^*\subset \mathcal{S}$. For an arbitrary classifier $g$,
\begin{eqnarray}
\underset{(f,\eta)\in \mathcal{S}}{\sup} (R-R^*)\geq \underset{(f,\eta)\in \mathcal{S}^*}{\sup} (R-R^*).
\label{eq:minoration}
\end{eqnarray}
To bound the right hand side of \eqref{eq:minoration}, we use the following lemma.
\begin{lem}\label{lem:assouad}
	(Modified from \cite{audibert2004classification}, Lemma 5.1)
	\begin{eqnarray}
	\underset{(f,\eta)\in \mathcal{S^*}}{\sup} (R-R^*)\geq \frac{1-L^2\sqrt{N\omega}}{2} n_0\omega L^2,
	\end{eqnarray}
	in which $\omega$ is the probability mass of $B(\mathbf{a}_j,L)$ for $j=1,\ldots,n_0$:
	\begin{eqnarray}
	\omega=\text{P}(B(\mathbf{a}_j,L))=mv_dL^d.
	\label{eq:omega}
	\end{eqnarray}
\end{lem}
\begin{proof}
	Lemma \ref{lem:assouad} is similar to the Assouad lemma for classification (\cite{audibert2004classification}, Lemma 5.1), except that some details are different. In Appendix \ref{sec:assouad}, we provide a simplified proof.
\end{proof}
Therefore, according to Lemma \ref{lem:assouad},
\begin{eqnarray}
\underset{(f,\eta)\in \mathcal{S}^*}{\sup} (R-R^*)&\geq& \frac{1}{2}(1-L^2 \sqrt{N\omega})n_0 \omega L^2
\gtrsim  (1-v_d m^{\frac{1}{2}} L^{\frac{d}{2}+2}N^{\frac{1}{2}})n_0 mL^{d+2},
\label{eq:sup}
\end{eqnarray}
in which the second step comes from \eqref{eq:omega}.

We then select a proper rule to let $m$, $n_0$, $L$ to vary with $N$. From Lemma \ref{lem:restriction}, we get the following bounds:
\begin{eqnarray}
mn_0L^d&=&\mathcal{O}(L^{2\alpha }),\label{eq:lb1}\\
mn_0L^d&=&\mathcal{O}(m^\beta),\\
L&=&\mathcal{O}(1).\label{eq:lb3}
\end{eqnarray}

In addition, we need to ensure that in the right hand side of \eqref{eq:sup}, the expression in the bracket is larger than a positive constant, i.e. $1-v_dm^\frac{1}{2}L^{\frac{d}{2}+2}N^{\frac{1}{2}}>C>0$, then
\begin{eqnarray}
N mL^{4+d}=\mathcal{O}(1).
\end{eqnarray} 
Based on these constraints, we can get a lower bound of the minimax convergence rate. 

\noindent Construct 1: Let 
$L\sim N^{-\frac{\beta}{\beta d+2(\alpha+2\beta)}},$
and 
$m\sim N^{-\frac{2\alpha}{\beta d +2(\alpha+2\beta)}},$
then
\begin{eqnarray}
R-R^*\sim n_0 mL^{d+2}
\gtrsim N^{-\frac{2\beta(\alpha+1)}{\beta d+2(\alpha+2\beta)}}.
\end{eqnarray}

\noindent Construct 2: Let $L\sim 1$, $m\sim N^{-1}$, then
\begin{eqnarray}
R-R^*\gtrsim N^{-\beta}.
\end{eqnarray}
Combine these two bounds, we get
\begin{eqnarray}
\underset{(f,\eta)\in S^*}{\sup} (R-R^*) \gtrsim N^{-\min\left\{\frac{2\beta (\alpha+1)}{\beta d+2(\alpha+2\beta)},\beta \right\}}.
\end{eqnarray}
We can check that when $\beta\leq 1$, both constructions (1) and (2) satisfy the conditions from \eqref{eq:lb1} to \eqref{eq:lb3}. The proof is complete.
\subsection{Proof of Lemma \ref{lem:restriction}}\label{sec:restriction}
In this section, we prove Lemma \ref{lem:restriction}. In particular, we prove that the Assumption 1 (a)-(d) are satisfied under conditions specified in the Lemma.

\textbf{For (a)}. According to condition (1), we have
\begin{eqnarray}
\text{P}(0<|\eta(\mathbf{X})|\leq t)=\left\{
\begin{array}{ccc}
0 &\text{if}& t< L^2\\
n_0 mv_dL^d &\text{if} & t\geq L^2.
\end{array}
\right.
\end{eqnarray}
If $n_0 mv_dL^d\leq C_aL^{2\alpha }$, then
$\text{P}(0<|\eta(\mathbf{X})|<t)\leq C_at^\alpha$.

\textbf{For (b)}. According to condition (2), we have
\begin{eqnarray}
\text{P}(f(\mathbf{X})<t)=\left\{
\begin{array}{ccc}
0 & \text{if} & t\leq m;\\
n_0 mv_dL^d &\text{if} & m<t\leq m_0;\\
1 &\text{if} & t>m_0.
\end{array}
\right.
\end{eqnarray}
If $n_0 mv_dL^d\leq C_bm^\beta$ and $m_0>C_b^{-\frac{1}{\beta}}$, then
$\text{P}(f(\mathbf{X})<t)\leq C_bt^{\beta}$.

\textbf{For (c)}. As Assumption 1 (c) holds only for $\mathbf{x}$ with $f(\mathbf{x})>0$, we only need to discuss the case where $B(\mathbf{x},r)\cap B(\mathbf{a}_j,L)\neq \varnothing$ for some $j$, or $B(\mathbf{x},r)\cap B(\mathbf{a}_0,L_0)\neq \varnothing$, i.e., among all $(n_0+1)$ balls, $B(\mathbf{x},r)$ intersects with at least one ball. 

To prove Assumption 1 (c), we discuss two cases:

Case 1: $r> r_0$. According to \eqref{eq:etav}, $|\eta(\mathbf{x})|\leq L^2$. Recall that $\eta(B(\mathbf{x},r))$ is the average of $\eta(\mathbf{x})$ in $B(\mathbf{x},r)$, therefore $|\eta(B(\mathbf{x},r))|\leq L^2$. If $r>r_0$, then
$C_c r^2 >C_c r_0^2\geq 2L^2$.
Therefore $|\eta(B(\mathbf{x},r))-\eta(\mathbf{x})|\leq C_c r^2$ holds.

Case 2: $r\leq r_0$. In this case, it is obvious that $B(\mathbf{x},r)$ intersects with at most one ball among the $(n_0+1)$ balls. Therefore the density is uniform, and $|\eta(B(\mathbf{x},r))-\eta(\mathbf{x})|=0$.


\textbf{For (d)}. Now we pick $D<\min\{r_M,L_0 \}$, and show that there exists a constant $C_d$, such that for any $\mathbf{x}$ with $f(\mathbf{x})>0$, $\text{P}(B(\mathbf{x},r))\geq f(\mathbf{x}) C_d v_d r^d$ for any $r<D$. Despite that quantities such as $r_0, L, L_0, m$ and $n_0$ change with the sample size $N$, to ensure that our derivation of minimax lower bound is effective, we must give a universal constant $C_d$, independent of sample size $N$. Obviously, $\underset{\mathbf{u}\in B(\mathbf{c},r)}{\inf} V(B(\mathbf{u},r)\cap B(\mathbf{c},r))/r^d$ is a constant for all $r$ and $\mathbf{c}$. We use $v_d'$ to define this constant.

We discuss two cases: 1) $\mathbf{x}\in B(\mathbf{a}_0, L_0)$; and 2) $\mathbf{x}\in B(\mathbf{a}_j,L)$ for some $j\in\{1,\ldots,n_0 \}$.

For the first case, recall that the choice of $D$ ensures that $L_0>D$. Furthermore, the density is uniform in $B(\mathbf{a}_0,L_0)$. Then for any $r<D$,
\begin{eqnarray}
\text{P}(B(\mathbf{x},r))&=&f(\mathbf{x})V(B(\mathbf{x},r)\cap B(\mathbf{a}_0,L_0))\nonumber\\
&\overset{(a)}{\geq} &f(\mathbf{x})\frac{r^d}{L_0^d}V(B(\mathbf{x},L_0)\cap B(\mathbf{a}_0,L_0))\nonumber\\
&\geq &f(\mathbf{x})\frac{r^d}{L_0^d} \underset{\mathbf{u}\in B(\mathbf{a}_0,L_0)}{\inf} V(B(\mathbf{u},L_0)\cap B(\mathbf{a}_0,L_0))\nonumber\\
&\geq &f(\mathbf{x})r^d v_d',
\label{eq:d1}
\end{eqnarray}
in which (a) holds because $r<D<L_0$, hence
\begin{eqnarray}
\frac{r}{L_0}[B(\mathbf{x},L_0)\cap B(\mathbf{a}_0,L_0)]=B(\mathbf{x},r)\cap B(\mathbf{a}_0,r)\subset B(\mathbf{x},r)\cap B(\mathbf{a}_0,L_0).
\end{eqnarray}

For the second case, if $f(\mathbf{x})>0$ and $r<4r_0$, then according to the definition of $r_0$ in \eqref{eq:r0}, $B(\mathbf{a}_j,L)\subset B(\mathbf{x},r_0)$ for some $j$. Hence 
\begin{eqnarray}
\text{P}(B(\mathbf{x},4r_0))\geq f(\mathbf{x})v_dL^d.
\end{eqnarray}
Then for $r<4r_0$,
\begin{eqnarray}
\text{P}(B(\mathbf{x},r))&\geq& \left(\frac{r}{4r_0}\right)^d \text{P}(B(\mathbf{x},4r_0))\nonumber\\
& \geq& f(\mathbf{x})v_d \frac{L^d}{4^dr_0^d}r^d\nonumber\\
&\geq& f(\mathbf{x})v_dr^d\frac{1}{8^d\left(\max\left\{\left(\frac{1}{C_c}\right)^\frac{1}{2},L \right\}\right)^d }.
\label{eq:d2}
\end{eqnarray}
If $r\geq 4r_0$, define $n'$:
\begin{eqnarray}
n'=\sum_{j=1}^{n_0} \mathbf{1}(B(\mathbf{a}_j,L)\subset B(\mathbf{x},r)).
\end{eqnarray}
Then
\begin{eqnarray}
n'\geq \text{P}(B(\mathbf{x},r-2r_0)\cap B(0,r_M)).
\label{eq:packlb}
\end{eqnarray}  
We prove \eqref{eq:packlb} by contradiction. Suppose \eqref{eq:packlb} is not true, then we can add at least one more ball with radius $r_0$ in $B(\mathbf{x},r-r_0)$. However, according to \eqref{eq:L0}, the $n_0$ balls $B(\mathbf{a}_j,r_0)$, $j=1,\ldots,n_0$ already form a maximum packing in $B(\mathbf{x},r_M)$. Therefore \eqref{eq:packlb} holds. Hence for all $r<r_M$,
\begin{eqnarray}
\text{P}(B(\mathbf{x},r))&=& f(\mathbf{x})v_d L^d n'\nonumber\\
&\geq &f(\mathbf{x})v_d L^d \text{P}(B(\mathbf{x},r-2r_0)\cap B(0,r_M))\nonumber\\
&\overset{(a)}{\geq}& f(\mathbf{x})v_dL^d \frac{V(B(\mathbf{x},r-2r_0)\cap B(0,r_M))}{V(B(0,r_0))}\nonumber\\
&\geq & f(\mathbf{x})v_dL^d \frac{V\left(B\left(0,\frac{1}{2}r\right)\right)\cap B(-\mathbf{x},r_M))}{v_dr_0^d}\nonumber\\
&\overset{(b)}{\geq} & f(\mathbf{x}) \frac{L^d}{r_0^d} \frac{r^d}{2^dr_M^d}\underset{\mathbf{u}\in B(0,r_M)}{\inf} V(B(0,r_M)\cap B(\mathbf{u},r_M))\nonumber\\
&=&f(\mathbf{x})r^d v_d'\frac{1}{2^{2d}\left(\max\left\{\left(\frac{1}{C_c}\right)^\frac{1}{2},L \right\}\right)^d },
\label{eq:d3}
\end{eqnarray}
in which (a) uses the lower bound of packing number \cite{vershynin2018high}. For (b), recall that for case 2, $\mathbf{x}\in B(\mathbf{a}_j,L)\subset B(0,r_M)$, hence $\norm{\mathbf{x}}<r_M$. Therefore
\begin{eqnarray}
\underset{\mathbf{u}\in B(0,r_M)}{\inf}\frac{r^d}{2^dr_M^d}V(B(0,r_M)\cap B(\mathbf{u},r_M))&=&\underset{\mathbf{u}\in B(0,r/2)}{\inf}V\left(B\left(0,r/2\right)\cap B\left(\mathbf{u},r/2\right)\right)\nonumber\\
&\leq &\underset{\mathbf{u}\in B(0,r_M)}{\inf}V\left(B\left(0,r/2\right)\cap B\left(\mathbf{u},r_M\right)\right)\nonumber\\
&\leq &V\left(B\left(0,r/2\right)\cap B\left(-\mathbf{x},r_M\right)\right).
\end{eqnarray}
\eqref{eq:d1}, \eqref{eq:d2} and \eqref{eq:d3} show that there exists an universal constant $C_d$ so that Assumption 1 (d) is satisfied.

Now we have shown that Assumption 1 (a)-(d) are all satisfied, hence the proof of Lemma \ref{lem:restriction} is complete.
\subsection{Proof of Lemma \ref{lem:assouad}}\label{sec:assouad}
Our proof is similar to the proof of Lemma 5.1 in \cite{audibert2004classification}. To begin with, we give a bound of the excess risk at a specific point $\mathbf{x}\in B(\mathbf{a}_1,L)$. Define $\text{P}_\mathbf{v}(\cdot)$ as the probability under $\eta=\eta_\mathbf{v}$. Denote $N_1$ as the number of training samples that falls in $B(\mathbf{a}_1,L)$. $N_1$ follows Binomial distribution $\text{Binom}(N,\omega)$, in which $\omega$ is defined in \eqref{eq:omega}, and Binom denotes Binomial distribution. Then
\begin{eqnarray}
&&\underset{(f,\eta)\in S^*}{\sup} (\text{P}(g(\mathbf{x})\neq Y)-\text{P}(g^*(\mathbf{x})\neq Y))\nonumber\\
&\overset{(a)}{\geq}&\underset{\mathbf{v}(1)\in \{-1,1\}, \mathbf{v}(2)=\ldots=\mathbf{v}(n_0)=0}{\sup} (\text{P}_\mathbf{v}(g(\mathbf{x})\neq Y)-\text{P}_\mathbf{v}(g^*(\mathbf{x})\neq Y))\nonumber\\
&\overset{(b)}{=}&\underset{\mathbf{v}(1)\in \{-1,1\}, \mathbf{v}(2)=\ldots=\mathbf{v}(n_0)=0}{\sup} L^2 \text{P}_\mathbf{v}(g(\mathbf{x})\neq \mathbf{v}(1))\nonumber\\
&=&\underset{\mathbf{v}(1)\in \{-1,1\}, \mathbf{v}(2)=\ldots=\mathbf{v}(n_0)=0}{\sup} L^2 \mathbb{E}[\text{P}_\mathbf{v}(g(\mathbf{x})\neq \mathbf{v}(1)|N_1)]\nonumber\\
&\overset{(c)}{\geq}& \underset{\mathbf{v}(1)\in \{-1,1\}, \mathbf{v}(2)=\ldots=\mathbf{v}(n_0)=0}{\sup} L^2 \mathbb{E}\left[\frac{1-TV\left(\text{Binom}\left(N_1,\frac{1-L^2}{2}\right),\text{Binom}\left(N_1,\frac{1+L^2}{2}\right)\right)}{2}\right].
\label{eq:specificx}
\end{eqnarray}
Here, (a) holds because $\mathbf{v}(1)\in \{-1,1\}, \mathbf{v}(2)=\ldots=\mathbf{v}(n_0)=0$ is more restrictive than $S^*$ defined in \eqref{eq:Sstar}. (b) comes from \eqref{eq:exrisk}. (c) gives a lower bound of the error probability of binary hypothesis testing problem \cite{tsybakov2009introduction}, in which $TV$ denotes the total variation distance between two distributions. The total variation distance between two Binomial distributions is bounded by
\begin{eqnarray}
TV(\text{Binom}(N,p_1),\text{Binom}(N,p_2))&\leq& H(B(N,p_1),B(N,p_2))\nonumber\\
&=&\sqrt{2\left[1-\left(1-\frac{H^2(\text{Bern}(p),\text{Bern}(1-p))}{2}\right)^{N}\right]},
\label{eq:TVbound}
\end{eqnarray}
in which $H$ is the Hellinger distance and Bern denotes Bernoulli distribution. Then we use \eqref{eq:TVbound} to bound the total variation distance:
\begin{eqnarray}
TV\left(\text{Binom}\left(N_1,\frac{1-L^2}{2}\right),\text{Binom}\left(N_1,\frac{1+L^2}{2} \right)\right)\leq \sqrt{2\left[1-(\sqrt{1-L^{4}})^{N_1}\right]}\leq \sqrt{N_1}L^2.
\label{eq:TVbound2}
\end{eqnarray}
Plugging \eqref{eq:TVbound2} into \eqref{eq:specificx} and considering that $\mathbb{E}[\sqrt{N_1}]\leq \sqrt{\mathbb{E}[N_1]}=\sqrt{N\omega}$, we have
\begin{eqnarray}
\underset{(f,\eta)\in S^*}{\sup} (\text{P}(g(\mathbf{x})\neq Y)-\text{P}(g^*(\mathbf{x})\neq Y))\geq L^2 \frac{1-\sqrt{N\omega}L^2}{2}.
\end{eqnarray}
For $\mathbf{x}\in B(\mathbf{a}_j,L)$ for $j=2,\ldots,n_0$, we can obtain the same bound. Hence
\begin{eqnarray}
\underset{(f,\eta)\in S^*}{\sup} (R-R^*)&=&\underset{(f,\eta)\in S^*}{\sup} (\text{P}(g(\mathbf{X})\neq Y)-\text{P}(g^*(\mathbf{X})\neq Y))\nonumber\\
&\geq & \sum_{j=1}^{n_0} \text{P}(\mathbf{X}\in B(\mathbf{a}_j,L))L^2 \frac{1-\sqrt{N\omega}L^2}{2}\nonumber\\
&=&n_0 \omega L^2 \frac{1-\sqrt{N\omega}L^2}{2}.
\end{eqnarray}
The proof of Lemma \ref{lem:assouad} is complete.
\section{Proof of Theorem \ref{thm:adaptive}: Convergence rate of the adaptive kNN classification}\label{sec:adaptive}
In this section, we prove the convergence rate of the adaptive kNN classifier. To begin with, define
\begin{eqnarray}
h(\mathbf{x})=\frac{f^{\frac{1}{1-q}}(\mathbf{x})}{\text{P}^\frac{q}{1-q}(B(\mathbf{x},A))}.
\label{eq:hdf}
\end{eqnarray}
 In this section, without loss of generality, we will assume $D\geq A$, as the assumption $D\geq A$ does not impose further restrictions on the distribution of $\mathbf{X}$. This can be seen from the fact that, if $D<A$, 
$\text{P}(B(\mathbf{x},r))\geq C_df(\mathbf{x})V(B(\mathbf{x},D))\geq C_d\left(D/A\right)^d f(\mathbf{x})V(B(\mathbf{x},r))$
for $D\leq r\leq A$, we can use $C_d(D/A)^d$ to replace $C_d$, and use $A$ to replace $D$. 

According to Assumption \ref{ass:basic}(d), the following relation holds between $\text{P}(B(\mathbf{x},A))$, $f(\mathbf{x})$ and $h(\mathbf{x})$:
\begin{eqnarray}
\text{P}(B(\mathbf{x},A))\geq C_dv_dA^d f(\mathbf{x})\geq (C_dv_dA^d)^\frac{1}{1-q}h(\mathbf{x}).
\label{eq:relation}
\end{eqnarray}

Moreover, According to Assumption 2,
\begin{eqnarray}
\text{P}(h(\mathbf{X})<t)&=&\text{P}\left(\frac{f(\mathbf{X})}{\text{P}^q(B(\mathbf{X},A))}<t^{1-q}\right)\leq C_b't^\beta.
\label{eq:smallh}
\end{eqnarray}

\eqref{eq:relation} and \eqref{eq:smallh} will be used frequently in the proof. Now we divide the support into four regions:
\begin{eqnarray}
S_1&=&\{\mathbf{x}|h(\mathbf{x})\geq N^{-\delta}, |\eta(\mathbf{x})|>2\Delta \},\label{eq:s1def}\\
S_2&=&\{\mathbf{x}|h(\mathbf{x})\geq N^{-\delta}, |\eta(\mathbf{x})|\leq 2\Delta \},\label{eq:s2def}\\
S_3&=&\{\mathbf{x}|C_0 N^{-1}<h(\mathbf{x})<N^{-\delta} \},\label{eq:s3def}\\
S_4&=&\{\mathbf{x}|h(\mathbf{x})\leq C_0 N^{-1} \},\label{eq:s4def}
\end{eqnarray}
in which $0<\delta<1$. $\delta$ and $\Delta$ will be determined later, and
\begin{eqnarray}
C_0=2(K+1)^{\frac{1}{1-q}} (C_dv_dA^d)^{-\frac{1}{1-q}}.
\label{eq:c0}
\end{eqnarray}
Recall that
\begin{eqnarray}
R-R^*=\mathbb{E}[\mathbf{1}(g(\mathbf{X})\neq g^*(\mathbf{X}))|\eta(\mathbf{X})|].
\end{eqnarray}
Define
\begin{eqnarray}
I_i=\mathbb{E}[\mathbf{1}(g(\mathbf{X})\neq g^*(\mathbf{X}))|\eta(\mathbf{X})|\mathbf{1}(\mathbf{X}\in S_i)],
\label{eq:Idef}
\end{eqnarray}
for $i=1,2,3,4$. Then we have
\begin{eqnarray}
R-R^*=\sum_{i=1}^4 I_i.
\label{eq:R}
\end{eqnarray}
Now we bound $I_1$, $I_2$, $I_3$ and $I_4$ separately.

\noindent\textbf{Bound of $I_1$}.
Define $\rho$ as the distance from test point $\mathbf{x}$ to its $(k+1)$-th nearest neighbor. In addition, define $r_n(\mathbf{x})$ as
\begin{eqnarray}
r_n(\mathbf{x})=\inf\left\{r\left|  \frac{\text{P}(B(\mathbf{x},r))}{\text{P}(B(\mathbf{x},A))}=\frac{2k+2}{n}\right. \right\}.
\label{eq:rndef}
\end{eqnarray}
If the density is positive everywhere, then the distance $r$ that satisfies $\text{P}(B(\mathbf{x},r))/\text{P}(B(\mathbf{x},A))=(2k+2)/n$ is unique. Otherwise $r$ may not be unique, in which case we define $r_n(\mathbf{x})$ to be the infimum. For both cases, since the distribution of $\mathbf{X}$ is continuous, we have
\begin{eqnarray}
\frac{\text{P}(B(\mathbf{x},r_n(\mathbf{x})))}{\text{P}(B(\mathbf{x},A))}=\frac{2k+2}{n}.
\label{eq:probratio}
\end{eqnarray}
We have the following lemma, which gives an lower bound of $n$ and upper bound of $\rho$ that hold with high probability.
\begin{lem}\label{lem:rhoub}
	We have
	\begin{eqnarray}
	\text{P}\left(n\leq \frac{1}{2}N\text{P}(B(\mathbf{x},A))\right)\leq \exp\left[-\frac{1}{2}(1-\ln 2)N\text{P}(B(\mathbf{x},A))\right].
	\label{eq:smalln}
	\end{eqnarray}
	Furthermore, if
	$n\geq NP(B(\mathbf{x},A))/2$,
	then for all $\mathbf{x}\in S_1$,
	\begin{eqnarray}
	\text{P}(\rho> r_n(\mathbf{x})|n)\leq\exp[-(1-\ln 2)(k+1)].
	\label{eq:rbound}
	\end{eqnarray}
\end{lem}
\begin{proof}
	According to Chernoff inequality,
	\begin{eqnarray}
	\text{P}\left(n\leq\frac{1}{2}N\text{P}(B(\mathbf{x},A))\right)\leq e^{-N\text{P}(B(\mathbf{x},A))}(2e)^{\frac{1}{2}N\text{P}(B(\mathbf{x},A))}=\exp\left[-\frac{1}{2}(1-\ln 2)N\text{P}(B(\mathbf{x},A))\right].\nonumber
	\end{eqnarray}
	Hence,	\eqref{eq:smalln} is true.
	
	Now we prove \eqref{eq:rbound}. Recall \eqref{eq:kdef}, $k$ is determined by $k=\left\lfloor Kn^q\right\rfloor+1$, thus $k/n\sim n^{q-1}$, thus for sufficiently large $N$, $k/n<1$ and hence $r_n(\mathbf{x})\leq A$. If we know that a sample is already in $B(\mathbf{x},A)$, then the conditional probability of that point falling in $B(\mathbf{x},r_n(\mathbf{x}))$ is
	\begin{eqnarray}
	\text{P}\left(\mathbf{X}\in B(\mathbf{x},r_n(\mathbf{x}))|\mathbf{X}\in B(\mathbf{x},A)\right)=
	\frac{\text{P}(B(\mathbf{x},r_n(\mathbf{x})))}{\text{P}(B(\mathbf{x},A))}.
	\label{eq:prob}
	\end{eqnarray}
	Define
	$n'=\sum_{i=1}^N \mathbf{1}(\mathbf{X}_i\in B(\mathbf{x},r_n(\mathbf{x})))$.
	According to \eqref{eq:prob}, $n'$ follows Binomial distribution conditional on $n$, i.e.
	$n'|n\sim \text{Binomial}(n,\text{P}(B(\mathbf{x},r_n(\mathbf{x})))/\text{P}(B(\mathbf{x},A)))$. Using Chernoff inequality again,
	\begin{eqnarray}
	\text{P}(\rho>r_n(\mathbf{x})|n)&=&\text{P}(n'\leq k|n)\nonumber\\
	&\leq &\exp\left[-n\frac{\text{P}(B(\mathbf{x},r_n(\mathbf{x})))}{\text{P}(B(\mathbf{x},A))}\right]\left(\frac{en\frac{\text{P}(B(\mathbf{x},r_n(\mathbf{x})))}{\text{P}(B(\mathbf{x},A))}}{k+1}\right)^{k+1}\nonumber\\
	&=&e^{-(2k+2)}(2e)^{k+1}\nonumber\\
	&=&\exp[-(1-\ln 2)(k+1)],
	\end{eqnarray}
	in which the second last step comes from \eqref{eq:probratio}. The proof of \eqref{eq:rbound} is complete.
\end{proof}
Now we bound $I_1$.
\begin{eqnarray}
I_1=\mathbb{E}[\mathbf{1}(g(\mathbf{X})\neq g^*(\mathbf{X}))|\eta(\mathbf{X})|\mathbf{1}(\mathbf{X}\in S_1)]\leq P_1+P_2+I_1',
\label{eq:I1decomp}
\end{eqnarray}
in which $P_1$, $P_2$ and $I_1'$ are defined as
\begin{eqnarray}
P_1&:=&\text{P}\left(n\leq \frac{1}{2}N\text{P}(B(\mathbf{X},A))\right),
\label{eq:p1def}\\
P_2&:=&\text{P}\left(n> \frac{1}{2}N\text{P}(B(\mathbf{X},A)), \rho>r_n(\mathbf{X})\right),
\label{eq:p2def}
\end{eqnarray}
\begin{eqnarray}
I_1':=\mathbb{E}\left[\mathbf{1}(g(\mathbf{X})\neq g^*(\mathbf{X}))|\eta(\mathbf{X})|\mathbf{1}\left(\mathbf{X}\in S_1,n>\frac{1}{2}N\text{P}(B(\mathbf{X},A)),\rho\leq r_n(\mathbf{x})\right)\right].
\label{eq:I1pdef}
\end{eqnarray}
According to \eqref{eq:smalln} and Assumption 1(d),
\begin{eqnarray}
P_1=\mathbb{E}\left[\text{P}\left(n\leq \left.\frac{1}{2}N\text{P}(B(\mathbf{X},A))\right|\mathbf{X}\right)\right]
\leq\mathbb{E}\left[\exp\left[-\frac{1}{2}(1-\ln 2)C_dv_dA^dNf(\mathbf{X})\right]\right].
\end{eqnarray}
$P_2$ can be bounded by
\begin{eqnarray}
P_2&\leq& \mathbb{E}\left[\text{P}(\rho>r_n(\mathbf{X})|n)\left|n>\frac{1}{2}N\text{P}(B(\mathbf{X},A))\right.\right]\nonumber\\
&\overset{(a)}{\leq}& \mathbb{E}\left[\exp[-(1-\ln 2)(k+1)] \left|n>\frac{1}{2}N\text{P}(B(\mathbf{X},A))\right.\right]\nonumber\\
&\overset{(b)}{\leq}& \mathbb{E}\left[\exp[-(1-\ln 2)Kn^q] \left|n>\frac{1}{2}N\text{P}(B(\mathbf{X},A))\right.\right]\nonumber\\
&\overset{(c)}{\leq}&\mathbb{E}\left[\exp\left[-(1-\ln 2)K2^{-q}(C_dv_dA^d)^q N^q f^q(\mathbf{X})\right]\right].
\end{eqnarray}
Here, (a) comes from \eqref{eq:rbound}. (b) comes from \eqref{eq:kdef}, which implies that $k>Kn^q$. (c) comes from Assumption 1 (d).

Use Lemma \ref{lem:tail} to be shown in Appendix~\ref{sec:lemmas}, we know that $P_1$ and $P_2$ can both be bounded by $\mathcal{O}(N^{-\beta})$.

Now we bound $I_1'$. For any test point $\mathbf{x}\in S_1$, if $\rho\leq r_n(\mathbf{x}) $ and $n>NP(B(\mathbf{x},A))/2$, then
\begin{eqnarray}
\text{P}(B(\mathbf{x},r_n(\mathbf{x})))=\frac{2k+2}{n}\text{P}(B(\mathbf{x},A))
\leq C_1 N^{-(1-q)}\text{P}^q(B(\mathbf{x},A)),
\end{eqnarray}
in which the second step uses $k>Kn^q$ and $n>N\text{P}(B(\mathbf{x},A))/2$.

In addition, from Assumption 1(d), $\text{P}(B(\mathbf{x},r_n(\mathbf{x})))\geq C_dv_dr_n^d(\mathbf{x})f(\mathbf{x})$, hence
\begin{eqnarray}
r_n(\mathbf{x})\leq \left[\frac{C_1}{C_dv_d}N^{-(1-q)}\frac{\text{P}^q(B(\mathbf{x},A))}{f(\mathbf{x})}\right]^\frac{1}{d}\leq \left[\frac{C_1}{C_dv_d}N^{-(1-q)}\frac{1}{h^{1-q}(\mathbf{x})}\right]^\frac{1}{d}\leq \left[\frac{C_1}{C_dv_d}N^{-(1-q)(1-\delta)}\right]^\frac{1}{d}:=a_N,\nonumber\\
\label{eq:annew}
\end{eqnarray}
in which the second step comes from the definition of $S_1$ in \eqref{eq:s1def}. 

Using Lemma \ref{lem:conc} that will be proved in Section \ref{sec:lemmas}, and Assumption 1(c), we have
\begin{eqnarray}
|\mathbb{E}[\hat{\eta}(\mathbf{x})|\rho]-\eta(\mathbf{x})|=\eta(B(\mathbf{x},\rho))-\eta(\mathbf{x})\leq C_c \rho^2
\leq  C_cr_n^2(\mathbf{x})
\leq C_ca_N^2.
\label{eq:bias}
\end{eqnarray}
Recall that in the definition \eqref{eq:s1def}, we let $|\eta(\mathbf{x})|>2\Delta$ for all $\mathbf{x}\in S_1$. Now we define $\Delta$ as
\begin{eqnarray}
\Delta=C_ca_N^2,
\label{eq:deltadef}
\end{eqnarray}
then there exists a constant $C_2$, such that for $\rho\geq r_n(\mathbf{x})$ and $n>N\text{P}(B(\mathbf{x},A))/2$, 
\begin{eqnarray}
\text{P}(g(\mathbf{x})\neq g^*(\mathbf{x})|\rho,n)&=&\text{P}(\text{sign}(\hat{\eta}(\mathbf{x}))\neq \text{sign}(\eta(\mathbf{x})|\rho)\nonumber\\
&\leq &\text{P}(|\hat{\eta}(\mathbf{x})-\eta(\mathbf{x})|>|\eta(\mathbf{x})||\rho)\nonumber\\
&\leq &\text{P}(|\hat{\eta}(\mathbf{x})-\mathbb{E}[\hat{\eta}(\mathbf{x})|\rho]|>|\eta(\mathbf{x})|-|\mathbb{E}[\hat{\eta}(\mathbf{x})|\rho]-\eta(\mathbf{x})||\rho)\nonumber\\
&\leq& \text{P}(|\hat{\eta}(\mathbf{x})-\mathbb{E}[\hat{\eta}(\mathbf{x})|\rho]|>|\eta(\mathbf{x})|-\Delta|\rho)\nonumber\\
&\overset{(a)}{\leq}&2\exp\left[-\frac{1}{2}k(|\eta(\mathbf{x})|-\Delta)^2\right]\nonumber\\
&\overset{(b)}{\leq}&2\exp\left[-\frac{1}{8}k\eta^2(\mathbf{x})\right]\nonumber\\
&\overset{(c)}{\leq}&2\exp\left[-\frac{1}{8}Kn^q\eta^2(\mathbf{x})\right]\nonumber\\
&\overset{(d)}{\leq}&2\exp\left[-C_2 N^{q(1-\delta)}\eta^2(\mathbf{x})\right].
\end{eqnarray}
(a) uses Hoeffding's inequality. For (b), note that in $S_1$, $|\eta(\mathbf{x})|>2\Delta$, hence
$\eta(\mathbf{x})-\Delta>\eta(\mathbf{x})/2$.
(c) comes from \eqref{eq:kdef}. (d) uses \eqref{eq:relation}: $n> N\text{P}(B(\mathbf{x},A))/2\gtrsim Nh(\mathbf{x})\gtrsim N^{1-\delta}$. Hence \eqref{eq:I1pdef} can be bounded using the same method as was already used in the derivation of \eqref{eq:etabound}:
\begin{eqnarray}
I_1'&\leq &2\mathbb{E}\left[|\eta(\mathbf{X})|\exp\left[-C_2 N^{q(1-\delta)}\eta^2(\mathbf{X})\right]\mathbf{1}(\mathbf{X}\in S_1)\right]= \mathcal{O}\left(N^{-\frac{\alpha+1}{2}q(1-\delta)}\right).
\label{eq:I1p}
\end{eqnarray}

Recall \eqref{eq:I1decomp} and the fact that $P_1$ and $P_2$ are both bounded by $\mathcal{O}(N^{-\beta})$,
\begin{eqnarray}
I_1=\mathcal{O}\left(N^{-\frac{\alpha+1}{2} q(1-\delta)} \right)+\mathcal{O}(N^{-\beta}).
\label{eq:I1new}
\end{eqnarray}

\noindent\textbf{Bound of $I_2$.} From \eqref{eq:Idef},
\begin{eqnarray}
I_2&=&\mathbb{E}[\mathbf{1}(g(\mathbf{X})\neq g^*(\mathbf{X}))|\eta(\mathbf{X})|\mathbf{1}(\mathbf{X}\in S_2)]\nonumber\\
&\leq &\mathbb{E}[|\eta(\mathbf{X})|\mathbf{1}(|\eta(\mathbf{X})|<2\Delta)]\nonumber\\
&\leq & 2\Delta \text{P}(|\eta(\mathbf{X})|\leq 2\Delta)\nonumber\\
&\overset{(a)}{=} &\mathcal{O}(\Delta^{\alpha+1})
\overset{(b)}{=}\mathcal{O}(a_N^{2(\alpha+1)})\overset{(c)}{=}\mathcal{O}\left(N^{-\frac{2(\alpha+1)}{d}(1-q)(1-\delta)}\right),
\label{eq:I2new}
\end{eqnarray}
in which (a) comes from Assumption 1(a), (b) comes from \eqref{eq:deltadef}, and (c) comes from \eqref{eq:annew}.

\noindent\textbf{Bound of $I_3$.}
Define
\begin{eqnarray}
\phi(\mathbf{x},n):=\mathbb{E}[|\hat{\eta}(\mathbf{x})-\eta(\mathbf{x})||n],
\end{eqnarray}
then
\begin{eqnarray}
I_3=\mathbb{E}[\mathbf{1}(g(\mathbf{X})\neq g^*(\mathbf{X}))|\eta(\mathbf{X})|\mathbf{1}(\mathbf{X}\in S_3)]
\leq \mathbb{E}[|\hat{\eta}(\mathbf{X})-\eta(\mathbf{X})|\mathbf{1}(\mathbf{X}\in S_3)]
=\mathbb{E}[\phi(\mathbf{X},n)\mathbf{1}(\mathbf{X}\in S_3)].\nonumber
\end{eqnarray}

Then we give a bound of $\phi(\mathbf{x},n)$.

\noindent Case 1): If $n\leq \frac{1}{2}N\text{P}(B(\mathbf{x},A))$, we bound it with
\begin{eqnarray}
\phi(\mathbf{x},n)\leq \mathbb{E}[|\hat{\eta}(\mathbf{x})||n]+|\eta(\mathbf{x})|\leq  2.
\label{eq:phib1}
\end{eqnarray}
\noindent Case 2):  If $n>\frac{1}{2}N\text{P}(B(\mathbf{x},A))$, then according to \eqref{eq:relation}, \eqref{eq:kdef}, \eqref{eq:c0} and \eqref{eq:s3def}, which requires that $h(\mathbf{x})>C_0/N$,  it can be shown that $k\leq n$. Recall that $\lambda$ is defined in \eqref{eq:lamdef}.
Then use Lemma \ref{lem:mseadaptive} in Appendix~\ref{sec:lemmas},
\begin{eqnarray}
\phi(\mathbf{x},n)\leq \sqrt{\mathbb{E}[(\hat{\eta}(\mathbf{x})-\eta(\mathbf{x}))^2]}
\leq  \sqrt{C_M} h^{-\lambda}(\mathbf{x})N^{-\lambda}.
\label{eq:phib2}
\end{eqnarray}
With \eqref{eq:phib1} and \eqref{eq:phib2}, $I_3$ can be bounded by:
\begin{eqnarray}
I_3
\leq  2\text{P}\left(\mathbf{X}\in S_3, n\leq \frac{1}{2}N\text{P}(B(\mathbf{X},A))\right)+\sqrt{C_M}\mathbb{E}[h^{-\lambda}(\mathbf{X})\mathbf{1}(\mathbf{X}\in S_3)]N^{-\lambda}.
\label{eq:I3decomp}
\end{eqnarray}
In the derivation about $I_1$, we have shown that the first term decays with $\mathcal{O}(N^{-\beta})$. The second term can be bounded using Lemma \ref{lem:tail}. If $\beta\neq \lambda$, the bound of $I_3$ can be expressed as
\begin{eqnarray}
I_3=\mathcal{O}\left(N^{-\lambda(1-\delta)-\delta \beta}\right)+\mathcal{O}(N^{-\beta}).
\label{eq:I3new}
\end{eqnarray}
If $\beta=\lambda$, then
$I_3=\mathcal{O}(N^{-\beta}\ln N)$.

\noindent\textbf{Bound of $I_4$}.
\begin{eqnarray}
I_4&=&\mathbb{E}[\mathbf{1}(g(\mathbf{X})\neq g^*(\mathbf{X}))|\eta(\mathbf{X})|\mathbf{1}(\mathbf{X}\in S_4)]\leq \text{P}(\mathbf{X}\in S_4)=\mathcal{O}(N^{-\beta}).
\label{eq:I4new}
\end{eqnarray}
Recall the expression of $R-R^*$ in \eqref{eq:R}, and combine \eqref{eq:I1new}, \eqref{eq:I2new}, \eqref{eq:I3new} and \eqref{eq:I4new}, if $\beta\neq\lambda$,
\begin{eqnarray}
R-R^*&=&\mathcal{O}\left(N^{-\frac{\alpha+1}{2}q(1-\delta)}\right)+\mathcal{O}\left(N^{-\frac{2}{d}(\alpha+1)(1-q)(1-\delta)}\right)+\mathcal{O}\left(N^{-\lambda(1-\delta)-\delta \beta}\right)+\mathcal{O}(N^{-\beta})\nonumber\\
&=&\mathcal{O}(N^{-(\alpha+1)\lambda(1-\delta)})+\mathcal{O}\left(N^{-\lambda(1-\delta)-\delta \beta}\right)+\mathcal{O}(N^{-\beta}).
\end{eqnarray}
The first and the second terms contain $\delta$. To optimize the overall convergence rate, let
$\delta=a\alpha/(a\alpha+\beta)$,
then
\begin{eqnarray}
R-R^*&=&\mathcal{O}\left(N^{-\frac{\lambda\beta(\alpha+1)}{\lambda\alpha+\beta}}\right)+\mathcal{O}(N^{-\beta})=\mathcal{O}\left(N^{-\min\left\{\beta,\frac{\lambda\beta(\alpha+1)}{\lambda\alpha+\beta} \right\}}\right).
\end{eqnarray}
Now consider $\beta=\lambda$. In this case,
$R-R^*=\mathcal{O}(N^{-\beta}\ln N)$.

\textbf{The optimal convergence rate}. From \eqref{eq:lamdef}, the maximal $\lambda$ is $2/(d+4)$, which is attained if $q=4/(d+4)$. Then the optimal convergence rate is

\begin{eqnarray}
R_{opt}-R^*=\left\{
\begin{array}{ccc}
\mathcal{O}\left(N^{-\min\left\{\frac{2\beta(\alpha+1)}{\beta d+2(\alpha+2\beta)},\beta \right\} }\right)& \text{if} & \beta\neq \frac{2}{d+4}\\
\mathcal{O}(N^{-\beta}\ln N),&\text{if} & \beta=\frac{2}{d+4}
\end{array}.
\right.
\end{eqnarray}

\section{Proof of Theorem \ref{thm:standardreg}: Convergence rate of the standard kNN regression with bounded $\eta$}\label{sec:std-b}
\subsection{Upper bound}
For any test point $\mathbf{x}$, define $\rho$ as the distance from $\mathbf{x}$ to its $(k+1)$-th nearest neighbor among the training dataset. Then
\begin{eqnarray}
\mathbb{E}[(g(\mathbf{x})-\eta(\mathbf{x}))^2|\rho]=(\mathbb{E}[g(\mathbf{x})|\rho]-\eta(\mathbf{x}))^2+\Var[g(\mathbf{x})|\rho]=(\eta(B(\mathbf{x},\rho))-\eta(\mathbf{x}))^2+\Var\left[\left.\frac{1}{k}\sum_{i=1}^k Y^{(i)}\right|\rho\right],\nonumber\\
\label{eq:mse}
\end{eqnarray}
in which the last step comes from \eqref{eq:etaexp} in Lemma \ref{lem:conc}.

Define the following two events:

\noindent Event 1: $f(\mathbf{X})>2k/(NC_dv_dD^d)$ and $\rho<D$;

\noindent Event 2: $f(\mathbf{X})\leq 2k/(NC_dv_dD^d)$ or $\rho\geq D$.

Define a random variable $E$, $E=1$ when event 1 occurs, and $E=2$ when event 2 occurs.

\textbf{Case 1}. If Event 1 happens, then according to Assumption 1(c),
$(\eta(B(\mathbf{x},\rho))-\eta(\mathbf{x}))^2\leq C_d^2\rho^{4}$.

For the second term in \eqref{eq:mse}, we use similar steps as \eqref{eq:variance}. In the derivation in \eqref{eq:variance}, we used $|\eta(\mathbf{x})|\leq 1$ and $\Var[Y^{(i)}|\rho,\mathbf{X}^{(1)},\ldots, \mathbf{X}^{(N)}]\leq 1$. Here these two bounds are replaced by $M$ and $C_a$, respectively. Hence
\begin{eqnarray}
\Var\left[\left.\frac{1}{k}\sum_{i=1}^k Y^{(i)}\right|\rho\right]\leq \frac{M^2+C_a}{k},
\label{eq:varg}
\end{eqnarray}

Therefore for all $\mathbf{x}$ that satisfies $f(\mathbf{x})> 2k/(NC_dv_dD^d)$,
\begin{eqnarray}
\mathbb{E}[(g(\mathbf{x})-\eta(\mathbf{x}))^2|\mathbf{x},\rho\leq D]
&\leq& C_d^2 \mathbb{E}[\rho^{4}|\rho<D,\mathbf{x}]+\frac{C_a+M^2}{k}.\nonumber\\
&\leq & C_d^2\mathbb{E}\left[\left.\frac{\text{P}^{\frac{4}{d}}(B(\mathbf{x},\rho))}{(C_dv_df(\mathbf{x}))^\frac{4}{d}}\right|\text{P}(B(\mathbf{x},\rho))\leq \text{P}(B(\mathbf{x},D))\right]+\frac{C_a+M^2}{k},\nonumber
\end{eqnarray}
in which the last step uses Assumption 1 (d). 

Using Lemma \ref{lem:order} to be shown in Section \ref{sec:lemmas}, we know that there exists a constant $C_1$ such that
\begin{eqnarray}
\mathbb{E}[(g(\mathbf{x})-\eta(\mathbf{x}))^2)\mathbf{1}(\rho<D)]\leq C_1\left(\frac{k}{N}\right)^\frac{4}{d} f^{-\frac{4}{d}}(\mathbf{x})+\frac{C_a+M^2}{k}.
\end{eqnarray}
Hence if $\beta\neq 4/d$,
\begin{eqnarray}
\mathbb{E}[(g(\mathbf{X})-\eta(\mathbf{X}))^2\mathbf{1}(E=1)]&\leq& C_1\left(\frac{k}{N}\right)^\frac{4}{d} \mathbb{E}\left[f^{-\frac{4}{d}}(\mathbf{X})\mathbf{1}\left(f(\mathbf{X})>\frac{2k}{NC_dv_dD^d}\right)\right]+\frac{C_a+M^2}{k}\nonumber\\
&=&\mathcal{O}\left(\left(\frac{k}{N}\right)^\frac{4}{d}\right)+\mathcal{O}\left(\left(\frac{k}{N}\right)^\beta\right)+\mathcal{O}\left(\frac{1}{k}\right).
\end{eqnarray}
Here, in the last step, we Lemma \ref{lem:tail} shown in Section \ref{sec:lemmas}. If $\beta=4/d$,
\begin{eqnarray}
\mathbb{E}[(g(\mathbf{X})-\eta(\mathbf{X}))^2\mathbf{1}(E=1)]=\mathcal{O}\left(\left(\frac{k}{N}\right)^\beta \ln \frac{N}{k}\right)+\mathcal{O}\left(\frac{1}{k}\right).
\label{eq:case1}
\end{eqnarray}
\textbf{Case 2}. If Event 2 happens, then according to Assumption 4, $|\eta(\mathbf{x})|\leq M$ for any $\mathbf{x}$. Hence the first term in \eqref{eq:mse} is bounded by $(\eta(B(\mathbf{x},\rho))-\eta(\mathbf{x}))^2\leq 4M^2$. For the second term in \eqref{eq:mse}, note that for any $i\in \{1,\ldots,k \}$,
\begin{eqnarray}
\Var[Y^{(i)}|\rho]=\mathbb{E}[\Var[Y^{(i)}|\mathbf{X}^{(i)},\rho]]+\Var[\mathbb{E}[Y^{(i)}|\rho,\mathbf{X}^{(i)}]]
\leq C_a+\Var[\eta(\mathbf{X}^{(i)})|\rho]
\leq C_a+M^2.
\end{eqnarray}
Using Cauchy inequality,
$\Var\left[\left.\frac{1}{k}\sum_{i=1}^k Y^{(i)}\right|\rho\right]\leq C_a+M^2$,
and thus
\begin{eqnarray}
\mathbb{E}[(g(\mathbf{x})-\eta(\mathbf{x}))^2|\rho]\leq C_a+5M^2.
\end{eqnarray}
Now we can give an overall bound of the loss function under case 2:
\begin{eqnarray}
&&\mathbb{E}[(g(\mathbf{X})-\eta(\mathbf{X}))^2 \mathbf{1}(E=2)]\nonumber\\
&\leq &(C_a+5M^2)\left(\text{P}\left(f(\mathbf{X})\leq \frac{2k}{NC_dv_dD^d}\right)+\text{P}\left(f(\mathbf{X})> \frac{2k}{NC_dv_dD^d} ,\rho\geq D\right)\right).
\label{eq:case2}
\end{eqnarray}
From Assumption 1 (b), the first term in the bracket in \eqref{eq:case2} decays with $\mathcal{O}((k/N)^\beta)$. Moreover, if $f(\mathbf{x})>2k/(NC_dv_dD^d)$, then $\text{P}(B(\mathbf{x},D))> 2k/N$, and \eqref{eq:largerho} still holds here. Therefore, the second term in the bracket in \eqref{eq:case2} decays faster than any polynomial. With this observation, and combine with \eqref{eq:case1}, we have
\begin{eqnarray}
R-R^*=\mathbb{E}[(g(\mathbf{X})-\eta(\mathbf{X}))^2]=\left\{
\begin{array}{ccc}
\mathcal{O}\left(\left(\frac{k}{N}\right)^{\min\{\beta,\frac{4}{d}\}} \right)+\mathcal{O}\left(\frac{1}{k}\right) &\text{if} & \beta\neq \frac{4}{d},\\
\mathcal{O}\left(\left(\frac{k}{N}\right)^\beta \ln \frac{N}{k}\right)+\mathcal{O}\left(\frac{1}{k}\right) &\text{if} & \beta=\frac{4}{d}.
\end{array}
\right.
\end{eqnarray}
The fastest rate is attained if 
\begin{eqnarray}
k\sim \left\{
\begin{array}{ccc}
N^\frac{4}{d+4} & \text{if} & \beta\geq \frac{4}{d},\\
N^{\frac{\beta}{\beta+1} }& \text{if} & \beta<\frac{4}{d}.
\end{array}
\right.
\end{eqnarray}
The corresponding optimal convergence rate is
\begin{eqnarray}
R-R^*=\left\{
\begin{array}{ccc}
\mathcal{O}\left(N^{-\min\left\{\frac{4}{d+4},\frac{\beta}{\beta+1}\right\}}\right) &\text{if} & \beta\neq \frac{4}{d},\\
\mathcal{O}\left(N^{-\frac{\beta}{\beta+1}}\ln N\right) &\text{if} & \beta=\frac{4}{d}.
\end{array}
\right.
\end{eqnarray}
\color{black}
\subsection{Lower bound}
We prove the following statements separately:
\begin{eqnarray}
\underset{(f,\eta)\in \mathcal{S}}{\sup} (R-R^*)&\gtrsim& \frac{1}{k};
\label{eq:reglb1}\\
\underset{(f,\eta) \in \mathcal{S}}{\sup} (R-R^*)&\gtrsim& \left(\frac{k}{N}\right)^\frac{4}{d};
\label{eq:reglb2}\\
\underset{(f,\eta)\in \mathcal{S}}{\sup} (R-R^*)&\gtrsim& \left(\frac{k}{N}\right)^\beta.
\label{eq:reglb3}
\end{eqnarray}

\textbf{Proof of \eqref{eq:reglb1}}.
Given arbitrary distribution with pdf $f(\mathbf{X})$, let $Y\sim \mathcal{N}(0,\sigma^2)$, in which $\sigma^2 \leq C_a$, $C_a$ is the constant in Assumption \ref{ass:basic}. Then $\eta(\mathbf{x})=0$ everywhere, and
\begin{eqnarray}
R-R^*=\mathbb{E}[(\hat{\eta}(\mathbf{X})-\eta(\mathbf{X}))^2]=\Var[\eta(\mathbf{X})]=\Var\left[\frac{1}{k}\sum_{i=1}^k Y^{(i)}\right]=\frac{\sigma^2}{k}.
\end{eqnarray}
Hence, \eqref{eq:reglb1} holds.

\textbf{Proof of \eqref{eq:reglb2}}.

For simplicity, in the following proof, we assume that we are using max norm in kNN regression.

Construct the following distribution. Let $\mathbf{X}=(X_1,\ldots,X_d)\sim \text{Uniform}([-1,1]^d)$, and 
\begin{eqnarray}
\eta(\mathbf{x})=\eta_1(x_1)=\left\{
\begin{array}{ccc}
x_1^2+2x_1 &\text{if} & x_1<0\\
-x_1^2+2x_1 &\text{if} & x_1\geq 0
\end{array}
\right.
\end{eqnarray}
Note that $\norm{\nabla^2 \eta}=2$ for $x_1\neq 0$. Then define
\begin{eqnarray}
I_\Delta =\{\mathbf{x}|-1+\Delta<x_1<-\Delta \text{ or } \Delta<x_1<1-\Delta \}.
\end{eqnarray}
For this distribution, 
\begin{eqnarray}
R-R^*&=&\mathbb{E}[(\hat{\eta}(\mathbf{X})-\eta(\mathbf{X}))^2]\nonumber\\
&\geq & \mathbb{E}\left[(\mathbb{E}[\hat{\eta}(\mathbf{X})|\mathbf{X}]-\eta(\mathbf{X}))^2 \mathbf{1}(\mathbf{X}\in I_\Delta)\right].
\end{eqnarray}
For any $\mathbf{x}$, we have $\mathbb{E}[\hat{\eta}(\mathbf{X})]=\mathbb{E}[\eta(B(\mathbf{x},\rho))]$, since $\norm{\nabla^2 \eta(\mathbf{x})}=2$ almost everywhere, 
\begin{eqnarray}
|\mathbb{E}[\hat{\eta}(\mathbf{x})-\eta(\mathbf{x})]|&=&\left|\mathbb{E}[\eta(B(\mathbf{x},\rho))]-\eta(\mathbf{x})\right|\nonumber\\
&=&\left|\mathbb{E}\left[\frac{\int_{x_1-\rho}^{x_1+\rho}( \eta_1(x_1')-\eta_1(x_1))dx_1'}{2\rho}\right]\right|\nonumber\\
&=&\left|\mathbb{E}\left[\frac{1}{2\rho}\int_{-\rho}^\rho \frac{1}{2} \eta^{''}(x_1) t^2 dt \right]\mathbf{1}(\rho\leq \Delta)\right|\nonumber\\
&\geq &\mathbb{E}\left[\frac{1}{3}\rho^2 \mathbf{1}(\rho\leq \Delta)\right]-2\text{P}(\rho>\Delta).
\end{eqnarray} 
Note that with max norm, for uniform distribution, $\text{P}(B(\mathbf{x},\rho))=2^df(\mathbf{x})\rho^d$ if $B(\mathbf{x},\rho)$ does not exceed $[-1,1]^d$. Here $f(\mathbf{x})=1/2^d$, hence $\text{P}(B(\mathbf{x},\rho))=\rho^d$ if $B(\mathbf{x},\rho)\subset [-1,1]^d$. Hence
\begin{eqnarray}
\mathbb{E}[\rho^2 \mathbf{1}(\rho\leq \Delta)]&=&\mathbb{E}[\text{P}^\frac{2}{d}(B(\mathbf{x},\rho))\mathbf{1}(\rho\leq \Delta)]\nonumber\\
&=& \mathbb{E}[\text{P}^\frac{2}{d}(B(\mathbf{x},\rho))]-\mathbb{E}[\text{P}^\frac{2}{d}(B(\mathbf{x},\rho))\mathbf{1}(\rho>\Delta)]\nonumber\\
&\geq & \frac{\Gamma\left(k+1+\frac{2}{d}\right)}{\Gamma(k+1)} \frac{\Gamma(N+1)}{\Gamma\left(N+1+\frac{2}{d}\right)}-\text{P}(\rho>\Delta),
\end{eqnarray}
in which the last step uses the fact that $\text{P}(B(\mathbf{x},\rho))$ follows $\text{Beta}(k+1,N-k)$ distribution. Therefore
\begin{eqnarray}
|\mathbb{E}[\hat{\eta}(\mathbf{x})]-\eta(\mathbf{x})|\geq \frac{1}{3} \frac{\Gamma\left(k+1+\frac{2}{d}\right)}{\Gamma(k+1)} \frac{\Gamma(N+1)}{\Gamma\left(N+1+\frac{2}{d}\right)}-\frac{7}{3}\text{P}(\rho>\Delta).
\end{eqnarray}
Since $\text{P}(\rho>\Delta)$ decays exponentially, $|\mathbb{E}[\hat{\eta}(\mathbf{x})]-\eta(\mathbf{x})|\sim (k/N)^{2/d}$, therefore
\begin{eqnarray}
R-R^* \sim \left(\frac{k}{N}\right)^\frac{4}{d} \text{P}(\mathbf{X}\in I_\Delta)\sim \left(\frac{k}{N}\right)^\frac{4}{d}.
\end{eqnarray}
Hence \eqref{eq:reglb2} holds.

\textbf{Proof of \eqref{eq:reglb3}}.
Construct $(n+1)$ cubes $I_1,\ldots, I_{n+1}$. $\mathbf{X}$ is supported by these cubes, and is uniformly distributed within each cube. Let $m$ be the pdf value in the first $n$ cubes. For the remaining cube, the density is $(1-2^d nm)/2^d$. This ensures that the total probability mass of all $(n+1)$ cubes is $1$. $m$ and $n$ change with $k$ and $N$. The precise definition of each cube $I_j$ is
\begin{eqnarray}
I_j=\left\{ \mathbf{x}|4j-1<x<4j+1 ,x_2,\ldots, x_d\in [-1,1] \right\}
\end{eqnarray}
for $j=1,\ldots, n+1$. Similar to the proof of \eqref{eq:reglb2}, define
\begin{eqnarray}
I_{j\Delta}=\left\{\mathbf{x}| 4j-1+\Delta<x_1<4j-\Delta \text{ or } 4j+\Delta<x_1<4j+1-\Delta, x_2,\ldots, x_d\in [-1,1] \right\}.
\end{eqnarray}
In $I_{(n+1)\Delta}$, let $\eta(\mathbf{x})=0$. Otherwise, let
\begin{eqnarray}
\eta(\mathbf{x})=\eta_1(x_1)=\left\{
\begin{array}{ccc}
(x_1-4j)^2 +2(x_1-4j) &\text{if} & 4j-1\leq x_1<4j\\
-(x_1-4j)^2+2(x_1-4j) &\text{if} & 4j\leq x\leq 4j+1.
\end{array}
\right.
\end{eqnarray}
Then
\begin{eqnarray}
R-R^*&=& \mathbb{E}\left[ (\hat{\eta}(\mathbf{X})-\eta(\mathbf{X}))^2\right]\nonumber\\
&\geq&\sum_{j=1}^n \mathbb{E}[(\mathbb{E}[\hat{\eta}(\mathbf{X})|\mathbf{X}]-\eta(\mathbf{X}))^2\mathbf{1}(\mathbf{X}\in I_{j\Delta})]\nonumber\\
&=& n\mathbb{E}[(\mathbb{E}[\hat{\eta}(\mathbf{X})|\mathbf{X}]-\eta(\mathbf{X}))^2\mathbf{1}(\mathbf{X}\in I_{1\Delta})].
\end{eqnarray}
Ensure that in $I_{1\Delta}$, $\text{P}(B(\mathbf{x},\Delta))= 2(k+1)/N$, i.e. $m2^d\Delta^d= 2(k+1)/N$, then for $\mathbf{x}\in I_{1\Delta}$, 
\begin{eqnarray}
\text{P}(\rho>\Delta)=\text{P}\left(\text{P}(B(\mathbf{x},\rho))>\frac{2(k+1)}{N}\right)\leq e^{-(1-\ln 2)(k+1)},
\end{eqnarray}
which decays exponentially.

The remaining steps are similar to the proof of \eqref{eq:reglb2}. Since $\text{P}(B(\mathbf{x},\rho))=m2^d\rho^d$, for $\mathbf{x}\in I_{1\Delta}$,
\begin{eqnarray}
|\mathbb{E}[\hat{\eta}(\mathbf{x})]-\eta(\mathbf{x})|&=&\frac{1}{3}\mathbb{E}[\rho^2 \mathbf{1}(\rho\leq \Delta)]-2\text{P}(\rho>\Delta)\nonumber\\
&=&\frac{1}{12m^\frac{2}{d}}\mathbb{E}\left[\text{P}^\frac{2}{d}(B(\mathbf{x},\rho))\mathbf{1}\left(\text{P}(B(\mathbf{x},\rho))\leq \frac{2(k+1)}{N}\right)\right]-2\text{P}(\rho>\Delta).
\end{eqnarray}
Note that $\text{P}(\rho>\Delta)$ decays exponentially, $\text{P}(B(\mathbf{x},\rho))\sim \text{Beta}(k+1,N-k-1)$, therefore there exists a constant $c$, such that $|\mathbb{E}[\hat{\eta}(\mathbf{x})]-\eta(\mathbf{x})|\geq c$ for $\mathbf{x}\in I_{1\Delta}$. Hence
\begin{eqnarray}
R-R^*\geq nc^2\text{P}(\mathbf{X}\in I_{1\Delta}).
\end{eqnarray}
Consider that the distribution should satisfy Assumption \ref{ass:basic}(b), 
\begin{eqnarray}
\text{P}(f(\mathbf{X})\leq m)=n\text{P}(\mathbf{X}\in I_1)\leq C_bm^\beta=C_b\left(\frac{2(k+1)}{2^d\Delta^dN}\right)^\beta.
\end{eqnarray}
Therefore, by using an appropriate $n$, let $n\text{P}(\mathbf{X}\in I_{1\Delta})\sim n\text{P}(\mathbf{X}\in I_1)\sim (k/N)^\beta$, then 
\begin{eqnarray}
R-R^*\sim \left(\frac{k}{N}\right)^\beta.
\end{eqnarray}
Hence \eqref{eq:reglb3} holds. The proof is complete.
\color{black}
\section{Proof of Theorem \ref{thm:minimaxreg}: Minimax convergence rate of regression with bounded $\eta$}
The proof of the minimax convergence rate for the regression is similar to the proof for the classification. Define $f(\mathbf{x})$, $r_0$, $L_0$, $\eta_\mathbf{v}(\mathbf{x})$, $r_M$, and $\mathcal{S}^*$ in the same way as \eqref{eq:pdf}, \eqref{eq:r0}, \eqref{eq:L0}, \eqref{eq:etav} and \eqref{eq:Sstar}. Then $(f,\eta_\mathbf{v})$ satisfies Assumptions 3 and 4 if $n_0 mv_dL^d\leq C_bm^\beta$, and $L\leq M$. Let the noise $\epsilon$ be normally distributed with variance $C_a$, in which $C_a$ is the constant in Assumption 3 (a), i.e., $Y=\eta(\mathbf{X})+\epsilon$, $\epsilon\sim \mathcal{N}(0,C_a)$.

Now we follow the proof of Lemma \ref{lem:assouad} shown in Appendix \ref{sec:assouad}.

For $\mathbf{x}\in B(\mathbf{a}_1,L)$, define $N_1$ as the number of training samples falling in $B(\mathbf{a}_1,L)$, then
\begin{eqnarray}
\underset{(f,\eta\in \mathcal{S}^*)}{\sup} \mathbb{E}[(g(\mathbf{x})-\eta(\mathbf{x}))^2]&\geq& \underset{\mathbf{v}(1)\in \{-1,1\}, \mathbf{v}(2)=\ldots=\mathbf{v}(n_0)=0}{\sup} \mathbb{E}[(g(\mathbf{x})-\eta(\mathbf{x})))^2]\nonumber\\
&\overset{(a)}{\geq}& L^4 \text{P}(\hat{\mathbf{v}}(1)\neq \mathbf{v}(1))\nonumber\\
&=&L^4 \mathbb{E}[\text{P}(\hat{\mathbf{v}}(1)\neq \mathbf{v}(1)|N_1)]\nonumber\\
&\overset{(b)}{\geq}&\frac{1}{2}L^4 \mathbb{E}[1-TV(P_+,P_-)]\nonumber\\
&\overset{(c)}{\geq}&\frac{1}{2}L^4 \mathbb{E}\left[1-\sqrt{\frac{1}{2}D(P_+||P_-)}\right]\nonumber\\
&=&\frac{1}{2}L^4 \mathbb{E}\left[1-\sqrt{\frac{N_1}{2}D(\mathcal{N}(L^2,C_a)||\mathcal{N}(-L^2,C_a))}\right]\nonumber\\
&=&\frac{1}{2}L^4 \left(1-\frac{L^2}{\sqrt{C_a}}\mathbb{E}[\sqrt{N_1}]\right)\nonumber\\
&\overset{(d)}{\geq} & \frac{1}{2}L^4 \left(1-\frac{L^2}{\sqrt{C_a}}\sqrt{N\omega}\right).
\end{eqnarray}
In (a), we define $\hat{\mathbf{v}}(1)=1$ if $g(\mathbf{x})>0$, and $-1$ otherwise. If $\hat{\mathbf{v}}(1)\neq \mathbf{v}(1)$, then $g(\mathbf{x})$ and $\eta(\mathbf{x})$ have different signs. According to the construction of $\eta_\mathbf{v}$ in \eqref{eq:etav}, $|g(\mathbf{x})-\eta(\mathbf{x})|>L^2$. Hence (a) holds. In (b), $TV$ denotes the total variation distance, and $P_+$ denotes the joint distribution of $N_1$ independent random variables, which is normal with mean $L^2$ and variance $C_a$, while $P_-$ is defined in the same way as $P_+$ except that the mean of the normal distribution becomes $-L^2$. (c) uses Pinsker's inequality \cite{tsybakov2009introduction}. In (d), $\omega$ is the probability mass of $B(\mathbf{a}_j,L)$ for $j=1,\ldots,n_0$, $\omega=mv_dL^d$, in which $v_d$ is the volume of unit ball.

For $\mathbf{x}\in B(\mathbf{a}_j)$ for $j=2,\ldots,n_0$, we can also get the same bound. Therefore
\begin{eqnarray}
\underset{(f,\eta)\in \mathcal{S}^*}{\sup} (R-R^*)\geq \frac{1}{2}n_0 \omega L^4 \left(1-\frac{L^2}{C_a}\sqrt{N\omega}\right).
\end{eqnarray}
Assumption 3 includes a tail assumption, i.e., Assumption 1 (b), under which we have $n_0 \omega\leq C_bm^\beta$. The proof of this statement can be found in the proof of Lemma \ref{lem:restriction} (2) in Appendix \ref{sec:restriction}. Moreover, from Assumption 4, $L\leq M$. To ensure that the expression in the above bracket is positive, i.e., $1-L^2\sqrt{N\omega/C_a}>0$, we need to ensure that $N\omega L^4\leq C_a$. Consider that $\omega\sim mL^d$, these above arguments show that (1) $n_0 mL^d=\mathcal{O}(m^\beta)$; (2) $L=\mathcal{O}(1)$; (3) $n_0 m L^{d+4}=\mathcal{O}(1)$. We then get the following lower bounds on the excess risk:

\noindent(1) Pick $L\sim N^{-\frac{1}{d+4}} $, $m\sim 1$, and $ n_0\sim N^{\frac{d}{d+4}}$, then
\begin{eqnarray}
\underset{(f,\eta)\in \mathcal{S}^*}{\sup} (R-R^*)\gtrsim N^{-\frac{4}{d+4}}.
\label{eq:mmxlb1}
\end{eqnarray}
\noindent(2) Pick $m\sim N^{-1}, L\sim 1, n_0\sim N^{1-\min\{\beta,1\}}$, then 
\begin{eqnarray}
\underset{(f,\eta)\in \mathcal{S}^*}{\sup} (R-R^*)\gtrsim N^{-\min\{\beta,1\}}.
\label{eq:mmxlb2}
\end{eqnarray}
Combine \eqref{eq:mmxlb1} and \eqref{eq:mmxlb2}, we get
\begin{eqnarray}
\underset{(f,\eta)\in \mathcal{S}^*}{\sup} (R-R^*)\gtrsim N^{-\min\left\{\frac{4}{d+4},\beta \right\}}.
\end{eqnarray}
The proof is complete.
\section{Proof of Theorem \ref{thm:adaptivereg}: Convergence rate of the adaptive kNN regression with bounded $\eta$}\label{sec:ada-b}
Without loss of generality, we assume $D\geq A$, since we have shown that this assumption does not impose further restrictions on the distribution of $\mathbf{X}$ in Section \ref{sec:adaptive}.

Define $h(\mathbf{x})$ in the same way as \eqref{eq:hdf}. Recall that $k$ is selected adaptively according to \eqref{eq:kdef}, since $0<q<1$, for any constant $K$, there exists a critical value $n_c$, so that when $n\geq n_c$, $k\leq n$, which means that the $k$-th nearest neighbor must fall in $B(\mathbf{x},A)$. We then discuss the following two cases:

Case 1: $h(\mathbf{x})>N^{-1}$, $n>N\text{P}(B(\mathbf{x},A))/2$, and $n>n_c$;

Case 2: $h(\mathbf{x})\leq N^{-1}$ or $n\leq N\text{P}(B(\mathbf{x},A))/2$ or $n\leq n_c$.

Now we discuss these two cases separately. Similar to the proof of the standard kNN regression, we still define a binary random variable $E$, in which $E=1$ if case 1 happens and $E=2$ if case 2 happens.

\noindent\textbf{Case 1}. 
For kNN regression, $g(\mathbf{x})=\hat{\eta}(\mathbf{x})$. Therefore $\mathbb{E}[(g(\mathbf{x})-\eta(\mathbf{x}))^2|n]$ can be bounded by Lemma \ref{lem:mseadaptive} of Appendix~\ref{sec:lemmas} when $n>n_c$. From \eqref{eq:smallh}, for any $t>0$, $\text{P}(h(\mathbf{X})<t)\leq C_b't^\beta$. Use Lemma \ref{lem:tail},
\begin{eqnarray}
\mathbb{E}[(g(\mathbf{X})-\eta(\mathbf{X}))^2\mathbf{1}(E=1)]&\leq& C_2 N^{-2\lambda}\mathbb{E}[h^{-2\lambda}(\mathbf{X})\mathbf{1}(h(\mathbf{X})>N^{-1})]\nonumber\\
&=&\left\{
\begin{array}{ccc}
\mathcal{O}(N^{-\beta})+\mathcal{O}(N^{-2\lambda}) &\text{if} & \beta\neq 2\lambda\\
\mathcal{O}(N^{-\beta}\ln N) & \text{if} & \beta=2\lambda.
\end{array}
\right.
\end{eqnarray}
\textbf{Case 2}. Similar to \eqref{eq:case2}, we have
$\mathbb{E}[(g(\mathbf{x})-\eta(\mathbf{x}))^2|n]\leq C_a+2M^2$,
hence
\begin{eqnarray}
\mathbb{E}[(g(\mathbf{X})-\eta(\mathbf{X}))^2\mathbf{1}(E=2)]
\leq  (C_a+2M^2)\left[ \text{P}(h(\mathbf{X})\leq N^{-1})+\text{P}\left(n\leq \frac{1}{2}\text{P}(B(\mathbf{X},A))\right)+\text{P}(n\leq n_c)\right].\nonumber
\end{eqnarray}
Now we bound these three probabilities. According to \eqref{eq:smallh}, the first term can be bounded by $\mathcal{O}(N^{-\beta})$. The second term was defined as $P_1$ in \eqref{eq:p1def}, and it has been proved in Appendix \ref{sec:adaptive} that $P_1=\mathcal{O}(N^{-\beta})$. It remains to bound $\text{P}(n\leq n_c)$. $n\leq n_c$ happens only if at least one of the following two events happen: (1) $n\leq \text{P}(B(\mathbf{X},A))/2$; (2)$\text{P}(B(\mathbf{X},A))/2\leq n_c$. The probability of these two events are both bounded by $\mathcal{O}(N^{-\beta})$, therefore $\text{P}(n\leq n_c)$ is bounded by $\mathcal{O}(N^{-\beta})$. As a result, $\mathbb{E}[(g(\mathbf{X})-\eta(\mathbf{X}))^2\mathbf{1}(E=2)]=\mathcal{O}(N^{-\beta})$.

Combine Case 1 and Case 2, we have
\begin{eqnarray}
\mathbb{E}[(g(\mathbf{X})-\eta(\mathbf{X}))^2]=\left\{
\begin{array}{ccc}
\mathcal{O}(N^{-\beta})+\mathcal{O}(N^{-2\lambda}) &\text{if} & \beta\neq 2\lambda\\
\mathcal{O}(N^{-\beta}\ln N) &\text{if} & \beta=2\lambda.
\end{array}
\right.
\label{eq:mse-b}
\end{eqnarray}
Now we calculate the optimal convergence rate. Recall that $2\lambda$ defined in \eqref{eq:lamdef},
\begin{eqnarray}
2\underset{q}{\max}\lambda =\underset{q}{\max}\left[\min\left\{q,\frac{4}{d}(1-q) \right\}\right] =\frac{4}{d+4},
\end{eqnarray}
with the maximum attained at
$q^*=4/(d+4)$.
Then the optimal convergence rate is:
\begin{eqnarray}
\mathbb{E}[(g(\mathbf{X})-\eta(\mathbf{X}))^2]=\left\{
\begin{array}{ccc}
\mathcal{O}(N^{-\min\left\{\beta,\frac{4}{d+4} \right\}}) &\text{if} & \beta\neq \frac{4}{d+4}\\
\mathcal{O}(N^{-\beta}\ln N) &\text{if} & \beta=\frac{4}{d+4}.
\end{array}
\right.
\end{eqnarray}
\section{Proof of Theorem \ref{thm:nounif}: No regression method is uniformly consistent without the new tail assumption}
In this section, we prove that no regressor can be uniformly consistent with Assumption 1 and 5 (e) but not 5 (b'). This indicates that Assumption 5 (b') is necessary.

Given the constants $C_a,\ldots,C_d,M,L$, let $\mathcal{S}$ be the set of pairs $(f,\eta)$ that satisfy the assumptions. For simplicity, we let $\beta=1$ and $L=1$. Other cases can be proved similarly. We first discuss one dimensional problems, and then generalize to arbitrary fixed dimension.

Define
\begin{eqnarray}
S^*=\{(f,\eta_v)|v\in \{-1,1\} \},
\end{eqnarray}
with
\begin{eqnarray}
f(x)=\left\{
\begin{array}{ccc}
1-\frac{1}{m} &\text{if} & -1<x<0\\
\frac{1}{m} &\text{if} & m<x<m+1
\end{array}
\right.
\label{eq:expf}
\end{eqnarray}
and
\begin{eqnarray}
\eta_1(x)=\left\{
\begin{array}{ccc}
0 & \text{if} & -1<x<0\\
m &\text{if} & m<x<m+1,
\end{array}
\right.
\eta_{-1}(x)=\left\{
\begin{array}{ccc}
0 & \text{if} & -1<x<0\\
-m &\text{if} & m<x<m+1.
\end{array}
\right.
\end{eqnarray}
In addition, define a variable 
\begin{eqnarray}
\hat{v}=\text{sign}\left(\int_m^{m+1} g(x)dx\right).
\end{eqnarray}
Recall that
\begin{eqnarray}
R-R^*=\mathbb{E}[(g(X)-\eta(X))^2]
=\mathbb{E}\left[\int (g(x)-\eta(x))^2 f(x)dx\right].
\end{eqnarray}
To give a lower bound of $R$, we have the following lemma.
\begin{lem}\label{lem:difsign}
	If $\hat{v}$ and $v$ have different sign, then $\int (g(x)-\eta(x))^2 f(x) dx\geq m$.
\end{lem}
\begin{proof}
	\begin{eqnarray}
	\int (g(x)-\eta(x))^2 f(x) dx&\geq & \int_m^{m+1} (g(x)-\eta(x))^2 \frac{1}{m}dx\nonumber\\
	&=&\int_m^{m+1} (g^2(x)-2g(x)\eta(x) +\eta^2(x))\frac{1}{m} dx\nonumber\\
	&=&m+\int_m^{m+1} g^2(x) \frac{1}{m} dx-2v\int_m^{m+1} g(x) dx.
	\end{eqnarray}
	Note that 
	$\int_m^{m+1} g^2(x)\frac{1}{m}dx\geq 0$,
	and 
	$-v\int g(x)dx\geq 0$,
	because $v$ and $\hat{v}$ have different sign. These facts imply that
	$\int (g(x)-\eta(x))^2 f(x) dx\geq m$.
\end{proof}
With Lemma \ref{lem:difsign}, we let $V$ follow distribution $\text{P}(V=1)=\text{P}(V=-1)=1/2$, and define $n$ as the number of training samples that fall in $[m,m+1]$, then
\begin{eqnarray}
 \underset{(f,\eta)\in \mathcal{S}^*}{\sup} (R-R^*)\geq  \mathbb{E}_V [R-R^*]\geq  \text{P}(V\neq \hat{V})m\overset{(a)}{\geq} \frac{1}{2}\text{P}(n=0) m= \frac{1}{2}\left(1-\frac{1}{m}\right)^N m,
\label{eq:nounif}
\end{eqnarray}
in which (a) is true because if there are no points falling in $[m,m+1]$, then for any detector $\hat{v}$, the conditional error probability given $n=0$ is $1/2$.

\eqref{eq:nounif} shows that if we pick a $\delta>0$, then for given $N$, we can find a $m$, such that $\underset{(f,\eta)\in \mathcal{S}}{\sup} R>\delta$. 

To generalize the above analysis to arbitrary fixed dimension, we only need to let $f(x_1)$ to replace $f(x)$ in \eqref{eq:expf}. Then let $X_2,\ldots,X_d$ follow uniform distribution in $[0,1]$ and $X_1,\ldots,X_d$ be independent. In this case, we can still get \eqref{eq:nounif}. Hence we claim that no regression method can be uniformly consistent without Assumption 5 (b').
\section{Proof of Theorem \ref{thm:stdunbounded}: Convergence rate of the standard kNN regression with unbounded $\eta$}\label{sec:std-u}
Note that from Assumption 5 (b'), we can show that
\begin{eqnarray}
\text{P}(f(\mathbf{X})<t)=\text{P}(e^{-bf(\mathbf{X})}>e^{-bt})\leq \underset{b}{\inf} e^{bt} \mathbb{E}[e^{-bf(\mathbf{X})}]=\underset{b}{\inf} e^{bt} t^{-\beta'}\leq et^{\beta'},
\label{eq:tail-u}
\end{eqnarray}
in which we let $b=1/t$ in the last step. Thus we recovered Assumption 1(b), except that $\beta$ is replaced by $\beta'$.

We still discuss two cases: Case 1: $f(\mathbf{x})>2k/(NC_dv_dD^d)$ and $\rho<D$; and Case 2: $f(\mathbf{x})\leq 2k/(NC_dv_dD^d)$ or $\rho\geq D$. Similarly, define a random variable $E$, for which $E=1$ for case 1 and $E=2$ for case 2.

\noindent\textbf{Case 1}. The analysis in Appendix \ref{sec:std-b} still holds here, because for any $\mathbf{x}'\in B(\mathbf{x},D)$, we have
\begin{eqnarray}
\eta(\mathbf{x}')-\eta(\mathbf{x})\leq L\norm{\mathbf{x}'-\mathbf{x}}\leq LD.
\end{eqnarray}
We can then use $LD$ to replace $M$ in Appendix \ref{sec:std-b}, and we can then get the same upper bound of the risk up to a constant factor. Hence
\begin{eqnarray}
\mathbb{E}[(g(\mathbf{X})-\eta(\mathbf{X}))^2 \mathbf{1}(E=1)]=\left\{
\begin{array}{ccc}
\mathcal{O}\left(\left(\frac{k}{N}\right)^\frac{4}{d}\right)+\mathcal{O}\left(\left(\frac{k}{N}\right)^{\beta'}\right)+\mathcal{O}\left(\frac{1}{k}\right) &\text{if} & {\beta'}\neq \frac{4}{d}\\
\mathcal{O}\left(\left(\frac{k}{N}\right)^{\beta'} \ln \frac{N}{k}\right)+\mathcal{O}\left(\frac{1}{k}\right) &\text{if} & {\beta'}=\frac{4}{d}.
\end{array}
\right.
\end{eqnarray}

\noindent\textbf{Case 2}. We conduct bias and variance decomposition again.
\begin{eqnarray}
\mathbb{E}[(g(\mathbf{x})-\eta(\mathbf{x}))^2 |\rho]=(\mathbb{E}[g(\mathbf{x})|\rho]-\eta(\mathbf{x}))^2+\Var[g(\mathbf{x})|\rho].
\label{eq:mse-u}
\end{eqnarray}
For the first term, i.e., the bias term, we have
\begin{eqnarray}
|\mathbb{E}[g(\mathbf{x})|\rho]-\eta(\mathbf{x})|= \left|\mathbb{E}\left[\left.\frac{1}{k}\sum_{i=1}^k \eta(\mathbf{X}^{(i)})\right\vert\rho \right]-\eta(\mathbf{x})\right|\leq  \frac{1}{k} \sum_{i=1}^k |\mathbb{E}[\eta(\mathbf{X}^{(i)})|\rho]-\eta(\mathbf{x})|\leq L\rho,
\label{eq:bias-u}
\end{eqnarray}
in which the last step uses $|\eta(\mathbf{X}^{(i)})-\eta(\mathbf{x})|\leq L\norm{\mathbf{X}^{(i)}-\mathbf{x}}\leq L\rho$.

Now we give a bound to the variance term:
\begin{eqnarray}
\Var[g(\mathbf{x})|\rho]&=&\Var\left[\left.\frac{1}{k} \sum_{i=1}^k Y^{(i)}\right\vert\rho\right]\nonumber\\
&=&\Var\left[\left.\frac{1}{k} \sum_{i=1}^k Y^{(i)} \right|\rho,\mathbf{X}^{(1)},\ldots,\mathbf{X}^{(k)}\right]+\Var\left[\left.\frac{1}{k}\sum_{i=1}^k \eta(\mathbf{X}^{(i)})\right\vert\rho\right]\nonumber\\
&\leq & \frac{C_a}{k}+\frac{1}{k}\sum_{i=1}^k \Var[\eta(\mathbf{X}^{(i)})|\rho].
\end{eqnarray}
In the last step, we use Assumption 1 (a) in the first term, and Cauchy inequality in the second term. Moreover,
\begin{eqnarray}
\Var[\eta(\mathbf{X}^{(i)})|\rho]\leq \mathbb{E}[(\eta(\mathbf{X}^{(i)})-\eta(\mathbf{X}))^2|\rho]\leq L^2 \mathbb{E}\left[\left.\norm{\mathbf{X}^{(i)}-\mathbf{x}}^2\right\vert\rho\right]\leq  L^2 \rho^2.
\label{eq:var-u0}
\end{eqnarray}
Hence
\begin{eqnarray}
\Var[g(\mathbf{x})|\rho]\leq C_a+L^2\rho^2.
\label{eq:var-u}
\end{eqnarray}
From \eqref{eq:mse-u}, \eqref{eq:bias-u} and \eqref{eq:var-u}, we get
\begin{eqnarray}
\mathbb{E}[(g(\mathbf{x})-\eta(\mathbf{x}))^2|\rho]\leq C_a+2L^2\rho^2.
\label{eq:errx}
\end{eqnarray}
Let $\rho_0$ be the $(k+1)$-th nearest neighbor distance of $\mathbf{x}=\mathbf{0}$. Then there are $k$ points in $B(\mathbf{0},\rho_0)$. Since $B(\mathbf{0},\rho_0)\subset B(\mathbf{x},\norm{\mathbf{x}+\rho_0})$, $B(\mathbf{x},\norm{\mathbf{x}+\rho_0})$ contains at least $k$ points. Hence
\begin{eqnarray}
\rho=\norm{\mathbf{X}^{(k+1)}-\mathbf{x}}\leq \norm{\mathbf{x}}+\rho_0.
\label{eq:rhoexp}
\end{eqnarray}
From Assumption 5, we know that there exists a constant $M_X$ such that $\mathbb{E}[\norm{\mathbf{X}}^2]<M_X<\infty$. Given this, we have the following lemma. 
\begin{lem}\label{lem:rho0}
	For some constant $C_1$ and sufficiently large $N$,	$\mathbb{E}[\rho_0^2]\leq C_1.$
\end{lem}
\begin{proof}
	Recall that $\rho_0$ is the $(k+1)$-th nearest neighbor distance of $\mathbf{x}=\mathbf{0}$. Since $\mathbb{E}[\norm{\mathbf{X}}^2]\leq M_X$, according to Chebyshev inequality,
	$\text{P}\left(\norm{\mathbf{X}}>r\right)\leq M_x/r^2.$ Therefore
	$\text{P}(B^c(\mathbf{0},r))\leq M_x/r^2$, in which $B^c(\mathbf{0},r)=\mathbb{R}^d\setminus B(\mathbf{0},r)$.
	Denote $n_r$ as the number of training samples in $B^c(\mathbf{0},r)$. For any $r>r_0>\sqrt{2M_X}$, we have $\text{P}(B^c(\mathbf{0},r))<1/2$. Hence for sufficiently large $N$,
	\begin{eqnarray}
	\text{P}(\rho_0>r)&=&\text{P}(n_r>N-k)\nonumber\\
	&\overset{(a)}{\leq} & \text{P}\left(n_r>\frac{1}{2}N\right)\nonumber\\
	&\overset{(b)}{\leq} & \exp[-N\text{P}(B^c(\mathbf{0},r))]\left(\frac{eN\text{P}(B^c(\mathbf{0},r))}{\frac{1}{2}N}\right)^\frac{N}{2}\nonumber\\
	&\leq &\left(2e\frac{M_X}{r^2}\right)^\frac{N}{2},
	\end{eqnarray}
	in which (a) holds because $k/N\rightarrow 0$, (b) comes from Chernoff inequality. Therefore
	\begin{eqnarray}
	\mathbb{E}[\rho_0^2]&=&\int_0^\infty \text{P}(\rho_0^2>t)dt\nonumber\\
	&=&\int_0^\infty \text{P}(\rho_0>\sqrt{t})dt\nonumber\\
	&=&\int_0^{2eM_X}\text{P}(\rho_0>\sqrt{t})dt+\int_{2eM_X}^\infty \text{P}(\rho_0>\sqrt{t})dt\nonumber\\
	&\leq &2eM_X +\int_{2eM_X}^\infty \left(\frac{2eM_X}{t}\right)^\frac{N}{2} dt\nonumber\\
	&=&2eM_X+\frac{2}{N-2}.
	\end{eqnarray}
	The proof of Lemma \ref{lem:rho0} is complete.
\end{proof}

From \eqref{eq:errx}, we get
\begin{eqnarray}
\mathbb{E}[(g(\mathbf{x})-\eta(\mathbf{x}))^2]&\leq &C_a+2L^2 \mathbb{E}[\rho^2|\mathbf{x}]\nonumber\\
&\overset{(a)}{\leq} & C_a+2L^2 (2\norm{\mathbf{x}}^2+2\mathbb{E}[\rho_0^2])\nonumber\\
&\leq& C_a+4L^2 (\norm{\mathbf{x}}^2+C_1)\nonumber\\
&\leq & 4L^2 \norm{\mathbf{x}}^2 +C_2,
\label{eq:loss1}
\end{eqnarray}
for some constant $C_2$. In (a), we used \eqref{eq:rhoexp} and Cauchy inequality.

Then
\begin{eqnarray}
\mathbb{E}[(g(\mathbf{X})-\eta(\mathbf{X}))^2\mathbf{1}(E=2)]
&\leq& \mathbb{E}\left[(g(\mathbf{X})-\eta(\mathbf{X}))^2 \mathbf{1}\left(f(\mathbf{X})\leq \frac{2k}{NC_dv_dD^d}\right) \right]\nonumber\\
&&+\mathbb{E}\left[(g(\mathbf{X})-\eta(\mathbf{X}))^2\mathbf{1}\left(f(\mathbf{X})>\frac{2k}{NC_dv_dD^d},\rho>D\right)\right].
\label{eq:case2-ub-exp}
\end{eqnarray}
The first term can be be bounded from \eqref{eq:loss1}:
\begin{eqnarray}
&&\mathbb{E}\left[(g(\mathbf{X})-\eta(\mathbf{X}))^2 \mathbf{1}\left(f(\mathbf{X})\leq \frac{2k}{NC_dv_dD^d}\right) \right]\nonumber\\
&\leq& C_2\text{P}\left(f(\mathbf{X})\leq \frac{2k}{NC_dv_dD^d}\right)+4L^2 \int  \mathbf{1}\left(f(\mathbf{x})\leq \frac{2k}{NC_dv_dD^d}\right) \norm{\mathbf{x}}^2 f(\mathbf{x})d\mathbf{x}=\mathcal{O}\left(\left(\frac{k}{N}\right)^{\beta'}\right),\nonumber\\
\label{eq:lowf}
\end{eqnarray}
in which the last step uses \eqref{eq:tail-u} to bound the first term, and Assumption \ref{ass:unbounded}(b') to bound the second term:
\begin{eqnarray}
\int \mathbf{1}\left(f(\mathbf{x})\leq \frac{2k}{NC_dv_dD^d}\right)\norm{\mathbf{x}}^2f(\mathbf{x})d\mathbf{x}\leq \int \exp\left[1-\frac{NC_dv_dD^d}{2k} f(\mathbf{x})\right]\norm{\mathbf{x}}^2 f(\mathbf{x})d\mathbf{x}=\mathcal{O}\left(\left(\frac{k}{N}\right)^{\beta'}\right).\nonumber
\end{eqnarray}
Now we bound the second term in \eqref{eq:case2-ub-exp}. In \eqref{eq:largerho}, we have proved that if $f(\mathbf{x})>2k/(NC_dv_dD^d)$, then $\text{P}(\rho>D|\mathbf{x})\leq \exp[-(1-\ln 2)k]$. Hence
\begin{eqnarray}
&&\mathbb{E}\left[(g(\mathbf{X})-\eta(\mathbf{X}))^2\mathbf{1}\left(f(\mathbf{X})>\frac{2k}{NC_dv_dD^d},\rho\geq D\right)\right]\nonumber\\
&\leq & \exp[-(1-\ln 2)k]\left[C_a+2L^2\int \mathbb{E}[\rho^2|\rho\geq D,\mathbf{x}] f(\mathbf{x})d\mathbf{x} \right],
\label{eq:largerho-u}
\end{eqnarray}
which decays faster than any polynomial. Combine \eqref{eq:lowf} and \eqref{eq:largerho-u}, \eqref{eq:lowf} dominates, i.e.
\begin{eqnarray}
\mathbb{E}[(g(\mathbf{X})-\eta(\mathbf{X}))^2\mathbf{1}(E=2)]=\mathcal{O}\left(\left(\frac{k}{N}\right)^{\beta'}\right).
\end{eqnarray}
Combining Case 1 and Case 2, we have
\begin{eqnarray}
R-R^*=\mathbb{E}[(g(\mathbf{X})-\eta(\mathbf{X}))^2]=\left\{
\begin{array}{ccc}
\mathcal{O}\left(\left(\frac{k}{N}\right)^{\min\{\beta',\frac{4}{d}\}} \right)+\mathcal{O}\left(\frac{1}{k}\right) &\text{if} & \beta'\neq \frac{4}{d}\\
\mathcal{O}\left(\left(\frac{k}{N}\right)^{\beta'} \ln \frac{N}{k}\right)+\mathcal{O}\left(\frac{1}{k}\right) &\text{if} & \beta'=\frac{4}{d}.
\end{array}
\right.
\end{eqnarray}
The fastest rate is attained if 
\begin{eqnarray}
k\sim \left\{
\begin{array}{ccc}
N^\frac{4}{d+4} & \text{if} & \beta'\geq \frac{4}{d}\\
N^{\frac{\beta'}{\beta'+1} }& \text{if} & \beta'<\frac{4}{d}
\end{array}.
\right.
\end{eqnarray}
The corresponding optimal convergence rate is
\begin{eqnarray}
R-R^*=\left\{
\begin{array}{ccc}
\mathcal{O}\left(N^{-\min\left\{\frac{4}{d+4},\frac{\beta'}{\beta'+1}\right\}}\right) &\text{if} & \beta'\neq\frac{4}{d}\\
\mathcal{O}\left(N^{-\frac{\beta'}{\beta'+1}}\ln N\right) &\text{if} & \beta'=\frac{4}{d}
\end{array}.
\right.
\end{eqnarray}
\section{Proof of Theorem \ref{thm:adaunbounded}: Convergence rate of the adaptive kNN regression with unbounded $\eta$}
In this section, we analyze the convergence rate of the adaptive kNN regression method when the regression function is not necessarily bounded. To obtain a bound on the convergence rate, we first consider three different events and then combine them. In particular, we consider:
\begin{itemize}
	\item Event 1: $h(\mathbf{x})>N^{-1}$, $n>N\text{P}(B(\mathbf{x},A))/2$ and $n>n_c$.
	\item Event 2: $n>n_c$, but $h(\mathbf{x})\leq N^{-1}$ or $n\leq N\text{P}(B(\mathbf{x},A))/2$.
	\item Event 3: $n\leq n_c$.
\end{itemize}
Similar to other sections, we define a random variable $E$, which will be equal to $i$ if $i$th event occurs. 

For the first two events, \eqref{eq:mse-b} in Appendix \ref{sec:ada-b} still holds, except that since Assumption \ref{ass:basic}(b) is replaced by \eqref{eq:tail-u}, $\beta$ is now replaced by $\beta'$:
\begin{eqnarray}
\mathbb{E}[(g(\mathbf{X})-\eta(\mathbf{X}))^2\mathbf{1}(E=1 \text{ or } E=2 )]=\left\{
\begin{array}{ccc}
\mathcal{O}(N^{-\beta'})+\mathcal{O}(N^{-2\lambda}) &\text{if} & \beta'\neq 2\lambda\\
\mathcal{O}(N^{-\beta'}\ln N) &\text{if} & \beta'=2\lambda,
\end{array}
\right.
\end{eqnarray}
in which
$\lambda=\min\left\{q/2, (2/d)(1-q) \right\}$.
Now we analyze Event 3. Note that when $n<n_c$, $k$ may be larger than $n$, thus some nearest neighbors may fall outside $B(\mathbf{x},D)$. Note that \eqref{eq:loss1} still holds here. Therefore
\begin{eqnarray}
\mathbb{E}[(g(\mathbf{X})-\eta(\mathbf{X}))^2\mathbf{1}(E=3)]
=\int 4L^2 \norm{\mathbf{x}}^2 \text{P}(n\leq n_c|\mathbf{x})f(\mathbf{x})d\mathbf{x}+C_2\int \text{P}(n\leq n_c|\mathbf{x})f(\mathbf{x})d\mathbf{x}.
\end{eqnarray}
Using similar argument as Appendix \ref{sec:ada-b}, we can show that the second term can be bounded by $\mathcal{O}(N^{-\beta'})$. Now we bound the first term. 
 If $n_c\leq N\text{P}(B(\mathbf{x},A))/2$, from \eqref{eq:smalln},
\begin{eqnarray}
\text{P}(n\leq n_c|\mathbf{x})\leq \exp\left[-\frac{1}{2}(1-\ln 2)NC_dv_dA^df(\mathbf{x})\right].
\end{eqnarray}
If $n_c>N\text{P}(B(\mathbf{x},A))/2$, then we just bound $\text{P}(n\leq n_c|\mathbf{x})$ with $1$. Therefore
\begin{eqnarray}
\int \norm{\mathbf{x}}^2\text{P}(n\leq n_c|\mathbf{x})f(\mathbf{x})d\mathbf{x}&\leq& \text{P}\left(\text{P}(B(\mathbf{X},A))\geq\frac{2n_c}{N}\right)\int \norm{\mathbf{x}}^2 \exp\left[-\frac{1}{2}(1-\ln 2)NC_dv_dA^d f(\mathbf{x})\right]d\mathbf{x}\nonumber\\
&&+\text{P}\left(\text{P}(B(\mathbf{X},A))<\frac{2n_c}{N}\right)\int \norm{\mathbf{x}}^2 f(\mathbf{x})d\mathbf{x}.
\label{eq:I31}
\end{eqnarray}
Both of two terms in \eqref{eq:I31} can be bounded by $\mathcal{O}(N^{-\beta'
})$. Therefore
\begin{eqnarray}
\mathbb{E}[(g(\mathbf{X})-\eta(\mathbf{X}))^2 \mathbf{1}(E=3)]=\mathcal{O}(N^{-\beta'}).
\end{eqnarray}
The overall convergence rate can be bounded by: 
\begin{eqnarray}
\mathbb{E}[(g(\mathbf{X})-\eta(\mathbf{X}))^2]=\left\{
\begin{array}{ccc}
\mathcal{O}(N^{-\beta'})+\mathcal{O}(N^{-\lambda}) &\text{if} & \beta'=\lambda\\
\mathcal{O}(N^{-\beta'}\ln N) &\text{if} & \beta'\neq \lambda.
\end{array}
\right.
\end{eqnarray}
The optimal convergence rate is attained when $q=4/(4+d)$. In this case,
\begin{eqnarray}
\mathbb{E}[(g(\mathbf{X})-\eta(\mathbf{X}))^2]=\left\{
\begin{array}{ccc}
\mathcal{O}\left(N^{-\min\left\{\beta',\frac{4}{d+4} \right\}}\right) &\text{if} & \beta'\neq \frac{4}{d+4}\\
\mathcal{O}(N^{-\beta'}\ln N) &\text{if} & \beta'=\frac{4}{d+4}.
\end{array}
\right.
\end{eqnarray}
\section{Technical Lemmas and Proofs}\label{sec:lemmas}
In this appendix, we state and prove some technical lemmas that are used in the proof of theorems. All of the following lemmas hold for both classification and regression problems.
\begin{lem}\label{lem:tail}
	(1) Under Assumption 1 (b), which says that $\text{P}(f(\mathbf{X})<t)\leq C_bt^\beta$, for any $p>0$ and $b>0$,
	\begin{eqnarray}
	\mathbb{E}[e^{-bf^p(\mathbf{X})}]\leq \frac{C_b\Gamma\left(1+\frac{\beta}{p}\right)}{b^{\frac{\beta}{p}}},
	\label{eq:tb2}
	\end{eqnarray}
	in which $\Gamma$ is the Gamma function defined in~\eqref{eq:Gamma}.
	
	(2)	For two sequences $r_N$, $s_N$ such that $r_N\rightarrow 0$ and $s_N\rightarrow 0$ as $N\rightarrow \infty$, and $r_N>s_N$ for sufficiently large $N$, then for any $p>0$, under Assumption \ref{ass:basic} (b)
	\begin{eqnarray}
	\mathbb{E}[f^{-p}(\mathbf{X})\mathbf{1}(s_N<f(\mathbf{X})<r_N)]=\left\{
	\begin{array}{ccc}
	\mathcal{O}\left(r_N^{\beta-p}\right) &\text{if} & \beta>p;\\
	\mathcal{O}\left(\ln \frac{r_N}{s_N}\right) &\text{if} & \beta=p;\\
	\mathcal{O}\left(s_N^{\beta-p}\right) &\text{if} & \beta<p,
	\end{array}
	\right.
	\label{eq:tb1}
	\end{eqnarray} 
	
	(3) For $\forall p>0$, and any sequence $\{s_N\}$ such that $s_N\rightarrow\infty$ as $N\rightarrow\infty$, then with Assumption \ref{ass:basic} (b), 
	\begin{eqnarray}
	\mathbb{E}[f^{-p}(\mathbf{X})\mathbf{1}(f(\mathbf{X})>s_N))]=\left\{
	\begin{array}{ccc}
	\mathcal{O}(1) &\text{if} & \beta>p\\
	\mathcal{O}\left(\ln \frac{1}{s_N} \right) &\text{if} & \beta=p\\
	\mathcal{O}\left(s_N^{\beta-p}\right) &\text{if} & \beta<p.
	\end{array}
	\right.
	\label{eq:tb3}
	\end{eqnarray}
	
	(4) With Assumption \ref{ass:additional}, the upper bounds of \eqref{eq:tb2}, \eqref{eq:tb1} and \eqref{eq:tb3} also holds for $h(\mathbf{X})$.
\end{lem}
\begin{proof}
	(1) Proof of \eqref{eq:tb2}:
	\begin{eqnarray}
	\mathbb{E}[e^{-bf^p(\mathbf{X})}]=\int_0^1 \text{P}\left(e^{-bf^p(\mathbf{X})}>t\right) dt=\int_0^1 \text{P}\left(f(\mathbf{X})<\left(\frac{\ln\frac{1}{t}}{b}\right)^\frac{1}{p}\right)dt &\leq& C_b\int_0^1 \left(\frac{\ln \frac{1}{t}}{b}\right)^\frac{\beta}{p} dt\nonumber\\
	&=&\frac{C_b\Gamma \left(1+\frac{\beta}{p}\right)}{b^\frac{\beta}{p}}.
	\end{eqnarray}
	
	(2) Proof of \eqref{eq:tb1}:
		\begin{eqnarray}
	\mathbb{E}[f^{-p}(\mathbf{X})\mathbf{1}(s_N<f(\mathbf{X})<r_N)]&=&\int_0^\infty \text{P}\left(f^{-p}(\mathbf{X})\mathbf{1}(s_N<f(\mathbf{X})<r_N)>t\right)dt\nonumber\\
	&=&\int_{r_N^{-p}}^{s_N^{-p}} \text{P}(f^{-p}(\mathbf{X})>t)dt\nonumber\\
	&=&\int_{r_N^{-p}}^{s_N^{-p}} \text{P}(f(\mathbf{X})<t^{-\frac{1}{p}})dt\nonumber\\
	&\leq &\int_{r_N^{-p}}^{s_N^{-p}}C_b t^{-\frac{\beta}{p}}dt,
	\end{eqnarray}
	in which the last step comes from Assumption 1(b). We then obtain \eqref{eq:tb1} by simple integral for cases with $\beta>p$, $\beta=p$ and $\beta<p$ separately.
	
	(3) \eqref{eq:tb3} can be proved in similar way as the proof of \eqref{eq:tb1}. We omit the proof for simplicity.
	
	(4) \eqref{eq:smallh} has the same form as Assumption \ref{ass:basic}(b). Thus the above derivation also holds for $h(\mathbf{X})$.
\end{proof}
\begin{lem}\label{lem:conc}
	(1) The expectation of $\hat{\eta}(\mathbf{x})$ is:
	\begin{eqnarray}
	\mathbb{E}[\hat{\eta}(\mathbf{x})|\rho]=\eta(B(\mathbf{x},\rho));
	\label{eq:etaexp}
	\end{eqnarray}
	(2) $\hat{\eta}(\mathbf{x})$ satisfies the following concentration inequality:
	\begin{eqnarray}
	\text{P}(|\hat{\eta}(\mathbf{x})-\mathbb{E}[\hat{\eta}(\mathbf{x})]|>t)\leq 2e^{-\frac{1}{2}kt^2}.
	\label{eq:etaconc}
	\end{eqnarray}
\end{lem}
\begin{proof}
	Our proof of Lemma \ref{lem:conc} follows the proof of Lemma 9 in \cite{chaudhuri2014rates}. We pick $(\mathbf{X}_i,Y_i)$ in the following way: firstly, pick a point $X_1$ according to the marginal distribution of $(k+1)$-th nearest neighbor of $\mathbf{x}$. Denote $\rho$ as the distance. Then pick $k$ points from conditional distribution $f(\cdot|\mathbf{X}\in B(\mathbf{x},\rho))$. Then pick $(N-k-1)$ points from the conditional distribution $f(\cdot|\mathbf{X} \notin B(\mathbf{x},\rho))$. The next step is to randomly permute these $N$ points. The joint distribution of these $N$ points obtained in this way are i.i.d, with pdf $f(\mathbf{x})$. Finally, assign all of the $N$ points with label $1$ or $-1$, with probability $\text{P}(Y_i=1|\mathbf{X}_i)=\frac{1}{2}(1+|\eta(\mathbf{X}_i)|)$.
	
	Note that the $k$ points picked according to distribution $f(\cdot|\mathbf{X}\in B(\mathbf{x},\rho))$ are i.i.d, and the expectation of the target is $\eta(B(\mathbf{x},\rho))$. This yields \eqref{eq:etaexp}. Besides, according to Hoeffding's inequality, we get \eqref{eq:etaconc}.
\end{proof}
\begin{lem}\label{lem:order}
	\begin{eqnarray}
	\mathbb{E}[\text{P}^\frac{4}{d}(B(\mathbf{x},\rho))]\leq \frac{\left(k+\frac{4}{d}+1\right)^\frac{4}{d}}{N^\frac{4}{d}}.
	\end{eqnarray}
\end{lem}
\begin{proof}
	Let random variable $U=\text{P}(B(\mathbf{x},\rho))$, using results from the order statistics \cite{david1970order}, we know that $U$ follows Beta distribution:
	$f(u)=(1/\text{Beta}(k+1,N-k)) u^k (1-u)^{N-k-1}$.
	Hence
	\begin{eqnarray}
	\mathbb{E}[\text{P}^{\frac{4}{d}}(B(\mathbf{x},\rho))] =\mathbb{E}[U^\frac{4}{d}]=\frac{\Gamma(N+1)}{\Gamma(N+\frac{4}{d}+1)} \frac{\Gamma\left(k+\frac{4}{d}+1\right)}{\Gamma(k+1)}\leq \frac{\left(k+\frac{4}{d}+1\right)^\frac{4}{d}}{N^\frac{4}{d}}.
	\label{eq:beta}
	\end{eqnarray}
\end{proof}
\begin{lem}\label{lem:mseadaptive}
	For adaptive kNN classification or regression, if $k\leq n$, then
	\begin{eqnarray}
	\mathbb{E}[(\hat{\eta}(\mathbf{x})-\eta(\mathbf{x}))^2|n]\leq C_M h^{-2\lambda}(\mathbf{x})N^{-2\lambda},
	\label{eq:mseadaptive}
	\end{eqnarray}
	in which $\lambda=\min\left\{2(1-q)/d,q/2 \right\}$, $C_M$ is a constant.
\end{lem}
\begin{proof}
\begin{eqnarray}
\mathbb{E}[(\hat{\eta}(\mathbf{x})-\eta(\mathbf{x}))^2|n,\rho]=(\mathbb{E}[\hat{\eta}(\mathbf{x})|n,\rho]-\eta(\mathbf{x}))^2+\Var[\hat{\eta}(\mathbf{x})|n,\rho].
\end{eqnarray}
Given $n$, $k$ is fixed, hence the second term has the same bound as \eqref{eq:varg}:
$\Var[\hat{\eta}(\mathbf{x})|n,\rho]\leq (C_a+M^2)/k$.
Besides, we have
$\mathbb{E}[\hat{\eta}(\mathbf{x})|n,\rho]=\eta(B(\mathbf{x},\rho))$.
According to Assumption 1 (c), 
$|\eta(B(\mathbf{x},\rho))-\eta(\mathbf{x})|\leq C_c\rho^2$.
Therefore
\begin{eqnarray}
\mathbb{E}[(\hat{\eta}(\mathbf{x})-\eta(\mathbf{x}))^2|n]\leq C_c^2 \mathbb{E}[\rho^{4}|n]+\frac{C_a+M^2}{k}.
\label{eq:phi}
\end{eqnarray}
We now bound $\mathbb{E}[\rho^{4}|n]$. $k\leq n$ implies $\rho<A$. Moreover, given $n$, the $n$ points in $B(\mathbf{x},A)$ are conditional i.i.d with pdf $f(\mathbf{x})/\text{P}(B(\mathbf{x},A))$. Use Lemma \ref{lem:order}, we have
\begin{eqnarray}
\mathbb{E}\left[\left.\frac{\text{P}^{\frac{4}{d}}(B(\mathbf{x},\rho))}{\text{P}^{\frac{4}{d}}(B(\mathbf{x},A))}\right|n\right]\leq \frac{\left(k+\frac{4}{d}+1\right)^\frac{4}{d}}{n^\frac{4}{d}}.
\label{eq:order2}
\end{eqnarray}
We have the following inequality that holds in general: for $a,b,c>0$,
\begin{eqnarray}
(a+b)^c\leq \left\{
\begin{array}{ccc}
2^{c-1}(a^c+b^c) &\text{if} & c>1\\
a^c+b^c & \text{if} & c\leq 1.
\end{array}
\right.
\end{eqnarray}
Hence
\begin{eqnarray}
\left(\frac{k+\frac{4}{d}+1}{n}\right)^\frac{4}{d}=\left(\frac{\left\lfloor Kn^q\right\rfloor+\frac{4}{d}+2}{n}\right)^\frac{4}{d}\leq  2^{\max\{\frac{4}{d}-1,0\} }\left(K^\frac{4}{d}n^{-\frac{4}{d}(1-q)}+\left(\frac{4}{d}+2\right)^\frac{4}{d} n^{-\frac{4}{d}}\right).\label{eq:abcbound}
\end{eqnarray}
Furthermore, according to Assumption 1 (d), $\text{P}(B(\mathbf{x},\rho))\geq C_dv_d\rho^df(\mathbf{x})$. Using this and~\eqref{eq:abcbound} in \eqref{eq:order2}, we obtain
\begin{eqnarray}
(C_dv_df(\mathbf{x}))^\frac{4}{d}\mathbb{E}[\rho^{4}|n]\leq 2^{\max\{\frac{4}{d}-1,0 \} }\left(K^\frac{4}{d}n^{-\frac{4}{d}(1-q)}+\left(\frac{4}{d}+2\right)^\frac{4}{d} n^{-\frac{4}{d}}\right)\text{P}^\frac{4}{d}(B(\mathbf{x},A)).
\label{eq:rhob1}
\end{eqnarray}
Under case 1, $n\geq N\text{P}(B(\mathbf{x},A))/2$. Plug it into \eqref{eq:rhob1}, and recall \eqref{eq:relation}, after some simplification, we eventually get:
\begin{eqnarray}
\mathbb{E}[\rho^{4}|n]\leq M_A h^{-\frac{4}{d}(1-q)}(\mathbf{x})N^{-\frac{4}{d}(1-q)}+M_Bh^{-\frac{4}{d}}(\mathbf{x})N^{-\frac{4}{d}},
\label{eq:rho4}
\end{eqnarray}
for some constants $M_A$ and $M_B$.

 Besides, note that the condition of case 1 says that $n>N\text{P}(B(\mathbf{x},A))/2$, therefore using \eqref{eq:relation},
\begin{eqnarray}
k=\left\lfloor Kn^q\right\rfloor +1\geq  K\left(\frac{1}{2}N\text{P}(B(\mathbf{x},A))\right)^q\geq 2^{-q} K^q (C_dv_dA^d)^\frac{q}{1-q} h^q(\mathbf{x})N^q.
\label{eq:klb}
\end{eqnarray}
\eqref{eq:phi}, \eqref{eq:rho4} and \eqref{eq:klb} yields
\begin{eqnarray}
\mathbb{E}[(\hat{\eta}(\mathbf{x})-\eta(\mathbf{x}))^2|n]\leq C_c^2 \left(M_A h^{-\frac{4}{d}(1-q)}(\mathbf{x})N^{-\frac{4}{d}(1-q)}+M_B h^{-\frac{4}{d}}(\mathbf{x})N^{-\frac{4}{d}}\right)+C_1 N^{-q} h^{-q}(\mathbf{x}),
\end{eqnarray}
for some constant $C_1$.

Moreover, when case 1 happens, $h(\mathbf{x})>N^{-1}$ always holds. Recall that $\lambda$ is defined as $\lambda:=\min\left\{2(1-q)/d,q/2 \right\}$,  for some constant $C_M$, which satisfies $C_M\leq C_c^2(M_A+M_B)+C_1$,
\begin{eqnarray}
\mathbb{E}[(\hat{\eta}(\mathbf{x})-\eta(\mathbf{x}))^2|n]\leq C_M h^{-2\lambda}(\mathbf{x})N^{-2\lambda}.
\end{eqnarray}
\end{proof}
\small \bibliography{macros,knn}
\bibliographystyle{ieeetran}
\end{document}